\documentclass[12pt]{article}
\usepackage{amssymb,amsmath,amsthm, amsfonts}
\usepackage{graphicx}
\usepackage{epsfig}
\usepackage{tikz}
\usepackage{mathrsfs}

\def\disp{\displaystyle}

\def\dref#1{(\ref{#1})}

\theoremstyle{plain}
\newtheorem{theorem}{Theorem}[section]
\newtheorem{lemma}{Lemma}[section]
\newtheorem{proposition}{Proposition}[section]

\theoremstyle{definition}
\newtheorem{definition}{Definition}[section]

\numberwithin{equation}{section}

\begin{document}

\title{\bf Global existence,
smooth and stabilization  in a   three-dimensional
Keller-Segel-Navier-Stokes system with  rotational flux
}

\author{
Jiashan Zheng\thanks{Corresponding author.   E-mail address:
 zhengjiashan2008@163.com (J.Zheng)}
 \\
    School of Mathematics and Statistics Science,\\
     Ludong University, Yantai 264025,  P.R.China \\
}
\date{}


\maketitle \vspace{0.3cm}
\noindent
\begin{abstract}
We consider the spatially $3$-D version of the following Keller-Segel-Navier-Stokes system with  rotational flux
$$
 \left\{
 \begin{array}{l}
   n_t+u\cdot\nabla n=\Delta n-\nabla\cdot(nS(x,n,c)\nabla c),\quad
x\in \Omega, t>0,\\
    c_t+u\cdot\nabla c=\Delta c-c+n,\quad
x\in \Omega, t>0,\\
u_t+\kappa(u \cdot \nabla)u+\nabla P=\Delta u+n\nabla \phi,\quad
x\in \Omega, t>0,\\
\nabla\cdot u=0,\quad
x\in \Omega, t>0\\
 \end{array}\right.\eqno(*)
 $$
 under no-flux boundary conditions in a bounded domain $\Omega\subseteq \mathbb{R}^{3}$ with smooth boundary, where $\phi\in W^{2,\infty} (\Omega)$ and
 $\kappa\in \mathbb{R}$ represent the prescribed gravitational potential and    the strength of nonlinear fluid
convection, respectively.   Here
   the matrix-valued function $S(x,n,c)\in C^2(\bar{\Omega}\times[0,\infty)^2 ;\mathbb{R}^{3\times 3})$ denotes the rotational effect which
satisfies $|S(x,n,c)|\leq C_S(1 + n)^{-\alpha}$ with some $C_S > 0$ and $\alpha\geq 0$.
Compared with the signal
consumption case as in  chemotaxis-(Navier-)Stokes system, the quantity $c$ of system $(*)$ is no longer a priori bounded
by its initial norm in $L^\infty$, which means that we have less regularity information on $c$. Moreover, the tensor-valued sensitivity functions result in new mathematical difficulties, mainly linked to the fact that a chemotaxis system with such rotational fluxes
thereby loses an energy-like structure.
 In this paper,
by seeking some new
functionals and using the bootstrap arguments on  system $(*)$, we establish the existence  of
global  weak solutions to  system $(*)$ for arbitrarily large initial data under the assumption $\alpha\geq1$. Moreover, under an explicit condition on the size of $C_S$ relative to $C_N$, we can  secondly prove that in fact any such {\bf weak} solution $(n,c,u)$ becomes
smooth ultimately, and that it approaches the unique spatially homogeneous steady
state $(\bar{n}_0,\bar{n}_0,0)$, where $\bar{n}_0=\frac{1}{|\Omega|}\int_{\Omega}n_0$ and $C_N$ is the best  Poincar\'{e} constant. To the best of our knowledge, there are the first results
on  asymptotic behavior of the system.

\end{abstract}

\vspace{0.3cm}


\newpage
\section{Introduction}
In this paper, we investigate the global existence and stabilization of the  solutions to the following
 Keller-Segel-Navier-Stokes system with  rotational flux
\begin{equation}
 \left\{\begin{array}{ll}
   n_t+u\cdot\nabla n=\Delta n-\nabla\cdot(nS(x,n,c)\nabla c),\quad
x\in \Omega, t>0,\\
    c_t+u\cdot\nabla c=\Delta c-c+n,\quad
x\in \Omega, t>0,\\
u_t+\kappa(u \cdot \nabla)u+\nabla P=\Delta u+n\nabla \phi,\quad
x\in \Omega, t>0,\\
\nabla\cdot u=0,\quad
x\in \Omega, t>0\\
 \end{array}\right.\label{33dfff4451.1fghyuisda}
\end{equation}
in a physical smoothly bounded domain $\Omega\subseteq R^3$, under zero-flux boundary conditions
\begin{equation}\disp{(\nabla n-nS(x, n, c)\nabla c)\cdot\nu=\nabla c\cdot\nu=0,u=0,}\quad
x\in \partial\Omega, t>0\label{33dfff44sdfff51.1fghfffyuisda}
\end{equation}
and with prescribed initial data
\begin{equation}\disp{n(x,0)=n_0(x),c(x,0)=c_0(x),u(x,0)=u_0(x),}\quad
x\in \Omega,\label{33dfff44sdfff51.1fgsdddfhyuisda}
\end{equation}
where $n,c,u,P,\phi$, respectively  denote the bacterial density, chemical concentration, fluid velocity, the associated pressure and the prescribed gravitational potential.
 This system was initially proposed by Wang-Winkler-Xiang (\cite{Wddffang11215,Wangss21215}) to describe
chemotaxis-fluid interaction in cases when the evolution of the chemoattractant is essentially
dominated by production through cells.

Before going into our mathematical analysis, we recall some important progresses on system \dref{33dfff4451.1fghyuisda} and its variants.
When $u\equiv 0,$ the PDE system \dref{33dfff4451.1fghyuisda} is reduced to the Keller-Segel system
\begin{equation}
 \left\{\begin{array}{ll}
 n_t=\Delta n-\nabla\cdot(nS(x,n,c)\nabla c),
 \quad
x\in \Omega,~ t>0,\\
 \disp{ c_t=\Delta c- c +n,}\quad
x\in \Omega, ~t>0,\\
 \end{array}\right.\label{ssd7223dssddsffgsdddddff4ssd4101.2x1sddd6677}
\end{equation}
where the tensor-valued function (or the scalar function)
 $S$
 measures the chemotactic sensitivity and satisfies
 \begin{equation}\label{x1.73142vghf48rtgyhu}
S\in C^2(\bar{\Omega}\times[0,\infty)^2;\mathbb{R}^{N\times N})
 \end{equation}
 and
 \begin{equation}\label{x1.73142vghf48gg}|S(x, n, c)|\leq C_S(1 + n)^{-\alpha} ~~~~\mbox{for all}~~ (x, n, c)\in\Omega\times [0,\infty)^2
 \end{equation}
with some $C_S > 0$ and $\alpha\geq 0$. Here $N$ denotes   the space dimension.
The global existence and finite time
blow-up 
of system \dref{ssd7223dssddsffgsdddddff4ssd4101.2x1sddd6677} have been widely investigated in the past decades. In fact, in the case $S(x,n,c) = 1$, solutions of \dref{ssd7223dssddsffgsdddddff4ssd4101.2x1sddd6677} may blow up in finite time when
$N\geq 2$ (\cite{Herreroddd,Winkler793}). On the other hand, such explosion phenomena can be ruled out when $S(x,n,c):= S(n)$ is related
to the prototypical assumption of volume-filling effect (\cite{Horstmann791,Winkler79}).
More precisely, 
$S(x, n, c):= (1+n)^{-\alpha}$ is a
scalar function, when
 %
$\alpha > 1-\frac{2}{N}$,
 all solutions are global and uniformly
bounded
(\cite{Horstmann791});  while  if
$\Omega\subseteq \mathbb{R}^N (N\geq 2)$ is a ball and $\alpha< 1-\frac{2}{N}$, the solution may blow up
under some technical assumptions (\cite{Horstmann791,Winkler79}).
Quite a number
of results
have been investigated, including the system with the logistic terms (see \cite{Zhengddfggghjjkk1,Wang76,Winkler21215,Tello710},
for instance),  the system with nonlinear  diffusion (see \cite{Cie72,Kowalczyk7101,Zhengssdefr23,Zheng00}) and so on. We refer to \cite{Bellomo1216,Hillen} for the further
reading.

When $n = c=0$, the PDE system
\dref{33dfff4451.1fghyuisda} becomes the Navier-Stokes system
\begin{equation}
 \left\{\begin{array}{ll}
u_t+\kappa(u \cdot \nabla)u+\nabla P=\Delta u,\quad
x\in \Omega, t>0,\\
\nabla\cdot u=0,\quad
x\in \Omega, t>0.\\
 \end{array}\right.\label{33dfff4451.1fghyufgggisda}
\end{equation}
As we all know that the corresponding Navier-Stokes system \dref{33dfff4451.1fghyufgggisda} does not admit a satisfactory solution
theory up to now. In fact,  global weak solvability was studied in \cite{LerayLerayer792,Sohr},  some local
existence of smooth solutions are proved by \cite{FujitaFujitahanggssssssdefr23}  and \cite{Wiegnerdd79}.

In various situations, however, the migration of bacteria may be more complex
because it can be effected by the changes in their living environment (see \cite{Dombrowskio12186} and \cite{Lorz1215}). Considering that in nature some bacteria, like Bacillus subtilis and
Escherichia coli, often live in a viscous fluid, Tuval and his cooperator (\cite{Tuval1215}) described the above biological phenomena and
proposed the mathematical model consisting of oxygen diffusion and {\bf consumption},
chemotaxis, and fluid dynamics. In fact, Tuval et al.  (\cite{Tuval1215})
proposed the following chemotaxis-(Navier-)Stokes system in the context of signal {\bf consumption} by cells
\begin{equation}
 \left\{\begin{array}{ll}
   n_t+u\cdot\nabla n=\Delta n-\nabla\cdot( nS(n)\nabla c),\quad
x\in \Omega, t>0,\\
    c_t+u\cdot\nabla c=\Delta c-nf(c),\quad
x\in \Omega, t>0,\\
u_t+\kappa (u\cdot\nabla)u+\nabla P=\Delta u+n\nabla \phi,\quad
x\in \Omega, t>0,\\
\nabla\cdot u=0,\quad
x\in \Omega, t>0,\\
 \end{array}\right.\label{1.ss1hdffffhjjddssssghhhhhddggddddffftyy}
\end{equation}
where  $f(c)$ and $\kappa$ represent   the consumption rate
of the bacteria and the strength of nonlinear fluid
convection, respectively. In particular, when the fluid flow is slow, we can use the Stokes equation instead of
the Navier-Stokes one, i.e., $\kappa = 0$ (see \cite{Lorz1215,Francesco12186}). The chemotaxis fluid system has been studied in the last few years.
In fact, by making use of the energy-type functionals,
system \dref{1.ss1hdffffhjjddssssghhhhhddggddddffftyy} and its variants have attracted extensive attention (see, e.g., Chae et al. \cite{Chaexdd12176},
Duan et al. \cite{Duan12186}, Liu and Lorz \cite{Liu1215,Lorz1215},
Tao and Winkler \cite{Tao41215,Winkler31215,Winkler61215,Winkler51215},
Zhang and Zheng \cite{Zhang12176}, and references therein). For example, the well-posedness and stabilization in a two-dimensional bounded
domain have been investigated in \cite{Winkler31215} and \cite{Winkler61215}, respectively.  In \cite{Winkler61215},
Winkler showed that in bounded {\bf convex} domains $\Omega\subseteq R^2$, the global classical solutions
obtained in \cite{Winkler31215} stabilize to the spatially uniform equilibrium $(\frac{1}{|\Omega|}\int_\Omega n_0 ,0,0)$
as $t\rightarrow\infty$. For $\Omega\subseteq R^3$,  the global existence of
weak solution for system \dref{1.ss1hdffffhjjddssssghhhhhddggddddffftyy}
has been proved in \cite{Winkler51215}.
Recently,  the eventual smoothness of weak
solutions in the three-dimensional case  has also been investigated by Winkler in \cite{Winklerpejoevsssdd793}.
For more literature related to this model, we can refer to   \cite{Tao61215,Tao71215,Winkler11215,Winklerssdff51215,LiLiLigZhanggssssssdefr23} and the references therein.

Compared with the chemotaxis system and the chemotaxis-(Navier-)Stokes system, the  Keller-Segel-Navier-Stokes system
 \dref{33dfff4451.1fghyuisda} is much less understood. In fact, compared with the {\bf signal
consumption case} as in  chemotaxis-(Navier-)Stokes system \dref{1.ss1hddffhjjdllddddssssddggddddffftyy}, the quantity $c$ of system \dref{33dfff4451.1fghyuisda} is no longer a priori bounded
by its initial norm in $L^\infty$, which means that we have less regularity information on $c$.
  Now, let us
briefly recapitulate the recent developments for system
 \dref{33dfff4451.1fghyuisda}. To this end, 
in 2018,
Wang-Winkler-Xiang (\cite{Wddffang11215})
showed  that when the tensor-valued sensitivity
$S$ satisfying \dref{x1.73142vghf48rtgyhu} and \dref{x1.73142vghf48gg} with $\alpha > 0$ and $\Omega\subseteq \mathbb{R}^2$ is a bounded {\bf convex} domain with smooth boundary,
 system \dref{33dfff4451.1fghyuisda}
admits  a global-in-time classical and bounded solution. By deriving a new type of entropy-energy estimate, Zheng (\cite{Zhenddsdddddgssddsddfff00})
extends the result of \cite{Wddffang11215} to general bounded domains.
In comparison to the result
for the corresponding fluid-free system, it is easy to see that the restriction on $\alpha$ here is
optimal (see \cite{Zhenddsdddddgssddsddfff00,Wddffang11215}).
 In the
case of $N = 3$  and tensor-valued $S(x,n,c)$ satisfying \dref{x1.73142vghf48rtgyhu} and \dref{x1.73142vghf48gg}, global
weak solutions were shown to exist for $\alpha\geq\frac{3}{7}$
(\cite{LiuZhLiuLiuandddgddff4556}).
 Recently, Zheng and Ke  (\cite{Kegssddsddfff00}) presented the existence of global and weak solutions for   Keller-Segel-Navier-Stokes system \dref{33dfff4451.1fghyuisda} under the assumption that $S$ satisfies
\dref{x1.73142vghf48rtgyhu} and  \dref{x1.73142vghf48gg} with $\alpha > \frac{1}{3}$,
which, in light of the known results for the fluid-free system, is an optimal restriction on $\alpha$.
 For more works
about the chemotaxis models with tensor sensitivity (and nonlinear  diffusion), we refer readers to see \cite{Zhengsddfffsdddssddddkkllssssssssdefr23,Winkler444ssdff51215,Zhenddddgssddsddfff00,BlackBlacsskBlacksdff51215,Peng55667,Liuddfffff,LiLiBlacsskBlacksdff51215,Zhengsddfffsdddssddddkkllssssssssdefr23}.

Experiments (see \cite{XueXuejjainidd793,Xusddeddff345511215}) find that the oriented migration of bacteria
or cells may not be parallel to the gradient of the chemical substance, but may rather involve
rotational flux components.
 This suggests that the sensitivity function $S(n)$
should be replaced by a matrix $S(x,n,c)$. As a consequence, chemotaxis systems with such rotational flux lose their energy structure, which plays a key role
in the mathematical analysis.
Thus, a generalization of the chemotaxis-fluid system \dref{1.ss1hdffffhjjddssssghhhhhddggddddffftyy}
should be of the form
\begin{equation}
 \left\{\begin{array}{ll}
   n_t+u\cdot\nabla n=\Delta n-\nabla\cdot( nS(x,n,c)\nabla c),\quad
x\in \Omega, t>0,\\
    c_t+u\cdot\nabla c=\Delta c-nf(c),\quad
x\in \Omega, t>0,\\
u_t+\kappa (u\cdot\nabla)u+\nabla P=\Delta u+n\nabla \phi,\quad
x\in \Omega, t>0,\\
\nabla\cdot u=0,\quad
x\in \Omega, t>0,\\
 \end{array}\right.\label{1.ss1hddffhjjdllddddssssddggddddffftyy}
\end{equation}
where $S(x,n,c)$ is a tensor-valued function, which indicates the rotational effect.
This generalization results in considerable mathematical difficulties
due to the fact that chemotaxis systems \dref{1.ss1hddffhjjdllddddssssddggddddffftyy} with such rotational fluxes lose some energy structure,
which has served as a key to the analysis for scalar-valued $S$ (see Winkler  et al. \cite{Wddffang11215,Winklesssssrdddsss51215,Winklerssddddrssdff51215} and Cao-Lankeit \cite{CaoCaoLiitffg11}, Zheng \cite{Zhengffssdfffcksdff51215007}).  Consequently, some new
approaches are necessary when studying the global existence or even boundedness (see \cite{Winklesssssrdddsss51215}).
Quite a number
of results on global existence and boundedness properties have also been obtained for the variant of
\dref{1.ss1hddffhjjdllddddssssddggddddffftyy} obtained on replacing $\Delta n$ by nonlinear diffusion operators generalizing the porous medium-type
choice  $\Delta n^m$ for several ranges of $m>1$ (see e.g. \cite{Winklerssdff51215,Winkler11215,Tao71215,Tao61215}).



Motivated by the above works,   the first purpose of this paper is to show the global solvability   of
weak  solutions to the problem \dref{33dfff4451.1fghyuisda}.

Throughout this paper, we assume $\phi$ and the initial data satisfy
%
\begin{equation}
\phi\in W^{2,\infty}(\Omega)
\label{dd1.1fghyuisdakkkllljjjkk}
\end{equation}
 and 
\begin{equation}\label{ccvvx1.731426677gg}
\left\{
\begin{array}{ll}
\displaystyle{n_0\in C(\bar{\Omega})~~~~ \mbox{with}~~ n_0\geq0,},\\
\displaystyle{c_0\in W^{1,\infty}(\Omega)~~\mbox{with}~~c_0\geq0~~\mbox{in}~~\bar{\Omega},}\\
\displaystyle{u_0\in D(A^\gamma)~~\mbox{for~~ some}~~\gamma\in ( \frac{3}{4}, 1),}\\
\end{array}
\right.
\end{equation}
where $A$ denotes the Stokes operator with domain $D(A) := W^{2,{2}}(\Omega)\cap  W^{1,{2}}_0(\Omega)
\cap L^{2}_{\sigma}(\Omega)$,
and
$L^{2}_{\sigma}(\Omega) := \{\varphi\in  L^{2}(\Omega)|\nabla\cdot\varphi = 0\}$ 
 (\cite{Sohr}).

In the context of these assumptions, the first of our main results  can be read as follows.
%
\begin{theorem}\label{theorem3}
Let $\Omega \subseteq \mathbb{R}^3$ be a bounded  domain with smooth boundary.
Suppose that the assumptions \dref{x1.73142vghf48rtgyhu}--\dref{x1.73142vghf48gg} and
 \dref{dd1.1fghyuisdakkkllljjjkk}--\dref{ccvvx1.731426677gg}
 hold.
 If 
\begin{equation}\label{x1.73142vghf48}\alpha\geq1,
\end{equation}
then   problem \dref{33dfff4451.1fghyuisda}--\dref{33dfff44sdfff51.1fgsdddfhyuisda} possesses at least one
global weak solution $(n, c, u)$ in terms of Definition \ref{df1} below. Furthermore, this
solution satisfies
\begin{equation}
 \left\{\begin{array}{ll}
 n\in L^{\infty}_{loc}([0,\infty),L^2(\Omega))\cap L^{2}_{loc}([0,\infty),W^{1,2}(\Omega)),\\
  c\in  L^{\infty}_{loc}([0,\infty),L^p(\Omega))\cap L^{2}_{loc}([0,\infty),W^{1,2}(\Omega))~~~\mbox{for any}~~p>1,\\
  u\in  L^{2}_{loc}([0,\infty),W^{1,2}_{0,\sigma}(\Omega))\cap L^{\infty}_{loc}((0,\infty),L^{2}_\sigma(\Omega))\cap L^{\frac{10}{3}}_{loc}(\bar{\Omega}\times[0,\infty))\\
   \end{array}\right.\label{1.1ddfghyuisdsdddda}
\end{equation}
and there exists $C > 0$ such that
\begin{equation}
\|n(\cdot, t)\|_{L^{2\alpha}(\Omega)}+\|c(\cdot, t)\|_{L^{p}(\Omega)}+\| u(\cdot, t)\|_{L^{2}(\Omega)}\leq C~~ \mbox{for all}~~ t>0.~~
\label{1.163072xgssddgttyyu}
\end{equation}
Moreover, $(n, c, u)$ can be obtained as the limit of solutions $(n_\varepsilon, c_\varepsilon, u_\varepsilon, P_\varepsilon)$ to the regularized problems \dref{1.1fghyuisda} below in the sense that there exists $(\varepsilon_j)_{j\in \mathbb{N}}\subset (0, 1)$ such that $\varepsilon_j\searrow 0$ as $j\rightarrow\infty$  and
$$n_\varepsilon\rightarrow n~~\mbox{as well as}~~ c_\varepsilon\rightarrow c ~~\mbox{and}~~~ u_\varepsilon\rightarrow u~~\mbox{a.e.}~~ \mbox{in}~~ \Omega\times (0,\infty)$$
as $\varepsilon =\varepsilon_j\searrow0.$
\end{theorem}
Going beyond these existence  statements, the study of global dynamic is a natural
continuation. In fact, in \cite{Winklerssdd51215}, Winkler
proved  that in the  three-dimensional bounded {\bf convex} domains, Keller-Segel-Navier-Stokes system with logistic source $\rho n-\mu n^2$
 possesses at least one globally generalized solution, moreover, if $\mu>\frac{\chi\sqrt{\rho_+}}{4}$, then this solution converge to the homogeneous steady state with respect to the
topology of $L^1 (\Omega) \times L^p (\Omega) \times L^2 (\Omega)$ for $p\in[1,6)$. However, the regularity properties and the eventual
smoothness of arbitrary weak solutions are  still  open.
The present paper shows that if  $C_S<2\sqrt{C_N}$,
 after some relaxation time, this weak solution
enjoys further regularity properties and thereby complies with the concept of
weak solutions.
 In fact,  under an explicit condition on the size of $C_S$ relative to $C_N$, we can  secondly prove that in fact any such {\bf weak} solution $(n,c,u)$ becomes
smooth ultimately, and that it approaches the unique spatially homogeneous steady
state $(\bar{n}_0,\bar{n}_0,0)$, where $\bar{n}_0=\frac{1}{|\Omega|}\int_{\Omega}n_0$ and $C_N$ is the best  Poincar\'{e} constant and $C_S$ is given by  \dref{x1.73142vghf48gg}.


\begin{theorem}\label{thaaaeorem3}
 Under the assumptions of Theorem \ref{theorem3},  and assume that in addition
\begin{equation}\label{x1.73142vgssdddhfjjk48}C_S<2\sqrt{C_N},
\end{equation}
where $C_N$ is the best  Poincar\'{e} constant and $C_S$ is given by  \dref{x1.73142vghf48gg}.
Then whenever $(n_0 ,c_0 ,u_0 )$ satisfies \dref{ccvvx1.731426677gg} with $n_0\not\equiv0$.  Then
there exist $T > 0$ and $P\in C^{1,0} (\bar{\Omega}\times[T,\infty))$ such that
\begin{equation}
 \left\{\begin{array}{ll}
 n\in  C^{2,1}(\bar{\Omega}\times(T,\infty)),\\
  c\in  C^{2,1}(\bar{\Omega}\times(T,\infty)),\\
  u\in  C^{2,1}(\bar{\Omega}\times(T,\infty);\mathbb{R}^3),\\
   \end{array}\right.\label{1ssddd.1ddfghyuisda}
\end{equation}
and such that $(n,c,u)$ solves the boundary value problem in \dref{33dfff4451.1fghyuisda}--\dref{33dfff44sdfff51.1fgsdddfhyuisda} classically in
$\bar{\Omega}\times [T,\infty)$. 
Moreover,
$$n(\cdot,t)\rightarrow \bar{n}_0,~~c(\cdot,t)\rightarrow  \bar{n}_0~~\mbox{and}~~~u(\cdot,t)\rightarrow0
~~\mbox{in}~~~L^\infty(\Omega)$$
as $t \rightarrow\infty$,
where $\bar{n}_0=\frac{1}{|\Omega|}\int_{\Omega}n_0$.
\end{theorem}

\subsection{Mathematical challenges and Organization of the paper}

The global existence in three dimensions requires a more delicate proof and only results in a weak solvability. The
main reason for these difficulties is the dimension dependency of the Gagliardo-Nirenberg inequalities.
System \dref{33dfff4451.1fghyuisda}--\dref{33dfff44sdfff51.1fgsdddfhyuisda} incorporates structures of {\bf $3$-dimensional} fluid dynamics and
rotational flux, which involves more complex chemotactic cross-diffusion reinforced
by {\bf signal production} and brings
about many considerable mathematical difficulties.
%
%
%
In fact, according to  \cite{XueXuejjainidd793,Xusddeddff345511215} (see also \cite{Wddffang11215,Winkler11215,Winklesssssrdddsss51215}),
the tensor-valued sensitivity functions result in new mathematical difficulties, mainly linked to the fact that a chemotaxis system with such rotational fluxes
thereby loses an energy-like structure, which is
totally different from \cite{Winkler444ssdff51215}. Moreover, unlike the {\bf signal consumption} system
\dref{1.ss1hddffhjjdllddddssssddggddddffftyy}, we cannot gain the $L^\infty(\Omega)$ estimates of $c$ via the maximum principle
directly. To overcome these difficulties, some new
approaches are necessary when studying the global existence and stabilization for system \dref{33dfff4451.1fghyuisda}--\dref{33dfff44sdfff51.1fgsdddfhyuisda}.
%
The strategy for the proof of the Theorem \ref{thaaaeorem3}  lies in identifying the functional
$$\|n_\varepsilon-\bar{n}_0\|^{{2}}_{L^{{2}}(\Omega)}+
\frac{B}{2}\|c_\varepsilon-\bar{n}_0\|^{{2}}_{L^{{2}}(\Omega)}~~~(\mbox{for some} ~~B>0)$$
 as eventual Lyapunov functional. Apparent estimates for its dissipation rate show $L^2$-convergence, which together with the $L^p$-$L^q$ estimates for the Stokes semigroup and  a contraction mapping
 implies the 
decay of $u_\varepsilon$ in $L^p(\Omega)$ with some $p\geq6$. Next, applying
the variation-of-constants formula
for $c_{\varepsilon}-\bar{n}_0$ and $n_{\varepsilon} $, we can derive the
exponential decay for $\int_\Omega|\nabla c_\varepsilon (x,t)|^{2}$ and
$
\int_\Omega|\nabla c_\varepsilon (x,t)|^{4}$ and boundeness of  $\|n_{\varepsilon}(\cdot,t)\|_{L^{\infty}(\Omega)}$ and $\|c_{\varepsilon}(\cdot, t)\|_{W^{1,\infty}(\Omega)}$ for some large enough $t>2.$
Then based on maximal
Sobolev regularity in the Stokes evolution system and inhomogeneous linear heat
equations, one can successively obtain further ultimate regularity properties of $u_\varepsilon, c_\varepsilon$ and $n_\varepsilon$
which by standard Schauder theory imply eventual smoothness (see Lemmas \ref{x344ccffggfhhlemma45625xxhjhjuioookloghyui}--\ref{lemma45630hhuujjsdfffggguuyy}).
Finally,  collecting the   regularity for $n_\varepsilon,c_\varepsilon$ and $u_\varepsilon$ allows for turning the
weak decay information previously gathered into the desired uniform convergence
statements and thereby complete the proof of Theorem \ref{thaaaeorem3} (see Lemmas \ref{lemma45630223}--\ref{lemma4dd5630hhuujjuuyy}).

\section{Preliminaries}

Our intention is to construct a global weak solution as the limit of smooth solutions of appropriately regularized problem.
Since, the nonlinear boundary condition on $n$ ($(\nabla n-nS(x, n, c))\cdot\nu=0~~x\in \partial\Omega, t>0$) and
 the nonlinear convective term ($\kappa\neq0$ in the
Navie-Stokes subsystem of \dref{33dfff4451.1fghyuisda}--\dref{33dfff44sdfff51.1fgsdddfhyuisda}) brings about a great challenge to the study of system \dref{33dfff4451.1fghyuisda}--\dref{33dfff44sdfff51.1fgsdddfhyuisda},  we shall first follow an idea
of \cite{Wddffang11215} (see also \cite{Kegssddsddfff00,Winkler11215}) to deal with some regularized approximate problems. To this end, we  consider the following approximate system of \dref{33dfff4451.1fghyuisda}--\dref{33dfff44sdfff51.1fgsdddfhyuisda}:
\begin{equation}
\left\{\begin{array}{ll}
   n_{\varepsilon t}+u_{\varepsilon}\cdot\nabla n_{\varepsilon}=\Delta n_{\varepsilon}-\nabla\cdot(n_{\varepsilon}F_{\varepsilon}(n_{\varepsilon})S_\varepsilon(x, n_{\varepsilon}, c_{\varepsilon})\nabla c_{\varepsilon}),\quad
x\in \Omega,\; t>0,\\
    c_{\varepsilon t}+u_{\varepsilon}\cdot\nabla c_{\varepsilon}=\Delta c_{\varepsilon}-c_{\varepsilon}+n_{\varepsilon},\quad
x\in \Omega,\; t>0,\\
u_{\varepsilon t}+\nabla P_{\varepsilon}=\Delta u_{\varepsilon}-\kappa (Y_{\varepsilon}u_{\varepsilon} \cdot \nabla)u_{\varepsilon}+n_{\varepsilon}\nabla \phi,\quad
x\in \Omega,\; t>0,\\
\nabla\cdot u_{\varepsilon}=0,\quad
x\in \Omega,\; t>0,\\
 \disp{\nabla n_{\varepsilon}\cdot\nu=\nabla c_{\varepsilon}\cdot\nu=0,u_{\varepsilon}=0,\quad
x\in \partial\Omega,\; t>0,}\\
\disp{n_{\varepsilon}(x,0)=n_0(x),c_{\varepsilon}(x,0)=c_0(x),\;u_{\varepsilon}(x,0)=u_0(x)},\quad
x\in \Omega,\\
 \end{array}\right.\label{1.1fghyuisda}
\end{equation}
where
\begin{equation}
F_{\varepsilon}(s):=\frac{1}{(1+\varepsilon s)^{3}}~\quad~\mbox{for all}~~s \geq 0~~\mbox{and}~~\varepsilon> 0,
\label{1.ffggvbbnxxccvvn1}
\end{equation}
as well as
\begin{equation}
\begin{array}{ll}
S_\varepsilon(x, n, c) := \rho_\varepsilon(x)S(x, n, c),~~ x\in\bar{\Omega},~~n\geq0,~~c\geq0
 \end{array}\label{3.10gghhjuuloollyuigghhhyy}
\end{equation}
and
$$
Y_{\varepsilon}w := (1 + \varepsilon A)^{-1}w \quad~\mbox{for all}~ w\in L^2_{\sigma}(\Omega)
$$
is a standard Yosida approximation and $A$ is  the realization of the Stokes operator (see \cite{Sohr}).
Here, $(\rho_\varepsilon)_{\varepsilon\in(0,1)} \in C^\infty_0 (\Omega)$ is a family of standard cutoff functions satisfying $0\leq\rho_\varepsilon\leq 1$ in $\Omega$
and $\rho_\varepsilon\nearrow1$ in $\Omega$  as $\varepsilon\searrow0$.

Due to the  nonlinear convective term $(u \cdot \nabla)u$ in third equation of \dref{33dfff4451.1fghyuisda}--\dref{33dfff44sdfff51.1fgsdddfhyuisda}, it is hard  to find a classical
solution. We intend to obtain a weak solution by approximation. 
To begin with,
 the definition of weak solutions to \dref{33dfff4451.1fghyuisda}--\dref{33dfff44sdfff51.1fgsdddfhyuisda}  is given as follows:
 \begin{definition}\label{df1}
 By a global weak solution of \dref{33dfff4451.1fghyuisda}--\dref{33dfff44sdfff51.1fgsdddfhyuisda} we mean a triple
$(n,c,u)$ of functions
%
%
\begin{equation}
 \left\{\begin{array}{ll}
   n\in L_{loc}^1([0,\infty);\bar{\Omega}),
   \\
    c \in L_{loc}^1([0,\infty); W^{1,1}(\Omega)),\\
u \in  L_{loc}^1([0,\infty); W^{1,1}_0(\Omega);\mathbb{R}^{3}),
\\
\end{array}\right.\label{dffff1.1fghyuisdakkklll}
\end{equation}
such that $n\geq 0$ and $c\geq 0$  a.e. in $\Omega\times(0, \infty)$,
\begin{equation}\label{726291hh}
\begin{array}{rl}
&u\otimes u \in L^1_{loc}(\bar{\Omega}\times [0, \infty);\mathbb{R}^{3\times 3})~\mbox{and}~ n~\mbox{belongs to}~ L^1_{loc}(\bar{\Omega}\times [0, \infty)),\\
&cu,~ nu, ~\mbox{and}~nS(x,n,c)\nabla c~ \mbox{belong to}~
L^1_{loc}(\bar{\Omega}\times [0, \infty);\mathbb{R}^{3})
\end{array}
\end{equation}
that $\nabla\cdot u = 0$ a.e. in $\Omega\times(0, \infty)$, and that
\begin{equation}
\begin{array}{rl}\label{eqx45xx12112ccgghh}
&\disp{-\int_0^{\infty}\int_{\Omega}n\varphi_t-\int_{\Omega}n_0\varphi(\cdot,0) }
\\
=&\disp{-
\int_0^\infty\int_{\Omega}\nabla n\cdot\nabla\varphi+\int_0^\infty\int_{\Omega}n
S(x,n,c)\nabla c\cdot\nabla\varphi}
+\disp{\int_0^\infty\int_{\Omega}nu\cdot\nabla\varphi}
\end{array}
\end{equation}
for any $\varphi\in C_0^{\infty} (\bar{\Omega}\times[0, \infty))$, 
\begin{equation}
\begin{array}{rl}\label{eqx45xx12112ccgghhjj}
&\disp{-\int_0^{\infty}\int_{\Omega}c\varphi_t-\int_{\Omega}c_0\varphi(\cdot,0)}
\\
=&\disp{-
\int_0^\infty\int_{\Omega}\nabla c\cdot\nabla\varphi-\int_0^T\int_{\Omega}c\varphi+\int_0^T\int_{\Omega}n\varphi+
\int_0^\infty\int_{\Omega}cu\cdot\nabla\varphi}
\end{array}
\end{equation}
for any $\varphi\in C_0^{\infty} (\bar{\Omega}\times[0, \infty))$  as well as
\begin{equation}
\begin{array}{rl}\label{eqx45xx12112ccgghhjjgghh}
&\disp{-\int_0^{\infty}\int_{\Omega}u\varphi_t-\int_{\Omega}u_0\varphi(\cdot,0) -\kappa
\int_0^\infty\int_{\Omega} u\otimes u\cdot\nabla\varphi }
\\
=&\disp{-
\int_0^\infty\int_{\Omega}\nabla u\cdot\nabla\varphi-
\int_0^\infty\int_{\Omega}n\nabla\phi\cdot\varphi}
\end{array}
\end{equation}
for any $\varphi\in C_0^{\infty} (\bar{\Omega}\times[0, \infty);\mathbb{R}^3)$ fulfilling $\nabla\varphi\equiv 0$. 
\end{definition}

%

The following notion of weak solutions to \dref{33dfff4451.1fghyuisda}--\dref{33dfff44sdfff51.1fgsdddfhyuisda} is taken from [39]. Here and
in the sequel, for vectors $v\in \mathbb{R}^3$ and $w\in \mathbb{R}^3$ we let $v\otimes w$ denote the matrix
$(a_{ij})_{i,j\in\{1,2,3\}}\in \mathbb{R}^{3\times3}$ defined on setting $a_{ij}:= v_iw_j$ for $i,j\in \{1,2,3\}$.

The first lemma concerns the local solvability of system \dref{1.1fghyuisda} in the classical sense. Applying  the Schauder fixed point theorem as well as  the standard regularity theory of parabolic equations and the bootstrap arguments, the local existence of solutions to system \dref{33dfff4451.1fghyuisda}--\dref{33dfff44sdfff51.1fgsdddfhyuisda}
can be established by
the similar arguments as  in Lemma 2.1 of \cite{Winkler31215} (see also \cite{Winklesssssrdddsss51215}), therefore, we give the following lemma without proof.

%
%

\begin{lemma}\label{lemma70}
Let $\Omega \subseteq \mathbb{R}^3$ be a bounded  domain with smooth boundary.
Suppose that \dref{x1.73142vghf48rtgyhu} and \dref{x1.73142vghf48gg}  hold.
%
Then there exist $T_{max,\varepsilon}\in  (0,\infty]$ and
a classical solution $(n_\varepsilon , c_\varepsilon ,u_\varepsilon , P_\varepsilon )$ of \dref{33dfff4451.1fghyuisda}--\dref{33dfff44sdfff51.1fgsdddfhyuisda} in
$\Omega\times(0, T_{max,\varepsilon})$ such that
\begin{equation}
 \left\{\begin{array}{ll}
 n_\varepsilon\in C^0(\bar{\Omega}\times[0,T_{max,\varepsilon}))\cap C^{2,1}(\bar{\Omega}\times(0,T_{max,\varepsilon})),\\
  c_\varepsilon \in  C^0(\bar{\Omega}\times[0,T_{max,\varepsilon}))\cap C^{2,1}(\bar{\Omega}\times(0,T_{max,\varepsilon})),\\
  u_\varepsilon \in  C^0(\bar{\Omega}\times[0,T_{max,\varepsilon}))\cap C^{2,1}(\bar{\Omega}\times(0,T_{max,\varepsilon})),\\
  P_\varepsilon \in  C^{1,0}(\bar{\Omega}\times(0,T_{max,\varepsilon}))\\
   \end{array}\right.\label{1.1ddfghyuisda}
\end{equation}
 classically solving \dref{33dfff4451.1fghyuisda}--\dref{33dfff44sdfff51.1fgsdddfhyuisda} in $\Omega\times[0,T_{max,\varepsilon})$.
%
Moreover,  $n_\varepsilon $ and $c_\varepsilon $ are nonnegative in
$\Omega\times(0, T_{max,\varepsilon})$, and
\begin{equation}
\|n_\varepsilon (\cdot, t)\|_{L^\infty(\Omega)}+\|c_\varepsilon (\cdot, t)\|_{W^{1,\infty}(\Omega)}+\|A^\gamma u_\varepsilon (\cdot, t)\|_{L^{2}(\Omega)}\rightarrow\infty~~ \mbox{as}~~ t\nearrow T_{max,\varepsilon},
\label{1.163072x}
\end{equation}
where $\gamma$ is given by \dref{ccvvx1.731426677gg}.
\end{lemma}

Now, we let $(n_\varepsilon,c_\varepsilon  ,u_\varepsilon )$ denote the local solution to \dref{1.1fghyuisda} with the maximal time
give by $T_{max,\varepsilon}$. In what follows, let
$C, C_i,C_{i,*}, \tilde{C}_i, \tilde{C}_{i,*}, \rho_{i},\lambda_{i,*},\rho_{i,*},\tilde{\rho}_{i,*},  M, M_i,\gamma_i,\alpha_i$ and $\mu_i$ denote some different constants, which are independent of $\varepsilon$ and $T_{max,\varepsilon}$, and
if no special explanation, they depend at most on $\Omega, \alpha,  \nabla\phi, n_0 , c_0$ and  $u_0$.

Let us start with recalling some properties which have been established in previous studies and
are fundamental when discussing results concerning the global existence of classical solutions in
the setting of \dref{33dfff4451.1fghyuisda}--\dref{33dfff44sdfff51.1fgsdddfhyuisda}.
The following estimates of $n_\varepsilon$ and $c_\varepsilon$ are basic but important in the proof of our result.

\begin{lemma}\label{fvfgfflemma45}
The solution of \dref{33dfff4451.1fghyuisda}--\dref{33dfff44sdfff51.1fgsdddfhyuisda} satisfies
\begin{equation}
\int_{\Omega}{n_\varepsilon}= \int_{\Omega}{n_{0}}~~\mbox{for all}~~ t\in(0, T_{max,\varepsilon})
\label{ddfgczhhhh2.5ghju48cfg924ghyuji}
\end{equation}
as well as
\begin{equation}
\int_{\Omega}{c_\varepsilon}\leq \max\{\int_{\Omega}{n_{0}},\int_{\Omega}{c_{0}}\}~~\mbox{for all}~~ t\in(0, T_{max,\varepsilon}).
\label{2344ddfgczhhhh2.5ghju48cfg924ghyuji}
\end{equation}
%
%
\end{lemma}

In the following,  we proceed to derive $\varepsilon$-independent estimates. 
In fact, using the idea comes from \cite{Kegssddsddfff00} and \cite{Zhenddsdddddgssddsddfff00}, one could  propose some regularity estimates
for $n_\varepsilon$ and $c_\varepsilon$ by tracking the time evolution of a certain combinational functional
of them.
%
Here we only state it as a lemma, while for the detailed proof readers can refer to Lemma 3.3 of \cite{Kegssddsddfff00}.


\begin{lemma}\label{lemmaghjffggsjjjjjsddgghhmk4563025xxhjklojjkkk}
 Suppose that 
 \dref{dd1.1fghyuisdakkkllljjjkk}--\dref{ccvvx1.731426677gg} and
 \dref{x1.73142vghf48rtgyhu}--\dref{x1.73142vghf48gg}
 hold with  $\alpha>\frac{1}{3}$.
Then
one can find $C > 0$   independent of $\varepsilon$ such that for all $t\in(0, T_{max,\varepsilon})$, the solution of \dref{33dfff4451.1fghyuisda}--\dref{33dfff44sdfff51.1fgsdddfhyuisda}  satisfies
\begin{equation}
\begin{array}{rl}
&\disp{\int_{\Omega} n_{\varepsilon}^{2\alpha }(x,t)+\int_{\Omega}   c_{\varepsilon}^2(x,t)+\int_{\Omega}  | {u_{\varepsilon}}(x,t)|^2\leq C~\mbox{for all}~ t\in (0, T_{max,\varepsilon}).}\\
\end{array}
\label{111czfvgb2.5ghhjuyuccvviihjj}
\end{equation}
Moreover, for all $t\in(0, T_{max,\varepsilon}-\tau)$,
it holds that
one can find a constant $C > 0$   such that
\begin{equation}
\begin{array}{rl}
&\disp{\int_{t}^{t+\tau}\int_{\Omega} \left[  n_{\varepsilon}^{2\alpha-2} |\nabla {n_{\varepsilon}}|^2+ |\nabla {c_{\varepsilon}}|^2+ |\nabla {u_{\varepsilon}}|^2\right]\leq C,}\\
\end{array}
\label{bnmbncz2.5ghhjuyuivvbnnihjj}
\end{equation}
where \begin{equation}\tau=\min\{1,\frac{1}{6}T_{max,\varepsilon}\}.
\label{bnmbncz2.5ghhjuyussdddsddisdddddvvbnnihjj}
\end{equation}
\end{lemma}

With Lemma \ref{lemmaghjffggsjjjjjsddgghhmk4563025xxhjklojjkkk} at hand,  we are now in the position to show the solution of the approximate
problem \dref{1.1fghyuisda} is actually global in time, that is,  $T_{max,\varepsilon}=\infty$.


\begin{lemma}\label{kkklemmaghjmk4563025xxhjklojjkkk}
Let $\alpha\geq1$. Then
for all $\varepsilon\in(0,1),$ the solution of  \dref{1.1fghyuisda} is global in time.
\end{lemma}
\begin{proof}
Firstly, we will
 prove that for any $\tau\in (0,T_{max,\varepsilon} ),$
\begin{equation}
\begin{array}{rl}
\|A^\gamma u_{\varepsilon}(\cdot, t)\|_{L^2(\Omega)}+\|c_{\varepsilon}(\cdot, t)\|_{W^{1,\infty}(\Omega)}+\|n_{\varepsilon}(\cdot, t)\|_{L^{\infty}(\Omega)}\leq&\disp{C({\tau})~~ \mbox{for all}~~ t\in(\tau,T_{max,\varepsilon})}\\
\end{array}
\label{cz2ddff.57151ccvssssssdffffhhjjjkkkuuifghhhivhccvvhjjjkkhhggjjllll}
\end{equation}
holds with some $C(\tau) > 0$. To do so, we fix $\tau \in(0,T_{max,\varepsilon})$ such that $\tau < 1.$
In view of  \dref{1.ffggvbbnxxccvvn1}, we derive
$$F_{\varepsilon}(n_{\varepsilon})\leq\disp\frac{1}{\varepsilon^3 n_{\varepsilon}^3},$$
so that, by multiplying the first equation in $\dref{1.1fghyuisda}$ by $ n_{\varepsilon}^{5+2\alpha}$ and using $\nabla\cdot u_\varepsilon=0$,
\begin{equation}
\begin{array}{rl}
&\disp{\frac{1}{{6+2\alpha}}\frac{d}{dt}\|{ n_{\varepsilon} }\|^{{{6+2\alpha}}}_{L^{{6+2\alpha}}(\Omega)}+
({5+2\alpha})\int_{\Omega}  n_{\varepsilon}^{4+2\alpha}|\nabla n_{\varepsilon}|^2}\\
=&\disp{-
\int_{\Omega}  n_{\varepsilon}^{5+2\alpha}\nabla\cdot(n_{\varepsilon}\frac{1}{(1+\varepsilon n_{\varepsilon})^3}S_\varepsilon(x, n_{\varepsilon}, c_{\varepsilon})\cdot\nabla c_{\varepsilon})}\\
\leq&\disp{({5+2\alpha})
\int_{\Omega}  n_{\varepsilon}^{5+2\alpha}\frac{1}{(1+\varepsilon n_{\varepsilon})^3}|S_\varepsilon(x, n_{\varepsilon}, c_{\varepsilon})||\nabla n_{\varepsilon}||\nabla c_{\varepsilon}|~~\mbox{for all}~~ t\in(0, T_{max,\varepsilon}).}
\end{array}
\label{55hhjjcffghhhjkklddffssddgglffghhhz2.5}
\end{equation}
Recalling \dref{x1.73142vghf48gg}, by Young inequality, one can see that
\begin{equation}
\begin{array}{rl}
&\disp\int_{\Omega} n_{\varepsilon}^{5+2\alpha}\frac{1}{(1+\varepsilon n_{\varepsilon})^3}|S_\varepsilon(x, n_{\varepsilon}, c_{\varepsilon})||\nabla n_{\varepsilon}||\nabla c_{\varepsilon}|\\
\leq&\disp C_S\int_{\Omega} n_{\varepsilon}^{5+2\alpha}\frac{1}{(1+\varepsilon n_{\varepsilon})^3}(1+n_\varepsilon)^{-\alpha}|\nabla n_{\varepsilon}||\nabla c_{\varepsilon}|\\
\leq&\disp{\frac{1}{\varepsilon^3}C_S\int_{\Omega} n_{\varepsilon}^{2+\alpha} |\nabla n_{\varepsilon}||\nabla c_{\varepsilon}|}\\
\leq&\disp{\frac{({5+2\alpha})}{2}\int_{\Omega} n_{\varepsilon}^{4+2\alpha} |\nabla n_{\varepsilon}|^2+C_1\int_{\Omega} |\nabla c_{\varepsilon}|^2~~\mbox{for all}~~ t\in(0, T_{max,\varepsilon}),}
\end{array}
\label{55hhjjcffghhhjkkllfffghggggghhfghhhz2.5}
\end{equation}
where $C_1$ is a positive constant, as all subsequently appearing constants $C_2, C_3, \ldots$ possibly depend on
 $\varepsilon$.
 This combined with \dref{bnmbncz2.5ghhjuyuivvbnnihjj} implies that
%
 \begin{equation}
\begin{array}{rl}
&\disp{\int_{\Omega}n^{6+2\alpha}_{\varepsilon}(x,t)\leq C_2~\mbox{for all}~ t\in (0, T_{max,\varepsilon})}\\
\end{array}
\label{czfvgb2.5ghhjuyucffhhhhhhhjjggcvviihjj}
\end{equation}
holds.
 Relying on properties of the Yosida approximation
$Y_{\varepsilon}$, we can also derive that
%
\begin{equation}
\|Y_{\varepsilon}u_{\varepsilon}(\cdot,t)\|_{L^\infty(\Omega)}=\|(I+\varepsilon A)^{-1}u_{\varepsilon}(\cdot,t)\|_{L^\infty(\Omega)}\leq C_3\|u_{\varepsilon}(\cdot,t)\|_{L^2(\Omega)}\leq C_4~~\mbox{for all}~~t\in(0,T_{max,\varepsilon})
\label{ssdcfvgdhhjjdfghgghjjnnhhkklld911cz2.5ghju48}
\end{equation}
and  some positive constansts $C_3$ and $C_4$ depend  on $\varepsilon.$

Next,  testing the projected Stokes equation $u_{\varepsilon t} +Au_{\varepsilon} =  \mathcal{P}[-\kappa (Y_{\varepsilon}u_{\varepsilon} \cdot \nabla)u_{\varepsilon}+n_{\varepsilon}\nabla \phi]$ by $Au_{\varepsilon}$, we derive from the Young inequality as well as  \dref{czfvgb2.5ghhjuyucffhhhhhhhjjggcvviihjj} and \dref{ssdcfvgdhhjjdfghgghjjnnhhkklld911cz2.5ghju48} that
%
\begin{equation}
\begin{array}{rl}
&\disp{\frac{1}{{2}}\frac{d}{dt}\|A^{\frac{1}{2}}u_{\varepsilon}\|^{{{2}}}_{L^{{2}}(\Omega)}+
\int_{\Omega}|Au_{\varepsilon}|^2 }\\
=&\disp{ \int_{\Omega}Au_{\varepsilon}\mathcal{P}(-\kappa
(Y_{\varepsilon}u_{\varepsilon} \cdot \nabla)u_{\varepsilon})+ \int_{\Omega}\mathcal{P}[n_{\varepsilon}\nabla\phi] Au_{\varepsilon}}\\
\leq&\disp{ \frac{1}{2}\int_{\Omega}|Au_{\varepsilon}|^2+\kappa^2\int_{\Omega}
|(Y_{\varepsilon}u_{\varepsilon} \cdot \nabla)u_{\varepsilon}|^2+ \|\nabla\phi\|^2_{L^\infty(\Omega)}\int_{\Omega}n_{\varepsilon}^2}\\
\leq&\disp{ \frac{1}{2}\int_{\Omega}|Au_{\varepsilon}|^2+C_5\int_{\Omega}
|\nabla u_{\varepsilon}|^2+ C_6~~\mbox{for all}~~t\in(0,T_{max,\varepsilon}),}\\
\end{array}
\label{ddfghgghjjnnhhkklld911cz2.5ghju48}
\end{equation}
which together with
\dref{bnmbncz2.5ghhjuyuivvbnnihjj} implies
 \begin{equation}
\begin{array}{rl}
&\disp{\int_{\Omega}|\nabla u_{\varepsilon}(x,t)|^{2}\leq C_7~~~\mbox{for all}~~ t\in (0, T_{max,\varepsilon})}\\
\end{array}
\label{111czfvgsssssb2.5gcccchddhjuyucffhhhhhhhjjggcvviihjj}
\end{equation}
 by some basic calculation.
Now, let $h_{\varepsilon}(x,t)=\mathcal{P}[n_{\varepsilon}\nabla \phi-\kappa (Y_{\varepsilon}u_{\varepsilon} \cdot \nabla)u_{\varepsilon} ]$.
Then combined with \dref{111czfvgsssssb2.5gcccchddhjuyucffhhhhhhhjjggcvviihjj}  yields  to \begin{equation}
\|h_{\varepsilon}(\cdot,t)\|_{L^{2}(\Omega)}\leq C_{8} ~~~\mbox{for all}~~ t\in(0,T_{max,\varepsilon})
\label{33444cfghhh29fgsssgggx96302222114}
\end{equation}
by using \dref{czfvgb2.5ghhjuyucffhhhhhhhjjggcvviihjj} and \dref{111czfvgsssssb2.5gcccchddhjuyucffhhhhhhhjjggcvviihjj}. Now, we express $A^\gamma u_{\varepsilon}$ by its variation-of-constants representation and make use of well-known
smoothing properties of the Stokes semigroup (\cite{Giga1215}) to obtain $C_{9}> 0$ such that for any $\tau\in (0,T_{max,\varepsilon} )$,
\begin{equation}
\begin{array}{rl}
\|A^\gamma u_{\varepsilon}(\cdot, t)\|_{L^2(\Omega)}\leq&\disp{C_{9}~~ \mbox{for all}~~ t\in(\tau,T_{max,\varepsilon}),}\\
\end{array}
\label{cz2.57151ccvvhccvvhjjjkkhhggdddjjllll}
\end{equation}
where $\gamma\in (\frac{3}{4}, 1)$. Since,
 $D(A^\gamma)$ is continuously embedded into $L^\infty(\Omega)$ by $\gamma>\frac{3}{4}$, so that, \dref{cz2.57151ccvvhccvvhjjjkkhhggdddjjllll} yields to
  for some positive constant $C_{10}$ such that
 \begin{equation}
\begin{array}{rl}
\|u_{\varepsilon}(\cdot, t)\|_{L^\infty(\Omega)}\leq  C_{10}~~ \mbox{for all}~~ t\in(\tau,T_{max,\varepsilon}).\\
\end{array}
\label{cz2.5jkkcvvvhjkfffffkhhgll}
\end{equation}

Now,  test the second equation of $\dref{1.1fghyuisda}$ by $-\Delta c_{\varepsilon}$ and obtain, upon two applications of
Young's inequality, that
\begin{equation}
\begin{array}{rl}
&\disp\frac{1}{{2}}\disp\frac{d}{dt}\|\nabla{c_{\varepsilon}}\|^{{{2}}}_{L^{{2}}(\Omega)}+
\int_{\Omega} |\Delta c_{\varepsilon}|^2+ \int_{\Omega} |\nabla c_{\varepsilon}|^2\\
=&\disp{-\int_{\Omega} n_{\varepsilon}\Delta c_{\varepsilon}+\int_{\Omega} \Delta c_{\varepsilon}u_{\varepsilon}\cdot\nabla c_{\varepsilon}}\\
\leq&\disp{\frac{1}{2}\int_{\Omega} |\Delta c_{\varepsilon}|^2+\int_{\Omega} n_{\varepsilon}^2+\sup_{t\in(\tau,T_{max,\varepsilon})}\|u_{\varepsilon}(\cdot,t)\|^2_{L^\infty(\Omega)}\int_{\Omega} |\nabla c_{\varepsilon}|^2 ~~\mbox{for all}~~ t\in(\tau, T_{max,\varepsilon}).}\\
\end{array}
\label{hhxxcdfvvjjczhghhhjj2.5}
\end{equation}
 Recalling the bounds provided by  \dref{cz2.5jkkcvvvhjkfffffkhhgll} as well as  \dref{czfvgb2.5ghhjuyucffhhhhhhhjjggcvviihjj} and \dref{111czfvgb2.5ghhjuyuccvviihjj}, this immediately implies
\begin{equation}
\|c_{\varepsilon}(\cdot, t)\|_{W^{1,2}(\Omega)}\leq C_{11}~~\mbox{for all}~~ t\in(\tau, T_{max,\varepsilon})
\label{ddxxxcvvdsssdddcvddffbbggddczv.5ghcfg924ghyuji}
\end{equation}
by some calculation.
In light of Lemma 2.1 of \cite{Ishida} and the Young inequality, we have
\begin{equation}
\begin{array}{rl}
&\|\nabla c_{\varepsilon}(\cdot, t)\|_{L^{\infty}(\Omega)}\\
\leq&\disp{C_{12}(1+\|c_{\varepsilon}(\cdot, t)-u_{\varepsilon}(\cdot, t)\cdot c_{\varepsilon}(\cdot, t)\|_{L^4(\Omega)})}\\
\leq&\disp{C_{12}(1+\|c_{\varepsilon}(\cdot, t)\|_{L^4(\Omega)}+\|u_{\varepsilon}(\cdot, t)\|_{L^\infty(\Omega)} \|\nabla c_{\varepsilon}(\cdot, t)\|_{L^4(\Omega)})}\\
\leq&\disp{C_{12}(1+\|c_{\varepsilon}(\cdot, t)\|_{L^4(\Omega)}+\|u_{\varepsilon}(\cdot, t)\|_{L^\infty(\Omega)} \|\nabla c_{\varepsilon}(\cdot, t)\|_{L^\infty(\Omega)}^{\frac{1}{2}}\|\nabla c_{\varepsilon}(\cdot, t)\|_{L^2(\Omega)}^{\frac{1}{2}})}\\
\leq&\disp{C_{13}(1+ \|\nabla c_{\varepsilon}(\cdot, t)\|_{L^\infty(\Omega)}^{\frac{1}{2}})~~ \mbox{for all}~~ t\in(\tau,T_{max,\varepsilon}),}\\
\end{array}
\label{zjccffgbhjcvvvbscddddz2.5297x96301ku}
\end{equation}
which combined  with \dref{ddxxxcvvdsssdddcvddffbbggddczv.5ghcfg924ghyuji} implies that
 \begin{equation}
\begin{array}{rl}
\|c_{\varepsilon}(\cdot, t)\|_{W^{1,\infty}(\Omega)}\leq  C_{14}~~ \mbox{for all}~~ t\in(\tau,T_{max,\varepsilon})\\
\end{array}
\label{cz2.5jkkcvvvhjkmmffffllfkhhgll}
\end{equation}
by using the Gagliardo--Nirenberg inequality.

Furthermore, applying the variation-of-constants formula to the $n_{\varepsilon}$-equation in \dref{1.1fghyuisda}, we derive that  for any $T\in (\tau, T_{max,\varepsilon})$,
\begin{equation}
n_{\varepsilon}(\cdot,t)=e^{(t-\tau)\Delta}n_{\varepsilon}(\cdot,\tau)-\int_{\tau}^{t}e^{(t-s)\Delta}\nabla\cdot(n_{\varepsilon}(\cdot,s)\tilde{h}_{\varepsilon}(\cdot,s)) ds,~~ t\in(\tau, T),
\label{5555fghbnmcz2.5ghjjjkkklu48cfg924ghyuji}
\end{equation}
where $\tilde{h}_{\varepsilon}:=\frac{1}{(1+\varepsilon n_{\varepsilon})^3}S_\varepsilon(x, n_{\varepsilon}, c_{\varepsilon})\nabla c_{\varepsilon}+u_\varepsilon$. Next, by \dref{x1.73142vghf48gg}, \dref{cz2.5jkkcvvvhjkmmffffllfkhhgll} and \dref{cz2.5jkkcvvvhjkfffffkhhgll}, we also have
 $$
\begin{array}{rl}
\|\tilde{h}_{\varepsilon}(\cdot, t)\|_{L^{\infty}(\Omega)}\leq  C_{15}~~ \mbox{for all}~~ t\in(\tau,T_{max,\varepsilon}),\\
\end{array}
$$
 so that, we can thus estimate
\begin{equation}
\begin{array}{rl}
\disp\|n_{\varepsilon}(\cdot,t)\|_{L^\infty(\Omega)}\leq&\|e^{(t-\tau)\Delta}n_{\varepsilon}(\cdot,\tau)\|_{L^\infty(\Omega)}+\disp\int_{\tau}^t\| e^{(t-s)\Delta}\nabla\cdot(n_{\varepsilon}(\cdot,s)\tilde{h}_{\varepsilon}(\cdot,s)\|_{L^\infty(\Omega)}ds\\
\leq&\disp C_{16}(t-\tau)^{-\frac{3}{2}}\|n_0\|_{L^1(\Omega)}+\disp C_{16}\int_{\tau}^t(t-s)^{-\frac{7}{8}}e^{-\lambda_1(t-s)}\|n_{\varepsilon}(\cdot,s)\tilde{h}_{\varepsilon}(\cdot,s)\|_{L^4(\Omega)}ds\\
\leq&\disp C_{18}+\disp C_{17}\int_{\tau}^t(t-s)^{-\frac{7}{8}}e^{-\lambda_1(t-s)}\|n_{\varepsilon}(\cdot,s)\|_{L^4(\Omega)}ds\\
\leq&\disp C_{18}+\disp C_{19}\int_{\tau}^t(t-s)^{-\frac{7}{8}}e^{-\lambda_1(t-s)}\|n_{\varepsilon}(\cdot,s)\|_{L^\infty(\Omega)}^{\frac{1}{2}}
\|n_{\varepsilon}(\cdot,s)\|_{L^2(\Omega)}^{\frac{1}{2}}ds\\
\leq&\disp C_{18}+\disp C_{20}\sup_{s\in(\tau, T_{max,\varepsilon})}\|n_{\varepsilon}(\cdot,s)\|_{L^\infty(\Omega)}^{\frac{1}{2}}~~\mbox{for all}~~ t\in(\tau, T),\\
\end{array}
\label{ccvbccvvbbnnndffghhjjvcvvbcsssscfbbnfgbghjjccmmllffvvggcvvvvbbjjkkdffzjscz2.5297x9630xxy}
\end{equation}
where
$$\begin{array}{rl}\disp C_{20}=\int_{0}^{\infty}s^{-\frac{7}{8}}e^{-\lambda_1s}\|n_{\varepsilon}(\cdot,s)\|_{L^2(\Omega)}^{\frac{1}{2}}
&\disp{<+\infty,}\\
\end{array}
$$
$\lambda_1$ is the first nonzero eigenvalue of $-\Delta$ on $\Omega$ under
the Neumann boundary condition.
And  thereby
\begin{equation}
\begin{array}{rl}
\|n_{\varepsilon}(\cdot, t)\|_{L^{\infty}(\Omega)}\leq&\disp{C_{21}~~ \mbox{for all}~~ t\in(\tau,T_{max,\varepsilon})}\\
\end{array}
\label{cz2ddff.57151ccvsssshhjjjkkkuuifghhhivhccvvhjjjkkhhggjjllll}
\end{equation}
by some basic calculation and $T\in(\tau,T_{max,\varepsilon})$ was
arbitrary.

Together with   \dref{cz2.57151ccvvhccvvhjjjkkhhggdddjjllll},
\dref{cz2.5jkkcvvvhjkmmffffllfkhhgll} and \dref{cz2ddff.57151ccvsssshhjjjkkkuuifghhhivhccvvhjjjkkhhggjjllll}, yields to \dref{cz2ddff.57151ccvssssssdffffhhjjjkkkuuifghhhivhccvvhjjjkkhhggjjllll}.
Assume that $T_{max,\varepsilon}< \infty$. Due to \dref{cz2ddff.57151ccvssssssdffffhhjjjkkkuuifghhhivhccvvhjjjkkhhggjjllll},
 we apply Lemma \ref{lemma70} to reach a contradiction and hence the proof of this Lemma is completed.
\end{proof}

\section{Global existence for   problem \dref{33dfff4451.1fghyuisda}--\dref{33dfff44sdfff51.1fgsdddfhyuisda}}
In order to prove Theorem \ref{theorem3}, we must state some further regularity properties for the approximate solutions obtained in
the last section.
To this end,  by Lemma \ref{aasslemmafggg78630jklhhjj}, we immediately
obtain the boundedness of $c_\varepsilon$ with respect to the norm in $L^p (\Omega)$ for any $p>1$.

\begin{lemma}\label{aasslemmafggg78630jklhhjj}
Let $(n_\varepsilon,c_\varepsilon,u_\varepsilon)$ be the solution of \dref{1.1ddfghyuisda}. Then there exists a positive constant $C$ independent of $\varepsilon$ such that
\begin{equation}
\|c_{\varepsilon}(\cdot, t)\|_{L^p(\Omega)}+\|n_{\varepsilon}(\cdot, t)\|_{L^2(\Omega)}\leq C~~ \mbox{for all}~~ t>0
\label{3.10gghhjuuloollgghhhyhh}
\end{equation}
and any $p>2.$
Moreover,  there exists $C>0$ independent of $\varepsilon$ such that, for each $T>0$,
\begin{equation}
\begin{array}{rl}
&\disp{\int_{0}^{ T}\int_{\Omega}c^{{{p}-2}}_{\varepsilon}|\nabla c_{\varepsilon}|^2+\int_{0}^{ T}\int_{\Omega}|\nabla n_{\varepsilon}|^2\leq C(T+1).}\\
\end{array}
\label{bnmbncz2.5ghhjugghjllllljdfghhjjyuivssdddvbnklllnihjj}
\end{equation}
\end{lemma}

\begin{proof}
Firstly,
multiplying  the first equation in $\dref{1.1fghyuisda}$ by $ n_{\varepsilon}$, in view of  \dref{1.ffggvbbnxxccvvn1}
 and  using $\nabla\cdot u_\varepsilon=0$, we derive
\begin{equation}
\begin{array}{rl}
&\disp{\frac{1}{{2}}\frac{d}{dt}\|{ n_{\varepsilon} }\|^{{{{2}}}}_{L^{{{2}}}(\Omega)}+
\int_{\Omega}  |\nabla n_{\varepsilon}|^2}\\
=&\disp{-
\int_{\Omega}  n_{\varepsilon}\nabla\cdot(n_{\varepsilon}F_{\varepsilon}(n_{\varepsilon})S_\varepsilon(x, n_{\varepsilon}, c_{\varepsilon})\cdot\nabla c_{\varepsilon})}\\
\leq&\disp{
\int_{\Omega}  n_{\varepsilon}F_{\varepsilon}(n_{\varepsilon})|S_\varepsilon(x, n_{\varepsilon}, c_{\varepsilon})||\nabla n_{\varepsilon}||\nabla c_{\varepsilon}|~\mbox{for all}~ t>0.}
\end{array}
\label{55hhjjcffghhhjkklddfgggffgglffghhhz2.5}
\end{equation}
Recalling \dref{x1.73142vghf48gg} and \dref{1.ffggvbbnxxccvvn1} and using $\alpha\geq1$, via the Young inequality, we derive
\begin{equation}
\begin{array}{rl}
&\disp\int_{\Omega} n_{\varepsilon}F_{\varepsilon}(n_{\varepsilon})|S_\varepsilon(x, n_{\varepsilon}, c_{\varepsilon})||\nabla n_{\varepsilon}||\nabla c_{\varepsilon}|\\
\leq&\disp{C_S\int_{\Omega} |\nabla n_{\varepsilon}||\nabla c_{\varepsilon}|}\\
\leq&\disp{\frac{1}{2}\int_{\Omega} |\nabla n_{\varepsilon}|^2+\frac{C_S^2}{2}\int_{\Omega} |\nabla c_{\varepsilon}|^2~\mbox{for all}~ t>0.}
\end{array}
\label{55hhjjcffghhhjkkllfffghggggghgghjjjhfghhhz2.5}
\end{equation}
Here we have used the fact that
$$ n_{\varepsilon}F_{\varepsilon}(n_{\varepsilon})|S_\varepsilon(x, n_{\varepsilon}, c_{\varepsilon})|\leq C_S n_{\varepsilon}(1 + n_{\varepsilon})^{-\alpha}\leq C_S$$
by using \dref{x1.73142vghf48gg} and $\alpha\geq1$.
Therefore, collecting  \dref{55hhjjcffghhhjkklddfgggffgglffghhhz2.5} and \dref{55hhjjcffghhhjkkllfffghggggghgghjjjhfghhhz2.5} and using \dref{bnmbncz2.5ghhjuyuivvbnnihjj}, we conclude that there exist positive constants $C_1$ and  $C_2$ such that
 \begin{equation}
\begin{array}{rl}
&\disp{\int_{\Omega} n_{\varepsilon}^{2}(x,t)\leq C_{1}~~\mbox{for all}~ t>0}\\
\end{array}
\label{czfvgb2.5ghhjuyucfkllllfhhhhhhhjjggcvviihjj}
\end{equation}
and
\begin{equation}
\begin{array}{rl}
&\disp{\int_{0}^{ T}\int_{\Omega}  |\nabla {n_{\varepsilon}}|^2\leq C_{2}(T+1)}\\
\end{array}
\label{bnmbncz2.5ghhjugghjllllljdfghhjjyuivvbnklllnihjj}
\end{equation}
for all $T>0$.
Let $p>2$. Testing the second equation in $\dref{1.1fghyuisda}$ by ${c^{p-1}_{\varepsilon}}$, using the fact $\nabla\cdot u_{\varepsilon}=0$,
and applying the H\"{o}lder inequality, we have
\begin{equation}
\begin{array}{rl}
&\disp{\frac{1}{p}\frac{d}{dt}\int_{\Omega}c^{{{p}}}_{\varepsilon}+({{p}-1})
\int_{\Omega}c^{{{p}-2}}_{\varepsilon}|\nabla c_{\varepsilon}|^2+\int_{\Omega}c^{{{p}}}_{\varepsilon}}
\\
=&\disp{\int_\Omega c^{p-1}_{\varepsilon}n_{\varepsilon}}
\\
\leq&\disp{\|n_{\varepsilon}(\cdot,t)\|_{L^2(\Omega)}
\left(\int_\Omega c^{2(p-1)}_{\varepsilon}\right)^{\frac{{1}}{2}}}\\
\leq&\disp{C_1^{\frac{1}{2}}
\left(\int_\Omega c^{2(p-1)}_{\varepsilon}\right)^{\frac{{1}}{2}}
~~\mbox{for all}~~t>0}\\
\end{array}
\label{3333cz2.5114114}
\end{equation}
by using \dref{czfvgb2.5ghhjuyucfkllllfhhhhhhhjjggcvviihjj}.
Then the Gagliardo--Nirenberg inequality and \dref{2344ddfgczhhhh2.5ghju48cfg924ghyuji} enables us to see that
$$
\begin{array}{rl}
&\disp\left(\int_\Omega c^{2(p-1)}_{\varepsilon}\right)^{\frac{{1}}{2}}
\\
=&\disp{\|  c^{\frac{p}{2}}_{\varepsilon}\|^{\frac{2(p-1)}{p}}_{L^{\frac{4(p-1)}{p}}(\Omega)}}
\\
\leq&\disp{C_{2}(\|\nabla c^{\frac{p}{2}}_{\varepsilon}\|_{L^2(\Omega)}^{\frac{6(p-2)}{3p-2}}\| c^{\frac{p}{2}}_{\varepsilon}\|_{L^{\frac{2}{p}}(\Omega)}^{\frac{2(p-1)}{p}-\frac{6(p-2)}{(3p-2)}}
+\|c^{\frac{p}{2}}_{\varepsilon}\|_{L^\frac{2}{p}(\Omega)}^{\frac{2(p-1)}{p}})}\\
\leq&\disp{C_{3}(\|\nabla c^{\frac{p}{2}}_{\varepsilon}\|_{L^2(\Omega)}^{\frac{6(p-2)}{3p-2}}+1)}~~\mbox{for all}~~t>0
\end{array}
$$
and some positive constants  $C_2$ and $C_3$.
Here we use the Young inequality to obtain  that there is $C_4>0$ such that
\begin{equation}
\begin{array}{rl}
&\disp{\frac{1}{p}\frac{d}{dt}\int_{\Omega}c^{{{p}}}_{\varepsilon}+({{p}-1})\int_{\Omega}c^{{{p}-2}}_{\varepsilon}|\nabla c_{\varepsilon}|^2+\int_{\Omega}c^{{{p}}}_{\varepsilon}}
\\
\leq&\disp{\frac{({{p}-1})}{2}\int_{\Omega}c^{{{p}-2}}_{\varepsilon}|\nabla c_{\varepsilon}|^2
+C_4~~\mbox{for all}~~t>0.}\\
\end{array}
\label{3333cz2.5ssdffggg114114}
\end{equation}
Whereupon, our conclusion comes from \dref{czfvgb2.5ghhjuyucfkllllfhhhhhhhjjggcvviihjj}
 as well as  \dref{bnmbncz2.5ghhjugghjllllljdfghhjjyuivvbnklllnihjj} and \dref{3333cz2.5ssdffggg114114} by using
some basic calculation. 
\end{proof}

By interpolation and the embedding theorem, the estimates from Lemma  \ref{aasslemmafggg78630jklhhjj} and Lemma \ref{lemmaghjffggsjjjjjsddgghhmk4563025xxhjklojjkkk} imply bounds for further spatio-temporal integrals.
\begin{lemma}\label{lemmddaghjsffggggsddgghhmk4563025xxhjklojjkkk}
Let $\alpha\geq1$. Then there exists $C>0$ independent of $\varepsilon$ such that, for each $T>0$, the solution of \dref{1.1fghyuisda} satisfies
\begin{equation}
\begin{array}{rl}
&\disp{\int_{0}^T\int_{\Omega}\left[ |u_{\varepsilon}|^{\frac{10}{3}}+n_{\varepsilon}^{\frac{6+4\alpha}{3}}\right]dt+\int_{0}^T\|u_{\varepsilon}(\cdot,t)\|^2_{L^6(\Omega)}dt\leq C(T+1)}\\
\end{array}
\label{bnmbncz2.ffgddffffhh5ghhjuyuivvbnnihjj}
\end{equation}
as well as
\begin{equation}
 \begin{array}{ll}
  \disp\int_0^T\int_{\Omega}\left[|n_\varepsilon F_{\varepsilon}(n_{\varepsilon})S_\varepsilon(x, n_{\varepsilon}, c_{\varepsilon})\nabla c_\varepsilon|^{2}+|n_\varepsilon u_\varepsilon|^{\frac{5(3+2\alpha)}{3(4+\alpha)}}\right] dt\leq C(T+1)\\
   \end{array}\label{1.1dddgghhjfgbhnjmdfgeddvbnhhjjjmklllhyussddisda}
\end{equation}
and
\begin{equation}
 \begin{array}{ll}
  \disp\int_0^T\int_{\Omega}\left[|u_\varepsilon\cdot\nabla c_\varepsilon|^{\frac{5}{4}} +|c_\varepsilon u_\varepsilon |^{q}\right]dt  \leq C(T+1),\\
   \end{array}\label{1.1dddfgbhnjmdfgeddvbnmklllhyussdddfgggdisda}
\end{equation}
where $q<\frac{10}{3}$.
\end{lemma}
\begin{proof}
From Lemma \ref{lemmaghjffggsjjjjjsddgghhmk4563025xxhjklojjkkk} and
\dref{x1.73142vghf48gg},
in light of the Gagliardo--Nirenberg inequality, we derive that there exist positive constants $C_{1}$, $C_{2}$, $C_{3}$, $C_{4}$, $C_{5}$ and $C_{6}$  such that
\begin{equation}
\begin{array}{rl}
\disp\int_{0}^T\disp\int_{\Omega}  n_{\varepsilon}^{\frac{6+4\alpha}{3}}dt \leq&\disp{C_{1}\int_{0}^T\left(\| \nabla{ n_{\varepsilon}}\|^{2}_{L^{2}(\Omega)}\|{ n_{\varepsilon} }\|^{{\frac{4\alpha}{3}}}_{L^{2\alpha}(\Omega)}+
\|{  n_{\varepsilon} }\|^{\frac{6+4\alpha}{3}}_{L^{2\alpha}(\Omega)}\right)dt}\\
\leq&\disp{C_{2}(T+1)~\mbox{for all}~ T > 0}\\
\end{array}
\label{ddffbnmbnddfgffddfghhggjjkkuuiicz2ddfvgbhh.htt678ddfghhhyuiihjj}
\end{equation}
as well as
\begin{equation}
\begin{array}{rl}
\disp\int_{0}^T\disp\int_{\Omega} |u_{\varepsilon}|^{\frac{10}{3}}dt \leq&\disp{C_{3}\int_{0}^T\left(\| \nabla{ u_{\varepsilon}}\|^{2}_{L^{2}(\Omega)}\|{ u_{\varepsilon}}\|^{{\frac{4}{3}}}_{L^{2}(\Omega)}+
\|{ u_{\varepsilon}}\|^{{\frac{10}{3}}}_{L^{2}(\Omega)}\right)dt}\\
\leq&\disp{C_{4}(T+1)~\mbox{for all}~ T > 0}\\
\end{array}
\label{ddffbnmbnddfgffggssssjjkkuuiicz2dvgbhh.t8ddhhhyuiihjj}
\end{equation}
and
\begin{equation}
\begin{array}{rl}
\disp\int_{0}^T\disp\|u_{\varepsilon}(\cdot,t)\|^{2}_{L^6(\Omega)} dt\leq&\disp{C_{5}\int_{0}^T\| \nabla{ u_{\varepsilon}}(\cdot,t)\|^{2}_{L^{2}(\Omega)}dt}\\
\leq&\disp{C_{6}(T+1)~\mbox{for all}~ T > 0,}\\
\end{array}
\label{ddffbnmbnddfgffggjjkkuuiicz2dvgbhh.t8ddhhhyuiihjj}
\end{equation}
where the last inequality we have used the  embedding $W^{1,2}_{0,\sigma} (\Omega) \hookrightarrow L^6 (\Omega)$ and the Poincar\'{e}
inequality.
Now, in light of \dref{x1.73142vghf48gg}, applying \dref{1.ffggvbbnxxccvvn1}, \dref{ddffbnmbnddfgffddfghhggjjkkuuiicz2ddfvgbhh.htt678ddfghhhyuiihjj}  and Lemma \ref{lemmaghjffggsjjjjjsddgghhmk4563025xxhjklojjkkk}, one can derive from the H\"{o}lder inequality that  for some positive constant  $C_{7}$ 
such that
\begin{equation}
 \begin{array}{ll}
  &\disp\int_0^T\int_{\Omega}\left[|n_\varepsilon F_{\varepsilon}(n_{\varepsilon})S_\varepsilon(x, n_{\varepsilon}, c_{\varepsilon})\nabla c_\varepsilon|^{2}+|u_\varepsilon\cdot\nabla c_\varepsilon|^{\frac{5}{4}}+|n_\varepsilon u_\varepsilon|^{\frac{5(3+2\alpha)}{3(4+\alpha)}}\right]\\
  \leq&\disp C_S\int_0^T\int_{\Omega}n_\varepsilon^2 (1+n_{\varepsilon})^{-2\alpha}|\nabla c_\varepsilon|^{2}\\
  &+\disp\left(\int_0^T\int_{\Omega}|\nabla c_\varepsilon|^{2}\right)^{\frac{5}{8}} \left(\int_0^T\int_{\Omega}|u_\varepsilon|^{\frac{10}{3}}\right)^{\frac{3}{8}}+\left(\int_0^T\int_{\Omega}n_\varepsilon^{\frac{6+4\alpha}{3}}\right)^{\frac{5}{2(4+\alpha)}} \left(\int_0^T\int_{\Omega}|u_\varepsilon|^{\frac{10}{3}}\right)^{\frac{3+2\alpha}{2(4+\alpha)}}\\
    \leq&\disp C_S\int_0^T\int_{\Omega}|\nabla c_\varepsilon|^{2}+\left(\int_0^T\int_{\Omega}|\nabla c_\varepsilon|^{2}\right)^{\frac{5}{8}} \left(\int_0^T\int_{\Omega}|u_\varepsilon|^{\frac{10}{3}}\right)^{\frac{3}{8}}\\
    &+\left(\disp\int_0^T\disp\int_{\Omega}n_\varepsilon^{\frac{6+4\alpha}{3}}\right)^{\frac{5}{2(4+\alpha)}} \left(\disp\int_0^T\disp\int_{\Omega}|u_\varepsilon|^{\frac{10}{3}}\right)^{\frac{3+2\alpha}{2(4+\alpha)}}\\
   \leq& C_{7}(T+1)~~\mbox{for all}~~ T > 0\\
   \end{array}\label{1.1dddgghhjfssdddgbhnjmdfgeddvbnhhjjjmklllhyussddisda}
\end{equation}
by using  $\alpha\geq1$, where $C_S$ is the same as   \dref{x1.73142vghf48gg}.

Next, for any $q<\frac{10}{3}$, by using \dref{3.10gghhjuuloollgghhhyhh}, one can find $C_{8}>0$ independent of $\varepsilon$ such that
\begin{equation}
\|c_{\varepsilon}(\cdot, t)\|_{L^{\frac{10q}{10-3q}}(\Omega)}\leq C_8~~ \mbox{for all}~~ t>0,
\label{3.10gghhjuuloollssddddgghhhyhh}
\end{equation}
so that, combined with \dref{ddffbnmbnddfgffggssssjjkkuuiicz2dvgbhh.t8ddhhhyuiihjj} yields to
\begin{equation}
 \begin{array}{ll}
  &\disp\int_0^T\int_{\Omega}\left[ |c_\varepsilon u_\varepsilon |^{q}\right]\\
  \leq&\disp\left(\int_0^T\int_{\Omega}c_\varepsilon^{\frac{10}{3}}\right)^{\frac{10-3q}{10}} \left(\int_0^T\int_{\Omega}|u_\varepsilon|^{\frac{10}{3}}\right)^{\frac{3q}{10}}\\
   \leq& C_{9}(T+1)~~\mbox{for all}~~ T > 0\\
   \end{array}\label{1.1dddgghhjfssssdddddgbhnjmdfgeddvbnhhjjjmklllhyussddisda}
\end{equation}
with  some positive constant $C_9$ independent of $\varepsilon$.
 Therefore a combination of \dref{ddffbnmbnddfgffddfghhggjjkkuuiicz2ddfvgbhh.htt678ddfghhhyuiihjj}--\dref{1.1dddgghhjfssssdddddgbhnjmdfgeddvbnhhjjjmklllhyussddisda} directly implies this lemma.
\end{proof}

\section{Regularity properties of time derivatives}

 To construct a weak solution of
problem \dref{33dfff4451.1fghyuisda}--\dref{33dfff44sdfff51.1fgsdddfhyuisda} by a limit process from the regularized problems \dref{1.1fghyuisda}, we shall rely
on 
%
%
the following regularity property with respect to the time variable.

\begin{lemma}\label{qqqqlemma45630hhuujjuuyytt}
Let $\alpha\geq1$,
and assume that \dref{dd1.1fghyuisdakkkllljjjkk} and \dref{ccvvx1.731426677gg}
 hold.
Then there exists $C>0$ independent of $\varepsilon$ such that
\begin{equation}
 \begin{array}{ll}
\disp\int_0^T\|\partial_tn_\varepsilon(\cdot,t)\|_{(W^{1,\frac{5}{2}}(\Omega))^*}^{\frac{5}{3}}dt  \leq C(T+1)~\mbox{for all}~ T\in(0,\infty)\\
   \end{array}\label{1.1ddfgeddvbnmklllhyuisda}
\end{equation}
as well as
\begin{equation}
 \begin{array}{ll}
  \disp\int_0^T\|\partial_tc_\varepsilon(\cdot,t)\|_{(W^{1,2}(\Omega))^*}^{2}dt  \leq C(T+1)~\mbox{for all}~ T\in(0,\infty)\\
   \end{array}\label{wwwwwqqqq1.1dddfgbhnjmdfgeddvbnmklllhyussddisda}
\end{equation}
and
\begin{equation}
 \begin{array}{ll}
  \disp\int_0^T\|\partial_tu_\varepsilon(\cdot,t)\|_{(W^{1,5}_{0,\sigma}(\Omega))^*}^{\frac{5}{4}}dt  \leq C(T+1)~\mbox{for all}~ T\in(0,\infty).\\
   \end{array}\label{wwwwwqqqq1.1dddfgkkllbhddffgggnjmdfgeddvbnmklllhyussddisda}
\end{equation}
\end{lemma}
\begin{proof}
Firstly, testing the first equation of \dref{1.1fghyuisda} by certain $\varphi\in C^{\infty}(\bar{\Omega})$, we have
$$
\begin{array}{rl}
&\disp\left|\int_{\Omega}(n_{\varepsilon,t})\varphi\right|\\
 =&\disp{\left|\int_{\Omega}\left[\Delta n_{\varepsilon}-\nabla\cdot(n_{\varepsilon}F_{\varepsilon}(n_{\varepsilon})S_\varepsilon(x, n_{\varepsilon}, c_{\varepsilon})\nabla c_{\varepsilon})-u_{\varepsilon}\cdot\nabla n_{\varepsilon}\right]\varphi\right|}
\\
=&\disp{\left|\int_\Omega \left[-\nabla n_{\varepsilon}\cdot\nabla\varphi+n_{\varepsilon}F_{\varepsilon}(n_{\varepsilon})S_\varepsilon(x, n_{\varepsilon}, c_{\varepsilon})\nabla c_{\varepsilon}\cdot\nabla\varphi+ n_{\varepsilon}u_{\varepsilon}\cdot\nabla  \varphi\right]\right|}\\
\leq&\disp{\left|\int_\Omega \left[\|\nabla n_{\varepsilon}\|_{L^{\frac{5}{3}}(\Omega)}+\|n_{\varepsilon}F_{\varepsilon}(n_{\varepsilon})S_\varepsilon(x, n_{\varepsilon}, c_{\varepsilon})\nabla c_{\varepsilon}\|_{L^{\frac{5}{3}}(\Omega)}+ \|n_{\varepsilon}u_{\varepsilon}\|_{L^{\frac{5}{3}}(\Omega)}\right]\right|\|\varphi\|_{W^{1,\frac{5}{2}}(\Omega)}}
\end{array}
$$
for all $t>0$.

Therefore,
 in view of $\alpha\geq1$, Lemma \ref{lemmaghjffggsjjjjjsddgghhmk4563025xxhjklojjkkk} and Lemma \ref{lemmddaghjsffggggsddgghhmk4563025xxhjklojjkkk},
we deduce from the Young inequality that for some $C_1>0$ and $C_2>0$ such that
\begin{equation}
\begin{array}{rl}
&\disp\int_0^T\|\partial_{t}n_{\varepsilon}(\cdot,t)\|_{(W^{1,\frac{5}{2}}(\Omega))^*}^{\frac{5}{3}}dt\\
\leq&\disp{C_1\left\{\int_0^T\int_{\Omega} |\nabla n_{\varepsilon}|^{2}+\int_0^T\int_{\Omega}|n_\varepsilon F_{\varepsilon}(n_{\varepsilon})S_\varepsilon(x, n_{\varepsilon}, c_{\varepsilon})\nabla c_\varepsilon|^{2}+\int_0^T\int_{\Omega}|n_\varepsilon u_\varepsilon|^{\frac{5}{3}}+T\right\}
}\\
\leq&\disp{C_2(T+1)~~\mbox{for all}~~ T > 0.
}\\
\end{array}
\label{yyygbhncvbmdcxxcdfvgbfvgcz2.5ghju48}
\end{equation}

Likewise, given any $\varphi\in  C^\infty(\bar{\Omega})$, we may test the second equation in \dref{1.1fghyuisda} against $\varphi$ to conclude that
$$
\begin{array}{rl}
\disp\left|\int_{\Omega}\partial_{t}c_{\varepsilon}(\cdot,t)\varphi\right|=&\disp{\left|\int_{\Omega}\left[\Delta c_{\varepsilon}-c_{\varepsilon}+n_{\varepsilon}-u_{\varepsilon}\cdot\nabla c_{\varepsilon}\right]\cdot\varphi\right|}
\\
=&\disp{\left|-\int_\Omega \nabla c_{\varepsilon}\cdot\nabla\varphi-\int_\Omega  c_{\varepsilon}\varphi+\int_\Omega n_{\varepsilon} \varphi+\int_\Omega c_{\varepsilon}u_\varepsilon\cdot\nabla\varphi\right|}\\
\leq&\disp{\left\{\|\nabla c_{\varepsilon}\|_{L^{{2}}(\Omega)}+\|c_{\varepsilon} \|_{L^{2}(\Omega)}+\|n_{\varepsilon} \|_{L^{2}(\Omega)}+\|c_{\varepsilon}u_\varepsilon\|_{L^{2}(\Omega)}\right\}\|\varphi\|_{W^{1,2}(\Omega)}}\\
\end{array}
$$
for all $t>0$.
Thus, from Lemma \ref{lemmddaghjsffggggsddgghhmk4563025xxhjklojjkkk} and Lemma \ref{lemmaghjffggsjjjjjsddgghhmk4563025xxhjklojjkkk} again, 
we invoke the Young inequality again and obtain that there exist positive constants $C_{3}$ and $C_{4}$ such that
$$
\begin{array}{rl}
&\disp\int_0^T\|\partial_{t}c_{\varepsilon}(\cdot,t)\|_{(W^{1,2}(\Omega))^*}^{2}dt\\
\leq&\disp{C_{3}\left(\int_0^T\int_\Omega|\nabla c_{\varepsilon}|^{2}+\int_0^T\int_\Omega n_{\varepsilon}^{2}+\int_0^T\int_\Omega c_\varepsilon^{5}+\int_0^T\int_\Omega |u_\varepsilon|^{\frac{10}{3}}+T\right)}\\
\leq&\disp{C_{4}(T+1)
~~\mbox{for all}~~ T>0.}\\
\end{array}
$$

Finally, for any given  $\varphi\in C^{\infty}_{0,\sigma} (\Omega;\mathbb{R}^3)$, we infer from the third equation in \dref{1.1fghyuisda} that
$$
\begin{array}{rl}
\disp\left|\int_{\Omega}\partial_{t}u_{\varepsilon}(\cdot,t)\varphi\right|=&\disp{\left|-\int_\Omega \nabla u_{\varepsilon}\cdot\nabla\varphi-\kappa\int_\Omega (Y_{\varepsilon}u_{\varepsilon}\otimes u_{\varepsilon})\cdot\nabla\varphi+\int_\Omega n_{\varepsilon}\nabla \phi\cdot\varphi\right|
~~\mbox{for all}~~ t>0.}
\end{array}
$$
Now,   Lemma \ref{lemmaghjffggsjjjjjsddgghhmk4563025xxhjklojjkkk}, Lemma \ref{lemmddaghjsffggggsddgghhmk4563025xxhjklojjkkk}  and \dref{dd1.1fghyuisdakkkllljjjkk}, we also derive   that there exist positive constants
 $C_{5}$ as well as $C_{6}$ and $C_{7}$ such that
$$
\begin{array}{rl}
&\disp\int_0^T\|\partial_{t}u_{\varepsilon}(\cdot,t)\|_{(W^{1,2}_{0,\sigma}(\Omega))^*}^{2}dt\\
\leq&\disp{C_{5}\left(\int_0^T\int_\Omega|\nabla u_{\varepsilon}|^{2}+\int_0^T\int_\Omega |Y_{\varepsilon}u_{\varepsilon}\otimes u_{\varepsilon}|^{2}+\int_0^T\int_\Omega n_\varepsilon^{2}\right)}\\
\leq&\disp{C_{6}\left(\int_0^T\int_\Omega|\nabla u_{\varepsilon}|^{2}+\int_0^T\int_\Omega |Y_{\varepsilon}u_\varepsilon|^{2}+\int_0^T\int_\Omega |u_\varepsilon|^{\frac{10}{3}}+\int_0^T\int_\Omega n_{\varepsilon}^{2}+T\right)}\\
\leq&\disp{C_{7}(T+1)
~~\mbox{for all}~~ T>0,}
\end{array}
$$
which implies \dref{wwwwwqqqq1.1dddfgkkllbhddffgggnjmdfgeddvbnmklllhyussddisda} and whereby the proof is completed. 
\end{proof}

\section{Passing to the Limit: Proof of Theorem  \ref{theorem3}}

In this section we discuss the convergence of approximate solutions and complete the proof of
Theorem \ref{theorem3}. 
%
%
%
%
%
To this end,
based on the above lemmas (see Lemma \ref{lemmaghjffggsjjjjjsddgghhmk4563025xxhjklojjkkk}, Lemma \ref{aasslemmafggg78630jklhhjj} as well as Lemma \ref{lemmddaghjsffggggsddgghhmk4563025xxhjklojjkkk} and Lemma \ref{qqqqlemma45630hhuujjuuyytt}), we can
take the limit of the approximated solution $(n_\varepsilon, c_\varepsilon, u_\varepsilon)$ as $\varepsilon\searrow0$ and show that the limit is a
global weak solution in the sense of the above definition (see Definition \ref{df1}).
%
\begin{lemma}\label{lesssmma45630223}
Let
 \dref{x1.73142vghf48rtgyhu}, \dref{x1.73142vghf48gg},
\dref{dd1.1fghyuisdakkkllljjjkk}, and \dref{ccvvx1.731426677gg}
 hold, and suppose that $\alpha\geq1.$
 There exists $(\varepsilon_j)_{j\in \mathbb{N}}\subset (0, 1)$ such that $\varepsilon_j\searrow 0$ as $j\rightarrow\infty$ and such that as $\varepsilon = \varepsilon_j\searrow 0$
we have
\begin{equation}
n_\varepsilon\rightarrow n ~\mbox{a.e.}~\mbox{in}~\Omega\times(0,\infty)~\mbox{and strongly in}~ L_{loc}^{2}(\bar{\Omega}\times[0,\infty))\label{zjscz2.5297x963ddfgh0ddfggg6662222tt3}
\end{equation}
\begin{equation}
n_\varepsilon\rightharpoonup n ~~\mbox{weakly in}~~ L_{loc}^{\frac{6+4\alpha}{3}}(\bar{\Omega}\times[0,\infty))\label{zjscz2.5297x9ssdd63ddfgh0ddfggg6662222tt3}
\end{equation}
\begin{equation}
\nabla n_\varepsilon\rightharpoonup \nabla n ~\mbox{weakly in}~ L_{loc}^{2}(\bar{\Omega}\times[0,\infty))
\label{zjscz2.5297x963ddfgh0ddgghjjfggg6662222tt3}
\end{equation}
\begin{equation}
c_\varepsilon\rightarrow c ~\mbox{strongly in}~ L^{2}_{loc}(\bar{\Omega}\times[0,\infty))~\mbox{and}~\mbox{a.e.}~\mbox{in}~\Omega\times(0,\infty),
 \label{zjscz2.fgghh5297x963ddfgh0ddfggg6662222tt3}
\end{equation}
\begin{equation}
\nabla c_\varepsilon^{\frac{p}{2}}\rightharpoonup \nabla c^{\frac{p}{2}} ~\mbox{weakly in}~ L^{2}_{loc}([0,\infty);L^2(\Omega))~~~\mbox{for all}~~~p>1,
 \label{zjscz2.fgghhkklll5ssddf297x963ddfgh0ddfggg6662222tt3}
\end{equation}
\begin{equation}
c_\varepsilon\rightharpoonup c ~\mbox{weakly * in}~ L^{\infty}_{loc}([0,\infty);L^p(\Omega))~~~\mbox{for all}~~~p>1,
 \label{zjscz2.fgghhkklll5297x963ddfgh0ddfggg6662222tt3}
\end{equation}
\begin{equation}
\nabla c_\varepsilon\rightarrow \nabla c ~\mbox{a.e.}~\mbox{in}~\Omega\times(0,\infty),
 \label{1.1ddhhyujiiifgghhhge666ccdf2345ddvbnmklllhyuisda}
\end{equation}
\begin{equation}
u_\varepsilon\rightarrow u~\mbox{strongly in}~ L_{loc}^2(\bar{\Omega}\times[0,\infty))~\mbox{and}~\mbox{a.e.}~\mbox{in}~\Omega\times(0,\infty),
 \label{zjscz2.5297x96302222t666t4}
\end{equation}
\begin{equation}
\nabla c_\varepsilon\rightharpoonup \nabla c~\begin{array}{ll}
 \mbox{weakly in}~ L_{loc}^{2}(\bar{\Omega}\times[0,\infty))
   \end{array}\label{1.1ddfgghhhge666ccdf2345ddvbnmklllhyuisda}
\end{equation}
as well as
\begin{equation}
 \nabla u_\varepsilon\rightharpoonup \nabla u ~\mbox{weakly in}~L^{2}_{loc}(\bar{\Omega}\times[0,\infty))
 \label{zjscz2.5297x96366602222tt4455}
\end{equation}
and
\begin{equation}
 u_\varepsilon\rightharpoonup u ~\mbox{weakly in}~L^{\frac{10}{3}}_{loc}(\bar{\Omega}\times[0,\infty))
 \label{zjscz2.5ffgtt297x96302266622tt4}
\end{equation}
 with some triple $(n, c, u)$ that is a global weak solution of \dref{33dfff4451.1fghyuisda}--\dref{33dfff44sdfff51.1fgsdddfhyuisda} in the sense of Definition \ref{df1}.
\end{lemma}
\begin{proof}
 Based on Lemma \ref{aasslemmafggg78630jklhhjj}, Lemma \ref{lemmaghjffggsjjjjjsddgghhmk4563025xxhjklojjkkk}, Lemma \ref{qqqqlemma45630hhuujjuuyytt}, Lemma \ref{lemmddaghjsffggggsddgghhmk4563025xxhjklojjkkk} and the Aubin--Lions lemma (see e.g. \cite{Simon}), we can pick a sequence $(\varepsilon_j)_{j\in \mathbb{N}}\subset (0, 1)$ with $\varepsilon=\varepsilon_j\searrow0$ as $j\rightarrow\infty$  such that \dref{zjscz2.5297x963ddfgh0ddfggg6662222tt3}--\dref{zjscz2.fgghhkklll5297x963ddfgh0ddfggg6662222tt3} and \dref{zjscz2.5297x96302222t666t4}--\dref{zjscz2.5ffgtt297x96302266622tt4} are valid with certain limit
functions $n, c$ and $u$ belonging to the indicated spaces.
Next, let $g_\varepsilon(x, t) := -c_\varepsilon+n_{\varepsilon}-u_{\varepsilon}\cdot\nabla c_{\varepsilon}.$
With this
notation, the   second equation of \dref{1.1fghyuisda} can be rewritten in  component form as
\begin{equation}
c_{\varepsilon t}-\Delta c_{\varepsilon } = g_\varepsilon.
\label{1.1666ddjjkllllfgffgghheccdfddfgg2345ddvbnmklllhyuisda}
\end{equation}

Thus, recalling Lemmas \ref{aasslemmafggg78630jklhhjj}--\ref{lemmddaghjsffggggsddgghhmk4563025xxhjklojjkkk}  and applying the H\"{o}lder inequality, we conclude that, for any  $\varepsilon\in(0,1)$, $ g_\varepsilon$
is bounded in $L^{\frac{5}{4}} (\Omega\times(0, T))$,  and then we may invoke the standard parabolic regularity theory  to \dref{1.1666ddjjkllllfgffgghheccdfddfgg2345ddvbnmklllhyuisda} and  infer that
$(c_{\varepsilon})_{\varepsilon\in(0,1)}$ is bounded in
$L^{\frac{5}{4}} ((0, T); W^{2,\frac{5}{4}}(\Omega))$.
 Here, noting from  \dref{wwwwwqqqq1.1dddfgbhnjmdfgeddvbnmklllhyussddisda} and the Aubin--Lions lemma that $(c_{\varepsilon})_{\varepsilon\in(0,1)}$ is    relative compact in
$L^{\frac{5}{4}} ((0, T); W^{1,\frac{5}{4}}(\Omega))$. We can pick an appropriate subsequence that is
still written as $(\varepsilon_j )_{j\in \mathbb{N}}$ such that $\nabla c_{\varepsilon_j} \rightarrow z_1$
 in $L^{\frac{5}{4}} (\Omega\times(0, T))$ for all $T\in(0, \infty)$ and some
$z_1\in L^{\frac{5}{4}} (\Omega\times(0, T))$ as $j\rightarrow\infty$. Therefore,  by \dref{wwwwwqqqq1.1dddfgbhnjmdfgeddvbnmklllhyussddisda}, we can also derive that $\nabla c_{\varepsilon_j} \rightarrow z_1$ a.e. in $\Omega\times(0, \infty)$
 as $j \rightarrow\infty$.
In view  of \dref{1.1ddfgghhhge666ccdf2345ddvbnmklllhyuisda} and   Egorov's theorem, we conclude  that
$z_1=\nabla c$ and hence \dref{1.1ddhhyujiiifgghhhge666ccdf2345ddvbnmklllhyuisda} holds.


In the following proof,  we shall prove that $(n,c,u)$ is a weak solution of problem \dref{33dfff4451.1fghyuisda}--\dref{33dfff44sdfff51.1fgsdddfhyuisda} in the sense of Definition \ref{df1},
 $n,c$ and
 $u$
satisfy (1.7), (3.9) and (3.10), (3.11), (3.12).
Therefore,
with the help of  \dref{zjscz2.5297x963ddfgh0ddfggg6662222tt3}--\dref{zjscz2.fgghh5297x963ddfgh0ddfggg6662222tt3} and \dref{zjscz2.5297x96302222t666t4}--\dref{zjscz2.5297x96366602222tt4455}, we can derive  \dref{dffff1.1fghyuisdakkklll}.
Now, by the nonnegativity of $n_\varepsilon$ and $c_\varepsilon$, we obtain  $n \geq 0$ and $c\geq 0$. Next, from
\dref{zjscz2.5297x96366602222tt4455} and $\nabla\cdot u_{\varepsilon} = 0$, we conclude that
$\nabla\cdot u = 0$ a.e. in $\Omega\times (0, \infty)$.
However, in view of \dref{1.ffggvbbnxxccvvn1}, \dref{x1.73142vghf48rtgyhu},
 as well as  \dref{1.1dddgghhjfgbhnjmdfgeddvbnhhjjjmklllhyussddisda} and Lemma \ref{lemmddaghjsffggggsddgghhmk4563025xxhjklojjkkk}, we conclude that
\begin{equation}n_\varepsilon F_{\varepsilon}(n_{\varepsilon})S_\varepsilon(x, n_{\varepsilon}, c_{\varepsilon})\nabla c_\varepsilon\rightharpoonup z_2
~\mbox{in}~ L^{2}(\Omega\times(0,T))~\mbox{as}~\varepsilon = \varepsilon_j\searrow 0~\mbox{for each}~ T\in(0,\infty).
\label{1.1ddddfddffttyygghhyujiiifgghhhgffgge6bhhjh66ccdf2345ddvbnmklllhyuisda}
\end{equation}
However, it follows from \dref{x1.73142vghf48rtgyhu}, \dref{3.10gghhjuuloollyuigghhhyy}, \dref{1.1dddgghhjfgbhnjmdfgeddvbnhhjjjmklllhyussddisda}, \dref{zjscz2.5297x963ddfgh0ddfggg6662222tt3}, \dref{zjscz2.fgghh5297x963ddfgh0ddfggg6662222tt3} and \dref{1.1ddhhyujiiifgghhhge666ccdf2345ddvbnmklllhyuisda} that
\begin{equation}n_\varepsilon F_{\varepsilon}(n_{\varepsilon})S_\varepsilon(x, n_{\varepsilon}, c_{\varepsilon})\nabla c_\varepsilon\rightarrow nS(x, n, c)\nabla c~\mbox{a.e.}~\mbox{in}~\Omega\times(0,\infty)~\mbox{as}~\varepsilon = \varepsilon_j\searrow 0.
\label{1.1ddddfddfftffghhhtyygghhyujiiifgghhhgffgge6bhhjh66ccdf2345ddvbnmklllhyuisda}
\end{equation}
Again by  Egorov's theorem, we gain $z_2=nS(x, n, c)\nabla c$, and therefore   \dref{1.1ddddfddffttyygghhyujiiifgghhhgffgge6bhhjh66ccdf2345ddvbnmklllhyuisda}
can be rewritten as
\begin{equation}n_\varepsilon F_{\varepsilon}(n_{\varepsilon})S_\varepsilon(x, n_{\varepsilon}, c_{\varepsilon})\nabla c_\varepsilon\rightharpoonup nS(x, n, c)\nabla c
~\mbox{weakly in}~ L^{2}(\Omega\times(0,T))~\mbox{as}~\varepsilon = \varepsilon_j\searrow 0
\label{1.1ddddfddffttyygghhyujiiffghhhifgghhhgffgge6bhhjh66ccdf2345ddvbnmklllhyuisda}
\end{equation}
for each $T\in(0,\infty)$,
which  implies  that the integrability of $nS(x, n, c)\nabla c$ in \dref{726291hh} as well.
   Thereupon, recalling \dref{zjscz2.5297x96302222t666t4}, \dref{zjscz2.5297x963ddfgh0ddfggg6662222tt3} and \dref{1.1dddgghhjfgbhnjmdfgeddvbnhhjjjmklllhyussddisda}, we infer that, for each $T\in(0, \infty),$
  \begin{equation}
n_\varepsilon u_{\varepsilon}\rightharpoonup z_3 ~\mbox{weakly in}~ L^{{\frac{5(3+2\alpha)}{3(4+\alpha)}}}(\Omega\times(0,T)),
   ~\mbox{as}~\varepsilon = \varepsilon_j\searrow 0,
   \label{zjscz2.529ffghhh7x963ddfgh0ddfggg6662222tt3}
\end{equation}
which together with \dref{zjscz2.5297x963ddfgh0ddfggg6662222tt3}  and \dref{zjscz2.5297x96302222t666t4} implies
 \begin{equation}n_\varepsilon u_\varepsilon\rightarrow nu~\mbox{a.e.}~\mbox{in}~\Omega\times(0,\infty)~\mbox{as}~\varepsilon = \varepsilon_j\searrow 0.
\label{1.1ddddfddfftffghhhtyygghhyujiiifgghhhgffgge6bhhffgggjh66ccdf2345ddvbnmklllhyuisda}
\end{equation}
 \dref{zjscz2.529ffghhh7x963ddfgh0ddfggg6662222tt3} along with  \dref{1.1ddddfddfftffghhhtyygghhyujiiifgghhhgffgge6bhhffgggjh66ccdf2345ddvbnmklllhyuisda} and  Egorov's theorem guarantees that $z_3=nu$, whereupon we derive from \dref{zjscz2.529ffghhh7x963ddfgh0ddfggg6662222tt3} that
 \begin{equation}
n_\varepsilon u_{\varepsilon}\rightharpoonup nu ~~\mbox{weakly in}~~ L^{{\frac{5(3+2\alpha)}{3(4+\alpha)}}}(\Omega\times(0,T))
   ~~\mbox{as}~~\varepsilon = \varepsilon_j\searrow 0,\label{zjscz2.529ffghhh7x963djkkkkdfgh0ddfggg6662222tt3}
\end{equation}
   for each $T\in(0, \infty)$.

As a straightforward consequence of \dref{zjscz2.fgghh5297x963ddfgh0ddfggg6662222tt3} and \dref{zjscz2.5297x96302222t666t4}, it holds that
\begin{equation}
c_\varepsilon u_\varepsilon\rightarrow cu ~\mbox{ in}~ L^{1}_{loc}(\bar{\Omega}\times(0,\infty))~\mbox{as}~\varepsilon=\varepsilon_j\searrow0.
 \label{zxxcvvfgggjscddfffcvvfggz2.5297x96302222tt4}
\end{equation}
Thus, the integrability of $nu$ and $cu$ in \dref{726291hh} is verified by \dref{zjscz2.fgghh5297x963ddfgh0ddfggg6662222tt3}
 and \dref{zjscz2.5297x96302222t666t4}.

 Moreover, a combination of \dref{zjscz2.5297x96302222t666t4} and arguments in the proof of
(6.26) of \cite{Kegssddsddfff00} yields that
\begin{equation}
\begin{array}{rl}
Y_{\varepsilon}u_{\varepsilon}\otimes u_{\varepsilon}\rightarrow u \otimes u ~\mbox{in}~L^1_{loc}(\bar{\Omega}\times[0,\infty))~\mbox{as}~\varepsilon=\varepsilon_j\searrow0.
\end{array}
\label{ggjjssdffzccfccvvfgghjjjvvvvgccvvvgjscz2.5297x963ccvbb111kkuu}
\end{equation}
Therefore, the integrability of $nS(x,n,c)\nabla c$, $nu$, $cu$, and $u\otimes u$ in \dref{726291hh} is verified
by \dref{1.1ddddfddffttyygghhyujiiffghhhifgghhhgffgge6bhhjh66ccdf2345ddvbnmklllhyuisda}, \dref{zjscz2.529ffghhh7x963djkkkkdfgh0ddfggg6662222tt3}, \dref{zxxcvvfgggjscddfffcvvfggz2.5297x96302222tt4} and \dref{ggjjssdffzccfccvvfgghjjjvvvvgccvvvgjscz2.5297x963ccvbb111kkuu}.
Finally, according to \dref{zjscz2.5297x963ddfgh0ddfggg6662222tt3}--\dref{zjscz2.fgghh5297x963ddfgh0ddfggg6662222tt3},
\dref{zjscz2.5297x96302222t666t4}--\dref{zjscz2.5297x96366602222tt4455},
\dref{1.1ddddfddffttyygghhyujiiffghhhifgghhhgffgge6bhhjh66ccdf2345ddvbnmklllhyuisda}, \dref{zjscz2.529ffghhh7x963djkkkkdfgh0ddfggg6662222tt3},  \dref{zxxcvvfgggjscddfffcvvfggz2.5297x96302222tt4} and \dref{ggjjssdffzccfccvvfgghjjjvvvvgccvvvgjscz2.5297x963ccvbb111kkuu}, we may pass to the limit in
the respective weak formulations associated with the regularized system \dref{1.1fghyuisda} and obtain the integral
identities \dref{eqx45xx12112ccgghh}--\dref{eqx45xx12112ccgghhjjgghh}.
\end{proof}
Thereby our main result on global existence has actually been established already:

{\bf Proof of Theorem \ref{theorem3}.} Theorem \ref{theorem3} is a direct result of Lemma \ref{lesssmma45630223}.

\section{Asymptotic behavior}

As the existence part of the main theorem is covered by the previous section, we will now turn
our attention to verifying the desired convergence result.
In the following, we will  state some auxiliary lemmas, which will be frequently used in the sequel.
%
To this end, we first briefly collect some known facts concerning
the Stokes operator from 
\cite{Friedmanddff,Winkler11215}.
\begin{lemma}\label{llssdrffmmgghhjjjnnccvvccvvkkkkgghhkkllvvlemma45630}(Theorem  1 and Theorem  2 of \cite{Fujiwara66612186})
The Helmholtz projection $\mathcal{P}$ defines a bounded linear operator $\mathcal{P}: L^p(\Omega,\mathbb{R}^3)\rightarrow L^p_\sigma (\Omega)$; in
particular, for any $p\in (1,\infty)$, there is $k_0(p) > 0$ such that
\begin{equation}
\|\mathcal{P} w\|_{L^p(\Omega)} \leq k_0(p)\|w\|_{L^{p}(\Omega)}
\label{1.1ddfghssdddssggyssddduisda}
\end{equation}
for every $w\in L^p (\Omega).$
%
\end{lemma}
\begin{lemma}(\cite{Horstmann791,Winkler792,Zhengddfggghjjkk1})\label{llssdrffmmggnnccvvccvvkkkkgghhkkllvvlemma45630}
The Stokes operator $A$ generates the analytic semigroup $(e^{-tA} )_{t>0}$ in $L^2_{\sigma}(\Omega)$.
Its spectrum
satisfies $\lambda_1 := \inf Re \sigma(A) > 0$ and we fix $\mu\in (0,\lambda_1)$. For any such $\mu$, the following holds:

(i) For any $p\in (1,\infty)$ and $\gamma \geq0$, there is $\kappa_1 (p,\gamma) > 0$ such that
\begin{equation}
\|A^\gamma e^{-tA}\varphi\|_{L^p(\Omega)} \leq \kappa_1(p,\gamma)t^{-\gamma}e^{-\mu t}\|\varphi\|_{L^p(\Omega)} ~~~\mbox{for all}~~~ t > 0~~~\mbox{and any}~~~\varphi\in L^{p}_\sigma(\Omega).
\label{1.1ddfghssdddyddffssddduisda}
\end{equation}
(ii) For $p,q$ satisfying $1 < p \leq  q < \infty$, there exists $\kappa_2 (p,q) > 0$ such that
\begin{equation}
\| e^{-tA}\varphi\|_{L^q(\Omega)} \leq \kappa_2(p,q)t^{-\frac{3}{2}(\frac{1}{p}-\frac{1}{q})}e^{-\mu t}\|\varphi\|_{L^{p}(\Omega)} ~~~\mbox{for all}~~~ t > 0~~~\mbox{and any}~~~\varphi\in L^{p}_\sigma(\Omega).
\label{1.1ddfghssdddyssddduisda}
\end{equation}
(iii) For $p,q$ satisfying $1 < p \leq  q < \infty$, there exists a positive constant  $\kappa_3 (p,q) $ such that
\begin{equation}
\|\nabla e^{-tA}\varphi\|_{L^q(\Omega)} \leq \kappa_3(p,q)t^{-\frac{1}{2}-\frac{3}{2}(\frac{1}{p}-\frac{1}{q})}e^{-\mu t}\|\varphi\|_{L^p(\Omega)} ~~~\mbox{for all}~~~ t > 0~~~\mbox{and any}~~~\varphi\in L^{p}_\sigma(\Omega).
\label{1.1ddddfghhfghssdddyssddduisda}
\end{equation}

(iv) If $\gamma\geq 0$ and $1 < q < p < \infty$ satisfy $2\gamma-\frac{3}{q}\geq 1-\frac{3}{p}$, then there is $k_4(\gamma,p,q)>0$ such that for all
 $\varphi\in D(A^\gamma)$
\begin{equation}
\|\varphi\|_{W^{1,p}(\Omega)} \leq \kappa_4(\gamma,p,q)\|A^\gamma\varphi\|_{L^q(\Omega)}.
\label{1.1ddddfghhsssfghssdddyssddduisda}
\end{equation}
%
\end{lemma}

\begin{lemma}\label{fhhghfbglemma4563025xxhjklojjkkkgyhuissddff} (Lemma 4.6 of \cite{EspejojjEspejojainidd793})
Let $\lambda > 0, C > 0$, and suppose that $y\in C^1 ([0,\infty))$ and
$h\in  C^0 ([0,\infty))$ are nonnegative functions satisfying $y'(t)+\lambda y(t)\leq h(t)$ for some $\lambda > 0$ and all
$t > 0$. Then if
$\int_0^\infty h(s)ds \leq C,$ we have $\lim_{t\rightarrow\infty}y(t)=0$.
\begin{lemma}\label{sedddlemmaddffffdfffgg4sssdddd5630}
Suppose that
\begin{equation}\label{x1.73142vgddddssdddhfjjk48}C_S<2\sqrt{C_N},
\end{equation}
where $C_N$ is the best  Poincar\'{e} constant and $C_S$ is given by  \dref{x1.73142vghf48gg}.
Then there exists  $B > 0$ 
such that for each $\varepsilon\in (0, 1)$ we have
\begin{equation}
\begin{array}{rl}
&\disp{\frac{B}{2}\frac{d}{dt}\|n_\varepsilon(\cdot,t)-\bar{n}_0\|^{{2}}_{L^{{2}}(\Omega)}+
\frac{1}{2}\frac{d}{dt}\|c_\varepsilon(\cdot,t)-\bar{n}_0\|^{{2}}_{L^{{2}}(\Omega)}}\\
&\disp{+(\frac{BC_N}{2}-\frac{1}{4})
\int_{\Omega}| n_\varepsilon-\bar{n}|^2+(1-\frac{BC_S^2}{2})\int_{\Omega}|\nabla c_\varepsilon|^2}\\
\leq&\disp{0~~\mbox{for all}~~ t>0, }\\
\end{array}
\label{wwwwwcz2.511ssssdsdddffssddsdsssdddfggsddffggg4ddfggg114}
\end{equation}
where  \begin{equation}
\bar{n}_0=\frac{1}{|\Omega|}\int_{\Omega}n_{0}.
\label{1111hhxxcdfvhhhvsddfffgssjjdfffsfffsddcsssz2.5}
\end{equation}
\end{lemma}
\begin{proof}
 Firstly, we use the first equation in \dref{33dfff4451.1fghyuisda}--\dref{33dfff44sdfff51.1fgsdddfhyuisda} and the
fact that $\nabla\cdot u_{\varepsilon}=0$  to compute
\begin{equation}
\begin{array}{rl}
&\disp{\frac{1}{2}\frac{d}{dt}\|n_\varepsilon(\cdot,t)-\bar{n}_\varepsilon\|^{{2}}_{L^{{2}}(\Omega)}}
\\
=&\disp{
\int_{\Omega}(n_\varepsilon-\bar{n}_\varepsilon)[\Delta n_\varepsilon-u_\varepsilon\cdot\nabla n_\varepsilon-\nabla\cdot(n_{\varepsilon}F_{\varepsilon}(n_{\varepsilon})S_\varepsilon(x, n_{\varepsilon}, c_{\varepsilon})\nabla c_{\varepsilon})]}\\
\leq&\disp{-
\int_{\Omega}|\nabla n_\varepsilon|^2+\int_{\Omega}|n_{\varepsilon}||S_\varepsilon(x, n_{\varepsilon}, c_{\varepsilon})||\nabla n_{\varepsilon}||\nabla c_{\varepsilon}|}\\
\leq&\disp{-\frac{1}{2}
\int_{\Omega}|\nabla n_\varepsilon|^2+\frac{C_S^2}{2}\int_{\Omega}\frac{n_{\varepsilon}^2}{(1 + n_{\varepsilon})^{2\alpha}}|\nabla c_{\varepsilon}|^2}\\
\leq&\disp{-\frac{1}{2}
\int_{\Omega}|\nabla n_\varepsilon|^2+\frac{C_S^2}{2}\int_{\Omega}|\nabla c_{\varepsilon}|^2~~\mbox{for all}~~ t>0 }\\
\end{array}
\label{wwwwwcz2.511ssssdssddfggsddffggg4ddfggg114}
\end{equation}
by using \dref{1.ffggvbbnxxccvvn1} as well as \dref{x1.73142vghf48gg} and $\alpha\geq1$, where
\begin{equation}
\bar{n}_\varepsilon=\frac{1}{|\Omega|}\int_{\Omega}n_{\varepsilon}
\label{1111hhxxcdfvhhhvsssdddddfffgssjjdfffsfffsddcsssz2.5}
\end{equation}
and $C_S$ is the same as \dref{x1.73142vghf48gg}.
We  note from the Poincar\'{e} inequality that there is $C_N> 0$ such that
\begin{equation}
\| \varphi-\frac{1}{|\Omega|}\int_{\Omega}\varphi\|_{L^2(\Omega)}^2 \leq C_N\int_\Omega|\nabla\varphi|^2 ~~~\mbox{for all}~~~~\varphi\in W^{1,2}(\Omega).
\label{1.1ddfghssdddsssssyddffssddduisda}
\end{equation}
This combined with \dref{wwwwwcz2.511ssssdssddfggsddffggg4ddfggg114}  yields to
\begin{equation}
\begin{array}{rl}
\disp{\frac{1}{2}\frac{d}{dt}\|n_\varepsilon(\cdot,t)-\bar{n}_\varepsilon\|^{{2}}_{L^{{2}}(\Omega)}}\leq&\disp{-\frac{C_N}{2}
\int_{\Omega}| n_\varepsilon-\bar{n}_\varepsilon|^2+\frac{C_S^2}{2}\int_{\Omega}|\nabla c_{\varepsilon}|^2~~\mbox{for all}~~ t>0, }\\
\end{array}
\label{ssdfffwwwwwcz2.511ssssdssdsssdddfggsddffggg4ddfggg114}
\end{equation}
which together with \dref{ddfgczhhhh2.5ghju48cfg924ghyuji} implies
\begin{equation}
\begin{array}{rl}
\disp{\frac{1}{2}\frac{d}{dt}\|n_\varepsilon(\cdot,t)-\bar{n}_0\|^{{2}}_{L^{{2}}(\Omega)}}\leq&\disp{-\frac{C_N}{2}
\int_{\Omega}| n_\varepsilon-\bar{n}_0|^2+\frac{C_S^2}{2}\int_{\Omega}|\nabla c_{\varepsilon}|^2~~\mbox{for all}~~ t>0. }\\
\end{array}
\label{wwwwwcz2.511ssssdssdsssdddfggsddffggg4ddfggg114}
\end{equation}

Next, by means of the testing procedure, 
 we may derive from  the Young inequality that 
\begin{equation}
\begin{array}{rl}
&\disp{\frac{1}{2}\frac{d}{dt}\|c_\varepsilon(\cdot,t)-\bar{n}_0\|^{{2}}_{L^{{2}}(\Omega)}}
\\
=&\disp{
\int_{\Omega}(c_\varepsilon-\bar{n}_0)[\Delta c_\varepsilon-u_\varepsilon\cdot\nabla c_\varepsilon-(c_\varepsilon-\bar{n}_0)+(n_\varepsilon-\bar{n}_0)]}\\
=&\disp{
\int_{\Omega}(c_\varepsilon-\bar{n}_0)(\Delta c_\varepsilon-u_\varepsilon\cdot\nabla c_\varepsilon)-\int_{\Omega}(c_\varepsilon-\bar{n}_0)^2+\int_{\Omega}(c_\varepsilon-\bar{n}_0)(n_\varepsilon-\bar{n}_0)}\\
\leq&\disp{-
\int_{\Omega}|\nabla c_\varepsilon|^2+\frac{1}{4}\int_{\Omega}(n_\varepsilon-\bar{n}_0)^2~~\mbox{for all}~~ t>0, }\\
\end{array}
\label{wwwwwcz2.511ssssssdfggsddffggg4ddfggg114}
\end{equation}
where we have used the fact that $\nabla\cdot u_\varepsilon = 0$ and $u_\varepsilon |_{\partial\Omega} = 0$.

In view of \dref{x1.73142vgddddssdddhfjjk48}, we can choose $B > 0$ such that
\begin{equation}
1-\frac{B}{2}C_S^2 >0
\label{1.1ddfghssdddsssssyddffssdddfffguisda}
\end{equation}
and
\begin{equation}
\frac{B}{2}C_{N}-\frac{1}{4}>0.
\label{1.1ddfghssdddsssssydssdffssddduisda}
\end{equation}
Collecting \dref{wwwwwcz2.511ssssdssdsssdddfggsddffggg4ddfggg114}--\dref{1.1ddfghssdddsssssydssdffssddduisda}, we thus infer that
\begin{equation}
\begin{array}{rl}
&\disp{\frac{B}{2}\frac{d}{dt}\|n_\varepsilon(\cdot,t)-\bar{n}_0\|^{{2}}_{L^{{2}}(\Omega)}+
\frac{1}{2}\frac{d}{dt}\|c_\varepsilon(\cdot,t)-\bar{n}_0\|^{{2}}_{L^{{2}}(\Omega)}}\\
&\disp{+(\frac{BC_N}{2}-\frac{1}{4})
\int_{\Omega}| n_\varepsilon-\bar{n}_0|^2+(1-\frac{BC_S^2}{2})\int_{\Omega}|\nabla c_\varepsilon|^2}\\
\leq&\disp{0~~\mbox{for all}~~ t>0, }\\
\end{array}
\label{wwwwwcz2.511ssssdsssddsdsssdddfggsddffggg4ddfggg114}
\end{equation}
which completes the proof.
\end{proof}
\begin{lemma}\label{11aaalemdfghkkmaddffffdfffgg4sssdddd5630}
Under the assumptions of Lemma \ref{sedddlemmaddffffdfffgg4sssdddd5630}, then for any $t>0,$ there exists $\rho_{*,1} > 0$ 
 such that for all $\varepsilon\in (0, 1)$,
\begin{equation}
\begin{array}{rl}
&\disp{\|n_\varepsilon(\cdot,t)-\bar{n}_0\|^{{2}}_{L^{{2}}(\Omega)}}+
\|c_\varepsilon(\cdot,t)-\bar{n}_0\|^{{2}}_{L^{{2}}(\Omega)}\\
\leq&\disp{e^{-\rho_{*,1}  t}[\disp{\|n_0(\cdot,t)-\bar{n}_0\|^{{2}}_{L^{{2}}(\Omega)}}+
\|c_0(\cdot,t)-\bar{n}_0\|^{{2}}_{L^{{2}}(\Omega)}], }\\
\end{array}
\label{wwwwwcz2.511ssssdsdddffsssdfffsddsdsssdddfggsddffggg4ddfggg114}
\end{equation}
and there exists $C_{*,1} > 0$ such that
\begin{equation}
\int_0^\infty\int_{\Omega}|\nabla c_{\varepsilon}|^2+\int_0^\infty\int_{\Omega}| n_\varepsilon-\bar{n}_0|^2+\int_0^\infty\int_{\Omega}|\nabla n_\varepsilon|^2 \leq C_{*,1} ~~~\mbox{for all}~~~~\varepsilon\in(0,1),
\label{1.1ddfghssdddssssssddsyddsssffssddduisda}
\end{equation}
where $\bar{n}_0$ is given by \dref{1111hhxxcdfvhhhvsddfffgssjjdfffsfffsddcsssz2.5}.
\end{lemma}
\begin{proof}
Let $y(t)={\disp\frac{B}{2}\disp\frac{d}{dt}\|n_\varepsilon(\cdot,t)-\bar{n}_0\|^{{2}}_{L^{{2}}(\Omega)}}+
\disp\frac{1}{2}\disp\frac{d}{dt}\|c_\varepsilon(\cdot,t)-\bar{n}_0\|^{{2}}_{L^{{2}}(\Omega)}.$  Then by \dref{wwwwwcz2.511ssssdsdddffssddsdsssdddfggsddffggg4ddfggg114}, one can conclude that there exists  a positive constant $\kappa_{1,*} $ (e.g. $\kappa_{1,*} \leq\min\{\frac{2(\frac{BC_N}{2}-\frac{1}{4})}{B},2C_N(1-\frac{BC_S^2}{2})\}$) such that
$$y(t)+\kappa_{1,*} y(t)\leq0,$$
so that, integrating the above inequality and \dref{wwwwwcz2.511ssssdsdddffssddsdsssdddfggsddffggg4ddfggg114}  in time, we can get \dref{wwwwwcz2.511ssssdsdddffsssdfffsddsdsssdddfggsddffggg4ddfggg114} and \dref{1.1ddfghssdddssssssddsyddsssffssddduisda} by using \dref{wwwwwcz2.511ssssdssddfggsddffggg4ddfggg114}.
\end{proof}

%
In view of the interpolation, the estimates from Lemma  \ref{11aaalemdfghkkmaddffffdfffgg4sssdddd5630} yield  the following Lemma.

\begin{proposition}\label{11aaaerrrlemdfghkkmaddffffdfffgg4sssdddd5630}
Under the assumptions of Lemma \ref{sedddlemmaddffffdfffgg4sssdddd5630}, 
there exists $C_{*,2}> 0$ 
 such that for all $\varepsilon\in (0, 1)$,
\begin{equation}
\int_0^\infty\int_{\Omega}n_{\varepsilon}^{\frac{10}{3}}\leq C_{*,2} ~~~\mbox{for all}~~~~\varepsilon\in(0,1).
\label{1.1ddfghssdddssssssddsyddsssffssddduisda}
\end{equation}
\end{proposition}
\begin{proof}
Firstly, in view of \dref{wwwwwcz2.511ssssdsdddffsssdfffsddsdsssdddfggsddffggg4ddfggg114}, there exists a positive constant $\kappa_{1,**}$ such that
$$\|{ n_{\varepsilon}}(\cdot,t)\|_{L^{2}(\Omega)}\leq \kappa_{1,**}~~\mbox{for all}~~t>0.$$
Therefore,
in light of the Gagliardo--Nirenberg inequality, we derive from Lemma  \ref{11aaalemdfghkkmaddffffdfffgg4sssdddd5630} that there exist positive constants $\kappa_{2,**}$ and $\kappa_{3,**}$  such that
\begin{equation}
\begin{array}{rl}
\disp\int_{0}^\infty\disp\int_{\Omega} n_{\varepsilon}^{\frac{10}{3}}dt \leq&\disp{\kappa_{2,**}\int_{0}^\infty\left(\| \nabla{ n_{\varepsilon}}\|^{2}_{L^{2}(\Omega)}\|{ n_{\varepsilon}}\|^{{\frac{4}{3}}}_{L^{2}(\Omega)}+
\|{ n_{\varepsilon}}\|^{{\frac{10}{3}}}_{L^{2}(\Omega)}\right)dt}\\
\leq&\disp{\kappa_{3,**}.}\\
\end{array}
\label{ddffbnmbnddfgffssdddg56777gssssjjkkuuiicz2dvgbhh.t8ddhhhyuiihjj}
\end{equation}
\end{proof}

We are now prepared to prove the claimed asymptotic behavior of $u_\varepsilon$.
In fact,  exploiting the basic identity \dref{wwwwwcz2.511ssssdsdddffsssdfffsddsdsssdddfggsddffggg4ddfggg114}, this time in a more elaborate
manner, yields the  decay of $u_\varepsilon$  with respect to the norm in $L^2 (\Omega)$ in the first instance.

\begin{lemma}\label{11aaalemmaddffffdsddfffffgg4sssdddd5630}
Under the assumptions of Lemma \ref{sedddlemmaddffffdfffgg4sssdddd5630},  for any $\eta > 0$,  there are $T_{*,3} > 0,\rho_{*,3}$ and $C_{*,3}>0$ such that for any $t > T_{*,3} $
such that whenever $\varepsilon\in (0, 1)$, we have
\begin{equation}\|u_\varepsilon(\cdot,t)\|_{L^2(\Omega)}\leq C_{*,3}e^{-\rho_{*,3} t},
\label{11111hhxxcdfvhhhvssssssssscccjjghjjsdgggddddgdddfffddffssddcssdz2.5}
\end{equation}
\begin{equation}\int_{t}^{t+1}\|\nabla u_\varepsilon(\cdot,s)\|_{L^2(\Omega)}^2ds<\eta
\label{11111hhxxcdfvhhhvsssssddfffssddffsscccjjghjjsdggggdddfffddffssddcssdz2.5}
\end{equation}
as well as
\begin{equation}\int_{t}^{t+1}\|u_\varepsilon(\cdot,s)\|_{L^q(\Omega)}^2ds<\eta
\label{11111hhxxcdfvhhhvsssddffssddfffssddffsscccjjghjjsdggggdddfffddffssddcssdz2.5}
\end{equation}
and
\begin{equation}\int_{t}^{t+1}\|u_\varepsilon(\cdot,s)\|_{L^q(\Omega)}ds<\eta
\label{11111hhxxcdfvhhhvddfffsssddffssddfffssddffsscccjjghjjsdggggdddfffddffssddcssdz2.5}
\end{equation}
for any $q\in[1,6)$.
\end{lemma}
\begin{proof}
 From straightforward calculations,  while relying on \dref{dd1.1fghyuisdakkkllljjjkk}, we derive  the third  equation in \dref{1.1fghyuisda} that 
\begin{equation}
\begin{array}{rl}
&\disp{\frac{1}{2}\frac{d}{dt}\|u_{\varepsilon}(\cdot,t)\|^{{2}}_{L^{{2}}(\Omega)}}
\\
=&\disp{-
\int_{\Omega}|\nabla u_{\varepsilon}|^2+\int_{\Omega}n_{\varepsilon}\nabla\phi\cdot u_{\varepsilon}-\int_{\Omega}\nabla P_{\varepsilon}\cdot u_{\varepsilon}}\\
=&\disp{-
\int_{\Omega}|\nabla u_{\varepsilon}|^2+\int_{\Omega}(n_{\varepsilon}-\bar{n}_0)\nabla\phi\cdot u_{\varepsilon}}\\
\leq&\disp{-
\int_{\Omega}|\nabla u_{\varepsilon}|^2+\|\nabla\phi\|_{L^\infty(\Omega)}\left(\int_{\Omega}|n_{\varepsilon}-\bar{n}_0|^2\right)^{\frac{1}{2}}\left(\int_{\Omega} |u_{\varepsilon}|^2\right)^{\frac{1}{2}}~~\mbox{for all}~~t>0,}\\
\end{array}
\label{cddddz2.51kkk1ssssdfssddjjkkkggsddffggg4ddfggg114}
\end{equation}
where we have used the fact that  $\nabla\cdot u_{\varepsilon} = 0$ and $u_{\varepsilon} |_{\partial\Omega} = 0$.
Due to the Poincar\'{e} inequality again, there exists a positive constant $\eta_N$ such that
$$ \eta_N\int_{\Omega} |u_{\varepsilon}|^2\leq \int_{\Omega}|\nabla u_{\varepsilon}|^2,$$
therefore, collecting \dref{wwwwwcz2.511ssssdsdddffsssdfffsddsdsssdddfggsddffggg4ddfggg114} and \dref{cddddz2.51kkk1ssssdfssddjjkkkggsddffggg4ddfggg114}, we derive that for some positive constants $\kappa_{1,***}$ and $\varrho_{1,***}$
\begin{equation}
\begin{array}{rl}
\disp\disp \int_{\Omega}|u_{\varepsilon}(x,t)|^2dx&\leq\disp{\kappa_{1,***}e^{-\varrho_{1,***}t}~~~\mbox{for all}~~t>0}\\
\end{array}
\label{hhxxcdfvhhhvsssssssssjjdfffddffssddcssdz2.5}
\end{equation}
and
\begin{equation}\int_{t}^{t+1}\|\nabla u_{\varepsilon}(\cdot,s)\|_{L^2(\Omega)}^2ds\rightarrow 0 ~~~\mbox{as}~~~t\rightarrow\infty
\label{11111hhxxcdfvhhhvddddsssssddfffssddffsscccjjghjjsdggggdddfffddffssddcssdz2.5}
\end{equation}
and thereby proves \dref{11111hhxxcdfvhhhvssssssssscccjjghjjsdgggddddgdddfffddffssddcssdz2.5}--\dref{11111hhxxcdfvhhhvsssssddfffssddffsscccjjghjjsdggggdddfffddffssddcssdz2.5}.
Finally, we make use of the embedding
$W^{1,2}(\Omega)\hookrightarrow L^q (\Omega)$ (for any $q\in[1,6)$) and the
Young inequality  to find 
 that \dref{11111hhxxcdfvhhhvsssddffssddfffssddffsscccjjghjjsdggggdddfffddffssddcssdz2.5} and \dref{11111hhxxcdfvhhhvddfffsssddffssddfffssddffsscccjjghjjsdggggdddfffddffssddcssdz2.5} hold.
\end{proof}
With the decay property of $n_\varepsilon(\cdot,t)-\bar{n}_0$ (see {Lemma} \ref{11aaalemdfghkkmaddffffdfffgg4sssdddd5630})  at hand, applying Lemma \ref{11aaalemmaddffffdsddfffffgg4sssdddd5630}, we can now make sure that \dref{wwwwwcz2.511ssssdsdddffsssdfffsddsdsssdddfggsddffggg4ddfggg114} implies the following 
%
a certain eventual regularity and
decay of $u_\varepsilon$ in $L^p(\Omega)$ with some $p\geq6$.
\begin{lemma}\label{aaalemmaddffffsddddfffgg4sssdddd5630}
Under the assumptions of Lemma \ref{sedddlemmaddffffdfffgg4sssdddd5630},
 then for any $p \in[6,\infty)$, 
 and $\varepsilon\in(0,1 )$, the solution of
\dref{1.1fghyuisda} satisfies
 \begin{equation}\|u_{\varepsilon}\|_{L^\infty((t,\infty);L^p(\Omega))}\rightarrow0~~~\mbox{as}~~~ t\rightarrow\infty.
\label{11111hhxxcdfvhhhvssssssssscccjjghjjsdgggddddgdddffddfffddffssddcssdz2.5}
\end{equation} 
\end{lemma}
\begin{proof}
We let $p\geq 6$ and choose $q \in (3,6)$ which is close to $6$ (e.g. $q=\frac{6p}{p+1}$) such that
\begin{equation}
\begin{array}{rl}
\disp{\frac{1}{2}-\frac{3}{2p}-3(\frac{1}{q}-\frac{1}{p})} >\disp{0.}\\
\end{array}
\label{hhxxcdfvhhhvsssjjdfffssssddssdddcz2.5}
\end{equation}
 Now, we can fix $\gamma_0=\frac{3}{2}(\frac{1}{q}-\frac{1}{p})$ and $r_0$ such that $r_0 \in(\frac{3}{2}(\frac{1}{2}-\frac{1}{p}), 1)$. Then,
in view of Lemmas \ref{llssdrffmmgghhjjjnnccvvccvvkkkkgghhkkllvvlemma45630} and \ref{llssdrffmmggnnccvvccvvkkkkgghhkkllvvlemma45630}, for any $\mu\in(0,\min\{\frac{\rho_{*,1}}{2},\lambda_1\})$,
we  recall 
the $L^p$-$L^q$ estimates for the Stokes semigroup (see also Lemma 2.3 of \cite{CaoCaoLiitffg11}) to
choose positive constants $\tilde{k}_1 , \tilde{k}_2 , \tilde{k}_3 $ and  $\tilde{k}_4$  such that
\begin{equation}
\|e^{-tA}\mathcal{P}\varphi\|_{L^p(\Omega)} \leq \tilde{\kappa}_1t^{-\gamma_0}\|\varphi\|_{L^q(\Omega)} ~~~\mbox{for all}~~~ t\in(0,2)~~~\mbox{and any}~~~\varphi\in L^{q}(\Omega)
\label{1.1ddfgsssdhssdddyddffssddduisda}
\end{equation}
as well as
\begin{equation}
\| e^{-tA}\mathcal{P}\nabla\cdot\varphi\|_{L^p(\Omega)} \leq \tilde{\kappa}_2t^{-\frac{1}{2}-\frac{3}{2}(\frac{2}{p}-\frac{1}{p})}e^{-\mu t}\|\varphi\|_{L^{\frac{p}{2}}(\Omega)} ~~~\mbox{for all}~~~ t\in(0,2)~~~\mbox{and any}~~~\varphi\in L^{\frac{p}{2}}(\Omega)
\label{1.1ddssdfghssdddyssddduisda}
\end{equation}
and
\begin{equation}
\begin{array}{rl}
\|e^{-tA}\mathcal{P}\varphi\|_{L^p(\Omega)} \leq &\|A^{r_0}e^{-tA}A^{-r_0}\mathcal{P}\varphi\|_{L^p(\Omega)}\\
\leq &\tilde{\kappa}_3t^{-r_0}e^{-\mu t}\|A^{-r_0}\mathcal{P}\varphi\|_{L^p(\Omega)} \\
\leq &\tilde{\kappa}_4t^{-r_0}e^{-\mu t}\|\varphi\|_{L^2(\Omega)} ~~~\mbox{for all}~~~ t\in(0,2)~~~\mbox{and any}~~~\varphi\in L^{2}(\Omega),\\
\end{array}\label{1.sss1ddddfghhfghssdddyssddduisda}
\end{equation}
where $A$ is the Stokes operator as well as $\lambda_1=\inf Re \sigma(A) > 0$ and $\rho_{*,1}$ is   the same as {Lemma} \ref{11aaalemdfghkkmaddffffdfffgg4sssdddd5630}.
Next, in view of \dref{hhxxcdfvhhhvsssjjdfffssssddssdddcz2.5} and using  $p\geq6$, $q\in(3,6)$, we have
 $$-2\gamma_0-\frac{1}{2}-\frac{3}{2p}>-1~~~\mbox{and}~~ -\frac{1}{2}-\frac{3}{2p}>-1.$$
 Therefore,
\begin{equation}
2\tilde{\kappa}_2\int_{ 0}^1(1-\tau)^{-\frac{1}{2}-\frac{3}{2p}}\tau^{-2\gamma_0}d\tau<+\infty,
 \label{2222dddddf1.1ddfghssdddyddffssddduisda}
\end{equation}
where $\tilde{\kappa}_2$ is the same as \dref{1.1ddssdfghssdddyssddduisda}.

We finally invoke the Cauchy-Schwarz inequality to fix $\tilde{\kappa}_5 > 0$ such that
$$\|v\oplus w\|_{L^{\frac{p}{2}}(\Omega)}\leq \tilde{\kappa}_5\|v\|_{L^{p}(\Omega)}\| w\|_{L^{p}(\Omega)}~~~\mbox{for all}~~v,w\in L^{p}(\Omega).$$
We now take any positive $\eta\leq \eta_0:=\min\{\frac{1}{3\tilde{\kappa}_2\tilde{\kappa}_5|\kappa|2^{\frac{1}{2}-\frac{3}{2p}-\gamma_0}\int_{ 0}^1(1-\tau)^{-\frac{1}{2}-\frac{3}{2p}}\tau^{-2\gamma_0}d\tau},1\}$
 and let $$\eta_1 :=\frac{\eta}{3\tilde{\kappa}_2}~~
\mbox{and}~~~\eta_2:=\frac{\eta }{3\cdot2^{\gamma_0}\tilde{\kappa}_4\|\nabla\phi\|_{L^\infty(\Omega)}\int_{0}^{+\infty}s^{-r_0}
e^{-\mu s}ds},$$
 and note that as a consequence of {Lemma} \ref{11aaalemdfghkkmaddffffdfffgg4sssdddd5630} and Lemma  \ref{11aaalemmaddffffdsddfffffgg4sssdddd5630},
for any such $\eta$ we can pick $T(\eta) > 2$ such that for all $t > T(\eta) - 2$
\begin{equation}\int_{t}^{t+1}\|u_\varepsilon(\cdot,s)\|_{L^q(\Omega)}^2ds\leq \eta_1^2
\label{23344223311111hhxxcdfvhhhvddfssdddffsssddffssddfffssddffsscccjjghjjsdggggdddfffddffssddcssdz2.5}
\end{equation}
and
\begin{equation}\|n_\varepsilon(\cdot,t)-\bar{n}_0\|_{L^{2}(\Omega)}\leq\eta_2.
\label{23344223311111hhxxcdfvhhssddddddhvddfssdddffsssddffssddfffssddffsscccjjghjjsdggggdddfffddffssddcssdz2.5}
\end{equation}
To see that these choices ensure that
\begin{equation}\|u_\varepsilon(\cdot,t_0)\|_{L^p(\Omega)}\leq\eta,~~\mbox{for all}~~t_0>T(\eta),
\label{11111hhxxcdfvhhhssddsdddvdssddddfssdddffsssddffssddfffssddffsscccjjghjjsdggggdddfffddffssddcssdz2.5}
\end{equation}
we fix any such $t_0$ and then infer from \dref{23344223311111hhxxcdfvhhhvddfssdddffsssddffssddfffssddffsscccjjghjjsdggggdddfffddffssddcssdz2.5} that there exists $t_*\in(t_0-2,t_0-1)$
such that
\begin{equation}\|u_\varepsilon(\cdot,t_*)\|_{L^q(\Omega)}\leq \eta_1.
\label{23344ssdd22331111ssdd1hhxxcdfvhhhvddfssdddffsssddffssddfffssddffsscccjjghjjsdggggdddfffddffssddcssdz2.5}
\end{equation}
We now follow a standard reasoning to construct,
independently of $u_\varepsilon$, another weak solution $\tilde{u}_\varepsilon$ of the initial value problem associated
with the Navier-Stokes system $\tilde{u}_{\varepsilon t} +A\tilde{u}_{\varepsilon} = -\mathcal{P}[\kappa (Y_{\varepsilon}u_{\varepsilon} \cdot \nabla)u_{\varepsilon}]+\mathcal{P}[n_{\varepsilon}\nabla \phi]$ in $\Omega\times(t_*,t_* +2)$
with $\tilde{u}_{\varepsilon}(\cdot,t_*)=u_{\varepsilon}(\cdot,t_*)$  and some favorable additional properties, finally implying
by a uniqueness argument that actually $u_{\varepsilon}=\tilde{u}_{\varepsilon}$ and that hence $u_{\varepsilon}$ itself has these
properties. To this end, for the above $t_*,p,\gamma_0$ and $\eta$, we let
\begin{equation}X:=\{v\in C^0((t_*,t_*+2];L^p_\sigma(\Omega))\left|\right \|v\|_X:=\sup_{t\in(t_*,t_*+2)}(t-t_*)^{\gamma_0}\|v(\cdot,t)\|_{L^p(\Omega)}<\infty\}.
\label{cz2.571hhhhh51lllllccvvddfgghddfccvvhjjjkkhhggjjlsdddlll}
\end{equation}
Then we consider the mapping
$\Psi$ defined by
$$\varphi(v) (\cdot,t):= e^{-(t-t_*)A}u_{\varepsilon}(\cdot,t_*) +\int_{t_*}^te^{-(t-s)A}
\mathcal{P}\left[-\kappa\nabla\cdot(Y_\varepsilon v(\cdot,\tau)\otimes v(\cdot,\tau) +n_\varepsilon (\cdot,\tau)\nabla\phi\right]d\tau~~t\in (t_*,t_*+ 2],$$
for $v$ belonging to the closed subset
$$S:=\left\{v\in X|\|v\|\leq\eta\right\}$$
of $X$.
 Then for $v\in S$ we can use \dref{1.1ddfgsssdhssdddyddffssddduisda}--\dref{1.sss1ddddfghhfghssdddyssddduisda} as well as \dref{23344223311111hhxxcdfvhhssddddddhvddfssdddffsssddffssddfffssddffsscccjjghjjsdggggdddfffddffssddcssdz2.5} and \dref{23344ssdd22331111ssdd1hhxxcdfvhhhvddfssdddffsssddffssddfffssddffsscccjjghjjsdggggdddfffddffssddcssdz2.5} to estimate
 \begin{equation}
\begin{array}{rl}
&\|\varphi(v) (\cdot, t)\|_{L^p(\Omega)}\\
\leq&\disp{\tilde{\kappa}_1(t- t_*)^{-\gamma_0}\|u_\varepsilon(\cdot, t_*)\|_{L^q(\Omega)} +\tilde{\kappa}_2|\kappa|
\int_{ t_*}^t(t-\tau)^{-\frac{1}{2}-\frac{3}{2}(\frac{2}{p}-\frac{1}{p})}e^{-\mu(t-\tau)}\|v(\cdot, \tau)\oplus v(\cdot, \tau)\|_{L^{\frac{p}{2}}(\Omega)}d\tau}\\
&\disp{+\int_{ t_*}^t\|\mathcal{P}\left\{[n_\varepsilon(\cdot,\tau)-\bar{n}_0]\nabla\phi\right\}\|_{L^p(\Omega)}d\tau}\\
\leq&\disp{\tilde{\kappa}_1(t- t_*)^{-\gamma_0}\|u_\varepsilon(\cdot, t_*)\|_{L^q(\Omega)} +
\tilde{\kappa}_2|\kappa|\int_{ t_*}^t(t-\tau)^{-\frac{1}{2}-\frac{3}{2}(\frac{2}{p}-\frac{1}{p})}e^{-\mu(t-\tau)}\|v(\cdot, \tau)\oplus v(\cdot, \tau)\|_{L^{\frac{p}{2}}(\Omega)}d\tau}\\
&\disp{+\tilde{\kappa}_4\|\nabla\phi\|_{L^\infty(\Omega)}\int_{ t_*}^t(t-\tau)^{-r_0}
e^{-\mu(t-\tau)}\|n_\varepsilon(\cdot,\tau)-\bar{n}_0\|_{L^2(\Omega)}d\tau}\\
\leq&\disp{\tilde{\kappa}_1(t- t_*)^{-\gamma_0}\|u_\varepsilon(\cdot, t_*)\|_{L^q(\Omega)} +
\tilde{\kappa}_2\tilde{\kappa}_5|\kappa|\int_{ t_*}^t(t-\tau)^{-\frac{1}{2}-\frac{3}{2p}(\frac{2}{p}-\frac{1}{p})}e^{-\mu(t-\tau)}\|v(\cdot, \tau)\|_{L^{p}(\Omega)}^2d\tau}\\
&\disp{+\tilde{\kappa}_4\|\nabla\phi\|_{L^\infty(\Omega)}\eta_2\int_{ t_*}^{ t_*+2}(t-\tau)^{-r_0}
e^{-\mu(t-\tau)}d\tau}\\
\leq&\disp{\tilde{\kappa}_1(t- t_*)^{-\gamma_0}\eta_1 +
\tilde{\kappa}_2\tilde{\kappa}_5|\kappa|\eta^2\int_{ t_*}^t(t-\tau)^{-\frac{1}{2}-\frac{3}{2p}}\tau^{-2\gamma_0}d\tau}\\
&\disp{+\tilde{\kappa}_4\|\nabla\phi\|_{L^\infty(\Omega)}\eta_2\int_{0}^{+\infty}s^{-r_0}
e^{-\mu s}ds~~~\mbox{for all}~~t\in(t_*,t_*+2),}\\
\end{array}
\label{cz2.571hhhhh51lllllccvvhddfccvvhjjjkkhhggjjlsdddlll}
\end{equation}
where $\mu$ and $r_0$ are the same as \dref{1.sss1ddddfghhfghssdddyssddduisda}.
Therefore, in light of the definition of $\gamma_0$,
\dref{2222dddddf1.1ddfghssdddyddffssddduisda}, \dref{23344ssdd22331111ssdd1hhxxcdfvhhhvddfssdddffsssddffssddfffssddffsscccjjghjjsdggggdddfffddffssddcssdz2.5} as well as \dref{cz2.571hhhhh51lllllccvvddfgghddfccvvhjjjkkhhggjjlsdddlll} and \dref{cz2.571hhhhh51lllllccvvhddfccvvhjjjkkhhggjjlsdddlll},  we see that
for every $t\in ( t_* , t_* + 2)$ and every $v\in S$
\begin{equation}
\begin{array}{rl}
&(t- t_*)^{\gamma_0}\|\varphi(v) (\cdot, t)\|_{L^p(\Omega)}\\
\leq&\disp{\tilde{\kappa}_1\eta_1 +
\tilde{\kappa}_2\tilde{\kappa}_5|\kappa|\eta^2(t- t_*)^{\gamma_0-\frac{1}{2}-\frac{3}{2p}-2\gamma_0+1}\int_{ 0}^1(1-\tau)^{-\frac{1}{2}-\frac{3}{2p}}\tau^{-2\gamma_0}d\tau}\\
&\disp{+\tilde{\kappa}_4\|\nabla\phi\|_{L^\infty(\Omega)}\eta_2(t- t_*)^{\gamma_0}\int_{0}^{+\infty}s^{-r_0}
e^{-\mu s}ds}\\
\leq&\disp{\tilde{\kappa}_1\eta_1 +
\tilde{\kappa}_2\tilde{\kappa}_5|\kappa|\eta^22^{\frac{1}{2}-\frac{3}{2p}-\gamma_0}\int_{ 0}^1(1-\tau)^{-\frac{1}{2}-\frac{3}{2p}}\tau^{-2\gamma_0}d\tau}\\
&\disp{+\tilde{\kappa}_4\|\nabla\phi\|_{L^\infty(\Omega)}\eta_22^{\gamma_0}\int_{0}^{+\infty}s^{-r_0}
e^{-\mu s}ds}\\
\leq&\disp{\frac{\eta}{3}+\frac{\eta}{3}+\frac{\eta}{3}}\\
<&\disp{\eta,}\\
\end{array}
\label{cz2.571hhhhh51lllllccvvhddfccvvhjjjkkhhggjjllll}
\end{equation}
from which it readily follows that $\varphi(S)\subset S$.
 Likewise, for $v\in S$  and $w\in S$ we can
use Lemma \ref{llssdrffmmggnnccvvccvvkkkkgghhkkllvvlemma45630} to find that
\begin{equation}
\begin{array}{rl}
&\|(\varphi(v)-\varphi(w))(\cdot, t) \|_{L^p(\Omega)}\\
\leq&\disp{
\tilde{\kappa}_2|\kappa|\int_{ t_*}^t(t-\tau)^{-\frac{1}{2}-\frac{3}{2p}}\|v(\cdot, \tau)\oplus v(\cdot, \tau)-w(\cdot, \tau)\oplus w(\cdot, \tau)\|_{L^{\frac{p}{2}}(\Omega)}d\tau}\\
=&\disp{
\tilde{\kappa}_2|\kappa|\int_{ t_*}^t(t-\tau)^{-\frac{1}{2}-\frac{3}{2p}}\|v(\cdot, \tau)\oplus (v(\cdot, \tau)-w(\cdot, \tau))+(v(\cdot, \tau)-w(\cdot, \tau))\oplus w(\cdot, \tau)\|_{L^{\frac{p}{2}}(\Omega)}d\tau}\\
\leq&\disp{
\tilde{\kappa}_2\tilde{\kappa}_5|\kappa|\int_{ t_*}^t(t-\tau)^{-\frac{1}{2}-\frac{3}{2p}}(\|v(\cdot, \tau)\|_{L^{p}}+\|w(\cdot, \tau)\|_{L^{p}})\|v(\cdot, \tau)-w(\cdot, \tau)\|_{L^{p}(\Omega)}d\tau}\\
\leq&\disp{
\tilde{\kappa}_2\tilde{\kappa}_5|\kappa|\int_{t_*}^t\tau^{-\frac{1}{2}-\frac{3}{2p}}2\eta \tau ^{-\gamma_0}\cdot \tau ^{-\gamma_0}\|v-w\|_{X} d\tau~~~\mbox{for all}~~t\in ( t_* , t_* + 2),}\\
\end{array}
\label{234cz2.571hhhhh51lllllccvvhsdddddfccvvhjjjkkhhggjjlsdddlll}
\end{equation}
which implies that
\begin{equation}
\begin{array}{rl}
&(t- t_*)^{\gamma_0}\|(\varphi(v)-\varphi(w))(\cdot, t)\|_{L^p(\Omega)}\\
\leq&\disp{
2^{\frac{3}{2}-\frac{3}{2p}-\gamma_0} \tilde{\kappa}_2\tilde{\kappa}_5|\kappa|\int_{ 0}^1(1-\tau)^{-\frac{1}{2}-\frac{3}{2p}}\tau^{-2\gamma_0}d\tau
\eta\|v-w\|_{X}~~~\mbox{for all}~~t\in ( t_* , t_* + 2).}\\
\end{array}
\label{234cz2.571hhssdhhh51lllllccvvhsdddddfccvvhjjjkkhhggjjlsdddlll}
\end{equation}
Recalling
$$\eta\leq\frac{1}{3\tilde{\kappa}_2\tilde{\kappa}_5|\kappa|2^{\frac{1}{2}-\frac{3}{2p}-\gamma_0}\int_{ 0}^1(1-\tau)^{-\frac{1}{2}-\frac{3}{2p}}\tau^{-2\gamma_0}d\tau},$$
thus,
$$2^{\frac{3}{2}-\frac{3}{2p}-\gamma_0} \tilde{\kappa}_2\tilde{\kappa}_5|\kappa|\int_{ 0}^1(1-\tau)^{-\frac{1}{2}-\frac{3}{2p}}\tau^{-2\gamma_0}d\tau
\eta\leq \frac{2}{3}<1,$$
so that,  \dref{234cz2.571hhssdhhh51lllllccvvhsdddddfccvvhjjjkkhhggjjlsdddlll} implies that $\varphi$ acts as a contraction on
$S$ and hence possesses a unique fixed point $\tilde{u}_\varepsilon$. In view of
a standard arguments (see e.g.
 \cite{Sohr}), it
follows that $\tilde{u}_\varepsilon$ in fact is a weak solution of the Navier-Stokes subsystem of \dref{1.1fghyuisda} in
$\Omega\times(t_*,t_* +2)$ subject to the initial condition $\tilde{u}_{\varepsilon}(\cdot,t_*)=u_{\varepsilon}(\cdot,t_*)$. 
Next,
using  the definition of $S$ and $X$, we also derive that
\begin{equation}\int_{t_*}^{t_*+2}\|\tilde{u}_\varepsilon(\cdot,t)\|_{L^p(\Omega)}^{q_0}dt\leq\eta,
\label{11111hhxxcdfvhhhsssddsddvdssddddfssdddffsssddffssddfffssddffsscccjjghjjsdggggdddfffddffssddcssdz2.5}
\end{equation}
where
$$q_0=\frac{2p}{p-3}$$
satisfy
\begin{equation}\frac{2}{q_0}+\frac{3}{p}=1~~\mbox{and}~~q_0\gamma_0=\frac{2p}{p-3}\times\frac{3}{2}(\frac{1}{q}-\frac{1}{p})<1~~~\mbox{by using}~~q>3.\label{ssddssd11111hhxxcdfvhhhsssddsddvdssddddfssdddffsssddffssddfffssddffsscccjjghjjsdggggdddfffddffssddcssdz2.5}
\end{equation}
Collecting \dref{11111hhxxcdfvhhhsssddsddvdssddddfssdddffsssddffssddfffssddffsscccjjghjjsdggggdddfffddffssddcssdz2.5}
 and \dref{ssddssd11111hhxxcdfvhhhsssddsddvdssddddfssdddffsssddffssddfffssddffsscccjjghjjsdggggdddfffddffssddcssdz2.5}, a well-known uniqueness
property of the Navier-Stokes equations (see e.g. Theorem V.2.5.1 of  \cite{Sohr}) entails that $\tilde{u}_\varepsilon$ must coincide with $u_\varepsilon$
in $\Omega\times(t_*,t_* +2)$.
 In particular, since $t_0\in  (t_*+ 1,t_* + 2)$, \dref{11111hhxxcdfvhhhsssddsddvdssddddfssdddffsssddffssddfffssddffsscccjjghjjsdggggdddfffddffssddcssdz2.5} thus  implies that
$$\|u_\varepsilon(\cdot,t_0)\|_{L^p(\Omega)}\leq \|\tilde{u}_\varepsilon(\cdot,t_0)\|_{L^p(\Omega)}\leq\eta(t_0-t_*)^{-\gamma_0}\leq\eta,
$$
and thereby establishes \dref{11111hhxxcdfvhhhssddsdddvdssddddfssdddffsssddffssddfffssddffsscccjjghjjsdggggdddfffddffssddcssdz2.5},
which together with the fact that $\eta\in(0,\eta_0]$ was
arbitrary yields to \dref{11111hhxxcdfvhhhvssssssssscccjjghjjsdgggddddgdddffddfffddffssddcssdz2.5}.
\end{proof}

With the help of {Lemma} \ref{11aaalemdfghkkmaddffffdfffgg4sssdddd5630} as well as {Lemma} \ref{lemmaghjffggsjjjjjsddgghhmk4563025xxhjklojjkkk}  and Lemma \ref{aaalemmaddffffsddddfffgg4sssdddd5630}, we can achieve the boundedness of $\nabla c_{\varepsilon}$ in
$L^\infty((T, \infty);L^4(\Omega))$ and $n_{\varepsilon}$ in $L^\infty(\Omega\times(T, \infty))$  for  some large  $T>2$.
\begin{lemma}\label{11x344ccffggfhhlemma45625xxhjhjuioookloghyui}
Under the assumptions of Theorem \ref{thaaaeorem3}, then 
one can find that $T_{*,4}> 2,\rho_{*,4}\in(0,1)$ and $C_{*,4}> 0$
   such that 
   for each $\varepsilon\in (0, 1)$,
\begin{equation}
\int_\Omega|\nabla c_\varepsilon (x,t)|^{2}< C_{*,4}e^{-\rho_{*,4} t}~~\mbox{for all}~~ t>T_{*,4}
\label{222ssssddddddddxxxcvvddcvdhjjjjdffbbggddczv.5ghcfg924ghyuji}
\end{equation}
and
\begin{equation}
\int_\Omega|\nabla c_\varepsilon (x,t)|^{4}<C_{*,4}e^{-\rho_{*,4} t}~~\mbox{for all}~~ t>T_{*,4}.
\label{ddxxxcvvddcvdhjddfffjjjdffbbggddczv.5ghcfg924ghyuji}
\end{equation}
\end{lemma}
\begin{proof}
To show this Lemma,  it will be convenient to introduce the following notation:
$$
w_\varepsilon=c_\varepsilon-\bar{n}_0.
$$
Accordingly, in light of the second equation of \dref{1.1fghyuisda}, by a
  straightforward computation, 
  we see $w_\varepsilon$ has the following property:
$$w_{\varepsilon t}-\Delta w_{\varepsilon}+u_{\varepsilon}\cdot\nabla w_{\varepsilon}+w_{\varepsilon}=n_{\varepsilon}-\bar{n}_0,\quad
x\in \Omega,\; t>0.$$
For   any $\eta\in(0,1)$,
let us apply Corollary \ref{11aaalemdfghkkmaddffffdfffgg4sssdddd5630},
Lemma \ref{aaalemmaddffffsddddfffgg4sssdddd5630} to fix $T_{1,*} > 2$ and
constants $\rho_{1,*}\in(0,1),C_{1,*}>0$ and $C_{2,*}>0$ such that for any $t>T_{1,*} $
 \begin{equation}\|u_{\varepsilon}\|_{L^\infty((t,+\infty);L^p(\Omega))}< \eta~~~\mbox{for all}~~p>1
\label{11111hhxxcdfvhssddffhhvsssssssssddfffcccjjghjjsdgggddddddgdddffddfddddffddffssddcssdz2.5}
\end{equation}
as well as
 \begin{equation}\|n_{\varepsilon}(\cdot,t)-\bar{n}_0\|_{L^\infty((t,+\infty);L^2(\Omega))}<C_{1,*}e^{-\rho_{1,*} t}
\label{11111hhxxcdfvhhhvssssssghhhhsssddfffcccjjghjjsdgggddddgdddffddfddddffddffssddcssdz2.5}
\end{equation}
and
 \begin{equation}\|w_\varepsilon\|_{L^\infty((t,+\infty);L^2(\Omega))}<C_{2,*}e^{-\rho_{1,*} t}.
\label{11111hhxxcdfvhhhvffgggsssssssssddfffcccjjghjjsdgggdssdddgdddffddfddddffddffssddcssdz2.5}
\end{equation}
On the other hand, by {Lemma} \ref{11aaalemdfghkkmaddffffdfffgg4sssdddd5630}, we also have
\begin{equation}
\int_{T_{1,*}}^{\infty}\int_{\Omega}|\nabla w_{\varepsilon}|^2 \leq C_{3,*} ~~~\mbox{for all}~~~~\varepsilon\in(0,1)
\label{ssddddff1.1ddfghssdddssssssddsyddsssffssddduisda}
\end{equation}
and some positive constant $C_{3,*}.$
Then \dref{ssddddff1.1ddfghssdddssssssddsyddsssffssddduisda} in particular implies that 
we can find $\tilde{t}_{*}\in(T_{1,*},\infty)$ such that
 \begin{equation}\|\nabla w_{\varepsilon}(\cdot,\tilde{t}_{*})\|_{L^2(\Omega)}\leq C_{3,*}.
\label{223311111hhxxcdfvhhhvsssssssssddfffcccjjgssddhjjsdggdffgggddddddgdddffddfddddffddffssddcssdz2.5}
\end{equation}

Next,  testing the second equation in \dref{1.1fghyuisda} against $-\Delta w_{\varepsilon}$ and employing  the Young inequality  
yields
\begin{equation}
\begin{array}{rl}
\disp{\frac{1}{{2}}\frac{d}{dt} \|\nabla w_{\varepsilon}(\cdot,t)\|^{{{2}}}_{L^{{2}}(\Omega)}}= &\disp{\int_{\Omega}  -\Delta w_{\varepsilon}(\Delta w_{\varepsilon}-w_{\varepsilon}+n_{\varepsilon}-\bar{n}_0-u_{\varepsilon}\cdot\nabla  w_{\varepsilon})}
\\
\leq&\disp{-\frac{1}{4}\int_{\Omega}  |\Delta w_{\varepsilon}|^2-\int_{\Omega} |\nabla w_{\varepsilon}|^{2}+
\int_\Omega (n_{\varepsilon}-\bar{n}_0)^2+\int_\Omega |u_{\varepsilon}\cdot\nabla  w_{\varepsilon}||\Delta w_{\varepsilon}|}\\
\end{array}
\label{cz2.5ghju4sdssssddffgf8156}
\end{equation}
for all $t>t_{*}$.
Next, one needs to estimate the last term on the right-hand side of \dref{cz2.5ghju4sdssssddffgf8156}.
Indeed, in view of   \dref{11111hhxxcdfvhssddffhhvsssssssssddfffcccjjghjjsdgggddddddgdddffddfddddffddffssddcssdz2.5}
  and \dref{11111hhxxcdfvhhhvffgggsssssssssddfffcccjjghjjsdgggdssdddgdddffddfddddffddffssddcssdz2.5},  we derive from  the H\"{o}lder inequality,
the Gagliardo--Nirenberg inequality, and  the Young inequality that there exist positive constants
$C_{4,*}$ as well as $C_{5,*}$ and $C_{6,*}$ such that
 \begin{equation}
\begin{array}{rl}
\disp{\int_\Omega |u_{\varepsilon}\cdot\nabla  w_{\varepsilon}||\Delta w_{\varepsilon}|}
\leq&\disp{\|u_{\varepsilon}\|_{L^{6}(\Omega)}\|\nabla w_{\varepsilon}\|_{L^{3}(\Omega)}}\|\Delta w_{\varepsilon}\|_{L^{2}(\Omega)}\\
\leq&\disp{C_{4,*}\|\nabla w_{\varepsilon}\|_{L^{3}(\Omega)}}\|\Delta w_{\varepsilon}\|_{L^{2}(\Omega)}\\
\leq&\disp{C_{5,*}\left(\|\Delta w_{\varepsilon}\|^{\frac{3}{4}}_{L^2(\Omega)}\| w_{\varepsilon}\|^{\frac{1}{4}}_{L^2(\Omega)}+\| w_{\varepsilon}\|^{2}_{L^2(\Omega)}\right)\|\Delta w_{\varepsilon}\|_{L^{2}(\Omega)}}\\
\leq&\disp{C_{5,*}\max\{C_{2,*}^2,C_{2,*}^{\frac{1}{4}}\}e^{-\frac{\rho_{1,*}}{4} t}(\|\Delta w_{\varepsilon}\|^{\frac{7}{4}}_{L^2(\Omega)}+\|\Delta w_{\varepsilon}\|_{L^{2}(\Omega)})}\\
\leq&\disp{\frac{1}{4}\|\Delta w_{\varepsilon}\|^{2}_{L^2(\Omega)}+C_{6,*}e^{-\frac{\rho_{1,*}}{2} t}~~\mbox{for all}~~ t>\tilde{t}_{*},}\\
\end{array}
\label{dd1ssddff1cfsdfggghhvggcz2.5ghju48156}
\end{equation}
where we have used the fact that $t>\tilde{t}_{*}\geq T_{1,*}>2.$
Inserting \dref{dd1ssddff1cfsdfggghhvggcz2.5ghju48156} into \dref{cz2.5ghju4sdssssddffgf8156} and using \dref{11111hhxxcdfvhhhvssssssghhhhsssddfffcccjjghjjsdgggddddgdddffddfddddffddffssddcssdz2.5}, we have
$$
\begin{array}{rl}
\disp{\frac{1}{{2}}\frac{d}{dt} \|\nabla w_{\varepsilon}(\cdot,t)\|^{{{2}}}_{L^{{2}}(\Omega)}}\leq &\disp{-\frac{1}{2}\int_{\Omega}  |\Delta w_{\varepsilon}|^2-\int_{\Omega} |\nabla w_{\varepsilon}|^{2}+
C_{6,*}e^{-\frac{\rho_{1,*}}{2} t}+
\int_\Omega (n_{\varepsilon}-\bar{n}_0)^2}\\
\leq &\disp{-\frac{1}{2}\int_{\Omega}  |\Delta w_{\varepsilon}|^2-\int_{\Omega} |\nabla w_{\varepsilon}|^{2}+
C_{6,*}e^{-\frac{\rho_{1,*}}{2} t}+
C_{1,*}e^{-\rho_{1,*} t}}\\
\leq &\disp{-\frac{1}{2}\int_{\Omega}  |\Delta w_{\varepsilon}|^2-\int_{\Omega} |\nabla w_{\varepsilon}|^{2}+
\max\{C_{6,*},C_{1,*}\}e^{-\frac{\rho_{1,*}}{2} t}~~\mbox{for all}~~ t>\tilde{t}_{*}.}\\
\end{array}
$$
Integrating the above inequality 
 in time and using \dref{223311111hhxxcdfvhhhvsssssssssddfffcccjjgssddhjjsdggdffgggddddddgdddffddfddddffddffssddcssdz2.5}, we derive that there exist positive constants $\rho_{2,*}$ and $C_{7,*}$ such that
\begin{equation}\int_\Omega|\nabla w_{\varepsilon}(x,t)|^2dx\leq C_{7,*}e^{-\rho_{2,*} t}~~~\mbox{for all}~~t\geq \tilde{t}_{*}.
\label{11111hhxxcdfddffgvhhhvsssssssssddfffcccjjghjjsdgggdssdddssddgdddffddfddddffddffssddcssdz2.5}
\end{equation}
However, from  \dref{11111hhxxcdfddffgvhhhvsssssssssddfffcccjjghjjsdgggdssdddssddgdddffddfddddffddffssddcssdz2.5} and \dref{11111hhxxcdfvhhhvffgggsssssssssddfffcccjjghjjsdgggdssdddgdddffddfddddffddffssddcssdz2.5}, with the help of the Sobolev imbedding theorem, it follows that there exist positive constants $C_{8,*}$ and $\rho_{3,*}$ such that
\begin{equation}
\begin{array}{rl}
\|w_{\varepsilon}(\cdot, t)\|_{L^{5}(\Omega)}\leq  C_{8,*}e^{-\rho_{3,*} t}~ \mbox{for all}~ t\geq \tilde{t}_{*}.\\
\end{array}
\label{cz2.5715jkkjjkkkkcvccvvhj4456777jjkddfffffkhhgll2233}
\end{equation}
Now, 
applying
the variation-of-constants formula
for $w_{\varepsilon}$ and applying $\nabla\cdot u_{\varepsilon}=0$ in $x\in \Omega, t>0$, we obtain that for any $t> 2\tilde{t}_{*},$
\begin{equation}
w_{\varepsilon}(\cdot,t)=e^{(t-\tilde{t}_{*})(\Delta-1)}w_\varepsilon(\cdot,\tilde{t}_{*}) +\int_{\tilde{t}_{*}}^{t}e^{(t-s)(\Delta-1)}(n_{\varepsilon}(\cdot,s)-\bar{n}_0+\nabla \cdot(u_{\varepsilon}(\cdot,s) w_{\varepsilon}(\cdot,s)) ds,
\label{5555hhjjjfghbnmcz2.5ghju48cfg924ghyuji}
\end{equation}
so that, we have
\begin{equation}
\begin{array}{rl}
&\disp{\|\nabla w_{\varepsilon}(\cdot, t)\|_{L^{4}(\Omega)}}\\
\leq&\disp{\|\nabla e^{(t-\tilde{t}_{*})(\Delta-1)} w_{\varepsilon}(\cdot,\tilde{t}_{*})\|_{L^{4}(\Omega)}+
\int_{\tilde{t}_{*}}^t\|\nabla e^{(t-s)(\Delta-1)}(n_{\varepsilon}(\cdot,s)-\bar{n}_0)\|_{L^4(\Omega)}ds}\\
&\disp{+\int_{\tilde{t}_{*}}^t\|\nabla e^{(t-s)(\Delta-1)}\nabla \cdot(u_{\varepsilon}(\cdot,s)
w_{\varepsilon}(\cdot,s))\|_{L^4(\Omega)}ds~~\mbox{for all}~ ~~ t>2\tilde{t}_{*}.}\\
\end{array}
\label{222244444zjccfgghfgjcvbscz2.5297x96301ku}
\end{equation}
To address the right-hand side of \dref{222244444zjccfgghfgjcvbscz2.5297x96301ku}, in view of \dref{cz2.5715jkkjjkkkkcvccvvhj4456777jjkddfffffkhhgll2233}, we first use Lemma \ref{llssdrffmmggnnccvvccvvkkkkgghhkkllvvlemma45630}    to get for some positive constants $\mu_{1,*}>0,\rho_{4,*}>0$ and $C_{9,*}>0$ such that
\begin{equation}
\begin{array}{rl}
\|\nabla e^{(t-\tilde{t}_{*})(\Delta-1)} w_\varepsilon(\cdot,\tilde{t}_{*})\|_{L^{4}(\Omega)}\leq &\disp{\kappa_3(t-\tilde{t}_{*})^{-\frac{1}{2}-\frac{3}{2}(\frac{1}{5}-\frac{1}{4})}e^{-\mu_{1,*} t}\|w_\varepsilon(\cdot, t)\|_{L^5(\Omega)}}\\
\leq &\disp{\kappa_3\|w_\varepsilon(\cdot, t)\|_{L^5(\Omega)}}\\
\leq &\disp{C_{9,*}e^{-\rho_{4,*} t}~~\mbox{for all}~ ~~ t> 2\tilde{t}_{*}}\\
\end{array}
\label{zjccffgbhjcghhhjjjvvvbscz2.5297x96301ku}
\end{equation}
by using $\tilde{t}_{*}>1,$ where $\kappa_3$ is the same as Lemma \ref{llssdrffmmggnnccvvccvvkkkkgghhkkllvvlemma45630}.
Since 
$$-\frac{1}{2}-\frac{3}{2}\left(\frac{1}{2}-\frac{1}{4}\right)>-1,$$
 together with this,
 by using Lemma \ref{llssdrffmmggnnccvvccvvkkkkgghhkkllvvlemma45630} and \dref{11111hhxxcdfvhhhvssssssghhhhsssddfffcccjjghjjsdgggddddgdddffddfddddffddffssddcssdz2.5},  the second term on  the right-hand side of  \dref{222244444zjccfgghfgjcvbscz2.5297x96301ku} is estimated
as
\begin{equation}
\begin{array}{rl}
&\disp{\int_{\tilde{t}_{*}}^t\|\nabla e^{(t-s)(\Delta-1)}(n_{\varepsilon}(\cdot,s)-\bar{n}_0)\|_{L^{4}(\Omega)}ds}\\
\leq&\disp{C_{10,*}\int_{t_*}^t[1+(t-s)^{-\frac{1}{2}-\frac{3}{2}(\frac{1}{2}-\frac{1}{4})}] e^{-\lambda_{1,*}(t-s)}\|n_{\varepsilon}(\cdot,s)-\bar{n}_0\|_{L^{2}(\Omega)}ds}\\
\leq&\disp{C_{10,*}\int_{\tilde{t}_{*}}^t[1+(t-s)^{-\frac{1}{2}-\frac{3}{2}(\frac{1}{2}-\frac{1}{4})}] e^{-\lambda_{1,*}(t-s)}C_{1,*}e^{-\rho_{1,*} s}ds}\\
\leq&\disp{C_{11,*}e^{-\rho_{5,*} t}~~ \mbox{for all}~ ~~ t> 2\tilde{t}_{*}}\\
\end{array}
\label{zjccffgbhjcvvvbscz2.5297x96301ku}
\end{equation}
and some positive constants $C_{10,*},\rho_{5,*},\lambda_{1,*}$ and  $C_{11,*}$. Here $C_{1,*}$ and $\rho_{1,*}$ are the same as \dref{11111hhxxcdfvhhhvssssssghhhhsssddfffcccjjghjjsdgggddddgdddffddfddddffddffssddcssdz2.5}.
Next, we will address the third term on the right-hand
side of \dref{222244444zjccfgghfgjcvbscz2.5297x96301ku}. To this end,  we choose $\iota =\frac{23}{48}$ and $\tilde{\kappa}=\frac{1}{96}$, then
 $\frac{1}{2} + \iota+\tilde{\kappa}<1$. Therefore, in view of the H\"{o}lder inequality,
 we derive from Lemma \ref{llssdrffmmggnnccvvccvvkkkkgghhkkllvvlemma45630}, \dref{11111hhxxcdfvhssddffhhvsssssssssddfffcccjjghjjsdgggddddddgdddffddfddddffddffssddcssdz2.5} and \dref{cz2.5715jkkjjkkkkcvccvvhj4456777jjkddfffffkhhgll2233}  that there exist positive constants $\lambda_{2,*},\rho_{6,*},C_{12,*}$, $C_{13,*}$, $C_{14,*},C_{15,*}$ and $C_{16,*}$ such that
\begin{equation}
\begin{array}{rl}
&\disp{\int_{\tilde{t}_{*}}^t\|\nabla e^{(t-s)(\Delta-1)}\nabla \cdot(u_{\varepsilon}(\cdot,s) w_{\varepsilon}(\cdot,s))\|_{L^{4}(\Omega)}ds}\\
\leq&\disp{C_{12,*}\int_{\tilde{t}_{*}}^t\|e^{(t-s)(\Delta-1)}\nabla \cdot(u_{\varepsilon}(\cdot,s) w_{\varepsilon}(\cdot,s))\|_{W^{1,4}(\Omega)}ds}\\
\leq&\disp{C_{13,*}\int_{\tilde{t}_{*}}^t\|(-\Delta+1)^\iota e^{(t-s)(\Delta-1)}\nabla \cdot(u_\varepsilon(\cdot,s) w_{\varepsilon}(\cdot,s))\|_{L^{\frac{9}{2}}(\Omega)}ds}\\
\leq&\disp{C_{14,*}\int_{\tilde{t}_{*}}^t(t-s)^{-\iota-\frac{1}{2}-\tilde{\kappa}} e^{-\lambda_{2,*}(t-s)}\|u_\varepsilon(\cdot,s) w_{\varepsilon}(\cdot,s)\|_{L^{\frac{9}{2}}(\Omega)}ds}\\
\leq&\disp{C_{15,*}\int_{\tilde{t}_{*}}^t(t-s)^{-\iota-\frac{1}{2}-\tilde{\kappa}} e^{-\lambda_{2,*}(t-s)}\|u_\varepsilon(\cdot,s)\|_{L^{45}(\Omega)}\| w_{\varepsilon}(\cdot,s)\|_{L^{5}(\Omega)}ds}\\
\leq&\disp{C_{16,*}e^{-\rho_{6,*} t}~~ \mbox{for all}~ ~~ t> 2\tilde{t}_{*}.}\\
\end{array}
\label{zjccffgbhjcvddfgghhvvbscz2.5297x96301ku}
\end{equation}
Here we have used the fact that
$$\begin{array}{rl}\disp\int_{\tilde{t}_{*}}^t(t-s)^{-\iota-\frac{1}{2}-\tilde{\kappa}} e^{-\lambda_{2,*}(t-s)}ds\leq&\disp{\int_{0}^{\infty}\sigma^{-\iota-\frac{1}{2}-\tilde{\kappa}} e^{-\lambda_{2,*}\sigma}d\sigma<+\infty.}\\
\end{array}
$$
Finally, collecting  \dref{222244444zjccfgghfgjcvbscz2.5297x96301ku}--\dref{zjccffgbhjcvddfgghhvvbscz2.5297x96301ku},
 we can obtain that there  exists a positive constant $C_{17,*}$ and $\rho_{7,*}$ such that
  \begin{equation}
\int_{\Omega}|\nabla {c_{\varepsilon}}(x,t)|^{4}dx\leq C_{17,*}e^{-\rho_{7,*} t}~~ \mbox{for all}~ ~~ t> 2\tilde{t}_{*}.
\label{hjui909klsdddfopji11ddfff5}
\end{equation}
This completes the proof of Lemma \ref{11x344ccffggfhhlemma45625xxhjhjuioookloghyui}.
\end{proof}

\begin{lemma}\label{x344ddfgghhjjccffggfhhlemma45625xxhjhjuioookloghyui}
Under the assumptions of Theorem \ref{thaaaeorem3}, one can find that $T_{*,5}> 2$ and $C_{*,5}> 0$  with the property that for all $\varepsilon\in (0, 1)$ we have
\begin{equation}
\|n_{\varepsilon}(\cdot,t)\|_{L^{\frac{10}{3}}(\Omega)}\leq C_{*,5}~~\mbox{for all}~~ t\geq T_{*,5}.
\label{ddxxxcvvddcvdhjjjjdffbbggddczv.5ghcfgssddd924ghyuji}
\end{equation}
\end{lemma}
\begin{proof}
Firstly,
let us apply 
Lemma \ref{11x344ccffggfhhlemma45625xxhjhjuioookloghyui} to fix $T_{1,**} > 2$   as well as
constants $\rho_{1,**}\in(0,1)$ and $C_{1,**}>0$ 
such that for any $t\geq T_{1,**} $
 \begin{equation}\|\nabla c_\varepsilon\|_{L^\infty((t,+\infty);L^4(\Omega))}<C_{1,**}e^{-\rho_{1,**} t}.
\label{3222566311111hhxxcdfvhhhvffgggsssssssssddfffcccjjghjjsdgggdssdddgdddffddfddddffddffssddcssdz2.5}
\end{equation}
On the other hand, by {Proposition} \ref{11aaaerrrlemdfghkkmaddffffdfffgg4sssdddd5630}, we also derive that
\begin{equation}
\int_{T_{1,**}}^{\infty}\int_{\Omega}n_{\varepsilon}^{\frac{10}{3}} \leq C_{2,**} ~~~\mbox{for all}~~~~\varepsilon\in(0,1)
\label{ssdd22ddff1.1ddfghssdddssssssddsyddsssffssddduisda}
\end{equation}
and some positive constant $C_{2,**}.$
Then \dref{ssdd22ddff1.1ddfghssdddssssssddsyddsssffssddduisda} in particular implies that 
we can find $t_{**}\in(T_{1,**},\infty)$ such that
 \begin{equation}\|n_{\varepsilon}(\cdot,t_{**})\|_{L^{\frac{10}{3}}(\Omega)}\leq C_{2,**}.
\label{223311111hhxxcdfvhhhvsssssssssddfffcccjjgssddhjjsdggdffgggddddddgdddffddfddddffddffssddcssdz2.5}
\end{equation}
Let $p=\frac{10}{3}$.
Testing the first  equation in $\dref{1.1fghyuisda}$ by ${n^{p-1}_{\varepsilon}}$, using the fact $\nabla\cdot u_{\varepsilon}=0$,
and applying the Young  inequality and the H\"{o}lder inequality, we have
\begin{equation}
\begin{array}{rl}
&\disp{\frac{1}{p}\frac{d}{dt}\int_{\Omega}n^{{{p}}}_{\varepsilon}+({{p}-1})
\int_{\Omega}n^{{{p}-2}}_{\varepsilon}|\nabla c_{\varepsilon}|^2}
\\
=&\disp{-\int_\Omega n^{p-1}_{\varepsilon}\nabla\cdot(n_{\varepsilon}F_{\varepsilon}(n_{\varepsilon})S_\varepsilon(x, n_{\varepsilon}, c_{\varepsilon})\nabla c_{\varepsilon})}
\\
=&\disp{(p-1)\int_\Omega n^{p-2}_{\varepsilon}F_{\varepsilon}(n_{\varepsilon})S_\varepsilon(x, n_{\varepsilon}, c_{\varepsilon})\nabla c_{\varepsilon}\cdot \nabla n_{\varepsilon}}
\\
\leq&\disp{\frac{{p}-1}{2}
\int_{\Omega}n^{{{p}-2}}_{\varepsilon}|\nabla n_{\varepsilon}|^2+\frac{{p}-1}{2}
\int_{\Omega}n^{{{p}}}_{\varepsilon}(1+n_{\varepsilon})^{-2\alpha}|\nabla c_{\varepsilon}|^2}\\
\leq&\disp{\frac{{p}-1}{2}
\int_{\Omega}n^{{{p}-2}}_{\varepsilon}|\nabla n_{\varepsilon}|^2+\frac{{p}-1}{2}
\left(\int_{\Omega}n^{{{2p}}}_{\varepsilon}(1+n_{\varepsilon})^{-4\alpha}\right)^{\frac{1}{2}}\left(\int_{\Omega}|\nabla c_{\varepsilon}|^4\right)^{\frac{1}{2}}~~\mbox{for all}~~t>t_{**}}\\
\end{array}
\label{3333cz2.5114114}
\end{equation}
by using \dref{x1.73142vghf48gg} and \dref{1.ffggvbbnxxccvvn1}. Without loss of generality, we may assume that $\alpha<\frac{17}{12}$, since $\alpha\geq\frac{17}{12}$ can be proved similarly and easily.
Then the Gagliardo--Nirenberg inequality as well as  \dref{3222566311111hhxxcdfvhhhvffgggsssssssssddfffcccjjghjjsdgggdssdddgdddffddfddddffddffssddcssdz2.5} and \dref{ddfgczhhhh2.5ghju48cfg924ghyuji} enables us to see that
$$
\begin{array}{rl}
&\disp\frac{{p}-1}{2}
\left(\int_{\Omega}n^{{{2p}}}_{\varepsilon}(1+n_{\varepsilon})^{-4\alpha}\right)^{\frac{1}{2}}\left(\int_{\Omega}|\nabla c_{\varepsilon}|^4\right)^{\frac{1}{2}}
\\
\leq&\disp\frac{{p}-1}{2}
\left(\int_{\Omega}n^{{{2p}-4\alpha}}_{\varepsilon}\right)^{\frac{1}{2}}C_{1,**}^2e^{-2\rho_{1,**} t}
\\
\leq&\disp{C_{1,**}^2\frac{{p}-1}{2}\|  n^{\frac{p}{2}}_{\varepsilon}\|^{\frac{2p-4\alpha}{p}}_{L^{\frac{4p-8\alpha}{p}}(\Omega)}}
\\
\leq&\disp{C_{3,**}(\|\nabla n^{\frac{p}{2}}_{\varepsilon}\|_{L^2(\Omega)}^{\frac{6p-12\alpha-3}{3p-1}}\| n^{\frac{p}{2}}_{\varepsilon}\|_{L^{\frac{2}{p}}(\Omega)}^{\frac{2(p-1)}{p}-\frac{6p-12\alpha-3}{3p-1}}
+\|n^{\frac{p}{2}}_{\varepsilon}\|_{L^\frac{2}{p}(\Omega)}^{\frac{2(p-1)}{p}})}\\
\leq&\disp{C_{4,**}(\|\nabla n^{\frac{p}{2}}_{\varepsilon}\|_{L^2(\Omega)}^{\frac{6p-12\alpha-3}{3p-1}}+1)}~~~~\mbox{for all}~~t>t_{**}
\end{array}
$$
and some positive constants  $C_{3,**}$ and $C_{4,**}$, where $C_{1,**}$ and $\rho_{1,**}$ are the same as \dref{3222566311111hhxxcdfvhhhvffgggsssssssssddfffcccjjghjjsdgggdssdddgdddffddfddddffddffssddcssdz2.5}
Here we use the Young inequality to obtain  that there is $C_{5,**}>0$ such that
\begin{equation}
\begin{array}{rl}
\disp{\frac{1}{p}\frac{d}{dt}\int_{\Omega}n^{{{p}}}_{\varepsilon}+\frac{({{p}-1})}{2}\int_{\Omega}n^{{{p}-2}}_{\varepsilon}|\nabla n_{\varepsilon}|^2}
\leq&\disp{C_{5,**}~~\mbox{for all}~~t>t_{**}.}\\
\end{array}
\label{3333cz2.5ssdffggg114114}
\end{equation}
Next, in view of \dref{ddfgczhhhh2.5ghju48cfg924ghyuji}, we derive from the Gagliardo--Nirenberg inequality  that for some positive constants $C_{6,**}$ and $C_{7,**}$ such that
$$
\begin{array}{rl}
&\disp \|n_\varepsilon \|^{{p}}_{L^{{p}}(\Omega)}\\
=& \| n_\varepsilon  ^{\frac{p}{2}}\|_{L^{2}(\Omega)}^{2}\\
\leq& C_{6,**}(\| \nabla n_\varepsilon  ^{\frac{p}{2}}\|_{L^{2}(\Omega)}^{\frac{6(p-1)}{3p-1}} \| n_\varepsilon  ^{\frac{p}{2}}\|_{L^{\frac{2}{p}}(\Omega)}^{2-\frac{6(p-1)}{3p-1}}+\| n_\varepsilon  ^{\frac{p}{2}}\|_{L^{\frac{2}{p}}(\Omega)}^{2})\\
\leq& C_{7,**}(\| \nabla n_\varepsilon  ^{\frac{p}{2}}\|_{L^{2}(\Omega)}^{\frac{6(p-1)}{3p-1}}+1)~~~\mbox{for all}~~t>t_{**},\\
\end{array}
$$
which together with the Young inequality  implies that
\begin{equation}\label{11113333ssddczssddd2.5kke3uuii45677ddff89001214114114rrggjjkk}
\begin{array}{rl}
\disp \|n_\varepsilon \|^{{p}}_{L^{{p}}(\Omega)}
\leq& \disp{\| \nabla n_\varepsilon  ^{\frac{p}{2}}\|_{L^{2}(\Omega)}^{2}+C_{8,**}}\\
=& \disp{\frac{p^2}{4}\int_{\Omega} n_\varepsilon  ^{{{{p}-2}}}|\nabla n_\varepsilon|^2+C_{8,**}}~~~~\mbox{for all}~~t>t_{**}.\\
\end{array}
\end{equation}
with some positive constant $C_{8,**}$.
Inserting \dref{11113333ssddczssddd2.5kke3uuii45677ddff89001214114114rrggjjkk} into \dref{3333cz2.5ssdffggg114114}, there exist positive constants $C_{9,**}$ and $C_{10,**}$ such that
\begin{equation}
\begin{array}{rl}
\disp{\frac{d}{dt}\int_{\Omega}n^{{{p}}}_{\varepsilon}+C_{9,**}\int_{\Omega}n^{{{p}}}_{\varepsilon}}
\leq&\disp{C_{10,**}~~\mbox{for all}~~t>t_{**}}\\
\end{array}
\label{3333cz2.5ssdffggg11sddd4114}
\end{equation}
by some basic calculation. Finally,  employing the Grownwall lemma to \dref{3333cz2.5ssdffggg11sddd4114}, we can obtain \dref{ddxxxcvvddcvdhjjjjdffbbggddczv.5ghcfgssddd924ghyuji}
 by using \dref{223311111hhxxcdfvhhhvsssssssssddfffcccjjgssddhjjsdggdffgggddddddgdddffddfddddffddffssddcssdz2.5}.
\end{proof}
\begin{lemma}\label{xccffgghhlemma4563025xxhjhjuioookloghyui}
Under the assumptions of Theorem \ref{thaaaeorem3}, one can find that $T_{*,6}>2$ and $C_{*,6}> 0$
   such that for all $t\geq T_{*,6}$ %
   and $\varepsilon\in(0,1)$,
\begin{equation}
\|c_{\varepsilon}(\cdot, s)\|_{W^{1,\infty}(\Omega)}\leq C_{*,6}.
\label{ddxxxcvvddcvdhjjjdddjdffbbggddczv.5ghcfg9ssddff24ghyuji}
\end{equation}
\end{lemma}
\begin{proof}
Firstly, 
applying {Lemma} \ref{11aaalemdfghkkmaddffffdfffgg4sssdddd5630} as well as Lemmas \ref{aaalemmaddffffsddddfffgg4sssdddd5630}--\ref{x344ddfgghhjjccffggfhhlemma45625xxhjhjuioookloghyui},  one can derive that for some positive constants $T_{1,***}>2$ and $C_{1,***}$ such that
\begin{equation}
\begin{array}{rl}
\|c_{\varepsilon}(\cdot, s)\|_{W^{1,4}(\Omega)}+\|c_{\varepsilon}(\cdot, s)\|_{L^{\infty}(\Omega)}+\|u_{\varepsilon}(\cdot, s)\|_{L^{q}(\Omega)}+\|n_{\varepsilon}(\cdot, s)\|_{L^{\frac{10}{3}}(\Omega)}\leq  C_{1,***}~~ \mbox{for all}~ s\geq T_{1,***}\\
\end{array}
\label{cz2.5715jkkjjkkkkcvsdffccvvhj4456777jjkddfffffkhhgll}
\end{equation}
and $q>2.$ Here we have used the Sobolev embedding $W^{1,4}(\Omega)\hookrightarrow L^{\infty}(\Omega)$.
To prove the boundedness of $\|\nabla c_{\varepsilon}(\cdot, t)\|_{L^\infty(\Omega)}$,
we rewrite the variation-of-constants formula for $c_{\varepsilon}$ in the form
$$c_{\varepsilon}(\cdot, t) = e^{(t-T_{1,***})(\Delta-1)}c_\varepsilon(\cdot,T_{1,***}) +\int_{T_{1,***}}^{t}e^{(t-s)(\Delta-1)}[n_{\varepsilon}(\cdot, s)-u_{\varepsilon}(\cdot, s)\cdot\nabla  c_{\varepsilon}(\cdot, s)]ds~ \mbox{for all}~ t>2T_{1,***}.$$
Now, we choose $\theta\in(\frac{1}{2}+\frac{3}{2}\times\frac{3}{10},1),$ 
 then the domain of the fractional power $D((-\Delta + 1)^\theta)\hookrightarrow W^{1,\infty}(\Omega)$ (see \cite{Horstmann791}). Hence, in view of $L^p$--$L^q$ estimates associated with the heat semigroup,  \dref{cz2.5715jkkjjkkkkcvsdffccvvhj4456777jjkddfffffkhhgll}, we derive  that there exist positive constants $\lambda_{1,***}$, $C_{2,***}$, $C_{3,***}$, $C_{4,***}$, and $C_{5,***}$ such that
\begin{equation}
\begin{array}{rl}
&\| c_{\varepsilon}(\cdot, t)\|_{W^{1,\infty}(\Omega)}\\
\leq&\disp{C_{2,***}\|(-\Delta+1)^\theta c_{\varepsilon}(\cdot, t)\|_{L^{\frac{10}{3}}(\Omega)}}\\
\leq&\disp{C_{3,***}(t-T_{1,***})^{-\theta}e^{-\lambda_{1,***} (t-T_{1,***})}\|c_\varepsilon(\cdot,T_{1,***})\|_{L^{\frac{10}{3}}(\Omega)}}\\
&\disp{+C_{3,***}\int_{T_{1,***}}^t(t-s)^{-\theta}e^{-\lambda_{1,***}(t-s)}
\|n_{\varepsilon}(\cdot,s)-u_{\varepsilon}(\cdot,s) \cdot \nabla c_{\varepsilon}(\cdot,s)\|_{L^{\frac{10}{3}}(\Omega)}ds}\\
\leq&\disp{C_{4,***}+C_{4,***}\int_{T_{1,***}}^t(t-s)^{-\theta}e^{-\lambda_{1,***}(t-s)}[\|n_{\varepsilon}(\cdot,s)\|_{L^{\frac{10}{3}}(\Omega)}+
\|u_{\varepsilon}(\cdot,s)\|_{L^{20}(\Omega)}
\|\nabla c_{\varepsilon}(\cdot,s)\|_{L^4(\Omega)}]ds}\\
\leq&\disp{C_{5,***}~ \mbox{for all}~ t>2T_{1,***}.}\\
\end{array}
\label{zjccffgbhjcvvvbscz2.5297x96301ku}
\end{equation}
Here we have used H\"{o}lder's inequality as well as
$$\int_{T_{1,***}}^t(t-s)^{-\theta}e^{-\lambda_{1,***}(t-s)}\leq \int_{0}^{\infty}\sigma^{-\theta}e^{-\lambda_{1,***}\sigma}d\sigma<+\infty$$
and
$t-T_{1,***}>T_{1,***}>1.$
Finally, collecting \dref{cz2.5715jkkjjkkkkcvsdffccvvhj4456777jjkddfffffkhhgll} and \dref{zjccffgbhjcvvvbscz2.5297x96301ku} imply
\dref{ddxxxcvvddcvdhjjjdddjdffbbggddczv.5ghcfg9ssddff24ghyuji} and
thereby completes the proof of Lemma \ref{xccffgghhlemma4563025xxhjhjuioookloghyui}.
\end{proof}

\begin{lemma}\label{x344ccffggfhhlemma45625xxhjhjuioookloghyui}
Under the assumptions of Theorem \ref{thaaaeorem3}, one can find that $T_{*,7}> 2$ and $C_{*,7}> 0$  with the property that for all $\varepsilon\in (0, 1)$ we have
\begin{equation}
\|n_{\varepsilon}(\cdot,t)\|_{L^{\infty}(\Omega)}\leq C_{*,7}~~\mbox{for all}~~ t\geq T_{*,7}.
\label{ddxxxcvvddcvdhjjjjdffbbggddczv.5ghcfgssddd924ghyuji}
\end{equation}
\end{lemma}
\begin{proof}
In view of
 Lemma  \ref{11aaalemdfghkkmaddffffdfffgg4sssdddd5630},
Lemma \ref{11x344ccffggfhhlemma45625xxhjhjuioookloghyui} and Lemma \ref{xccffgghhlemma4563025xxhjhjuioookloghyui}, we can pick  $T_{1,****} > 2$ as well as
 $\rho_{1,****}\in(0,1)$ and $C_{1,****}>0$ 
such that for any $t\geq T_{1,****} $
 \begin{equation}\|u_{\varepsilon}\|_{L^\infty((t,+\infty);L^p(\Omega))}< C_{1,****}~~~\mbox{for all}~~p>1
\label{3311111hhxxcdfvhssddffhhvsssssssssddfffcccjjghjjsdgggddddddgdddffddfddddffddffssddcssdz2.5}
\end{equation}
as well as
 \begin{equation}\|n_{\varepsilon}(\cdot,t)-\bar{n}_0\|_{L^\infty((t,+\infty);L^2(\Omega))}<C_{1,****}
\label{11111hhxxcdfvhhhvssssssghhhhsssddfffcccjjghjjsdgggddddgdddffddfddddffddffssddcssdz2.5}
\end{equation}
and
 \begin{equation}\|\nabla c_\varepsilon\|_{L^\infty((t,+\infty);L^\infty(\Omega))}<C_{1,****}.
\label{3311111hhxxcdfvhhhvffgggsssssssssddfffcccjjghjjsdgggdssdddgdddffddfddddffddffssddcssdz2.5}
\end{equation}
Next, 
let 
$\tilde{h}_{\varepsilon} :=F_{\varepsilon}(n_{\varepsilon})S_{\varepsilon} (x, n_{\varepsilon} , c_{\varepsilon} )\nabla c_{\varepsilon} +u_{\varepsilon} $. In view of 
\dref{x1.73142vghf48gg}, \dref{1.ffggvbbnxxccvvn1} as well as
\dref{3311111hhxxcdfvhssddffhhvsssssssssddfffcccjjghjjsdgggddddddgdddffddfddddffddffssddcssdz2.5} and \dref{3311111hhxxcdfvhhhvffgggsssssssssddfffcccjjghjjsdgggdssdddgdddffddfddddffddffssddcssdz2.5},  there exists $C_{2,****} > 0$ such that
\begin{equation}
\begin{array}{rl}
\|\tilde{h}_{\varepsilon} (\cdot, t)\|_{L^{\frac{190}{3}}(\Omega)}\leq&\disp{C_{2,****}~~ \mbox{for all}~~ t\geq T_{1,****}.}\\
\end{array}
\label{cz2ddff.57151ccvhhjjjkkkuuifghhhivhccvvhjjjkkhhggjjllll}
\end{equation}
Hence, due to the fact that $\nabla\cdot u_{\varepsilon} =0$,  again,  by means of an
associate variation-of-constants formula for $n_{\varepsilon} $, we can derive
\begin{equation}
n_{\varepsilon} (\cdot,t)=e^{(t-T_{1,****})\Delta}n_{\varepsilon} (\cdot,T_{1,****})-\int_{T_{1,****}}^{t}e^{(t-s)\Delta}\nabla\cdot(n_{\varepsilon} (\cdot,s)\tilde{h}_{\varepsilon} (\cdot,s)) ds,~~ t>2T_{1,****}.
\label{5555fghbnmcz2.5ghjjjkkklu48cfg924ghyuji}
\end{equation}
Since, $T_{1,****}>1$, then in light of the  $L^p$-$L^q$ estimates for the Neumann heat semigroup and {Lemma} \ref{11aaalemdfghkkmaddffffdfffgg4sssdddd5630}, we conclude that for some positive constants $C_{3,****}$ as well as $C_{4,****}$ and $\tilde{\rho}_{2,****}$ such that for  all $t>2T_{1,****}$
\begin{equation}
\begin{array}{rl}
\|e^{(t-T_{1,****})\Delta}n_\varepsilon (\cdot,T_{1,****})\|_{L^{\infty}(\Omega)}\leq &\disp{C_{3,****}(t-T_{1,****})^{-\frac{3}{4}}e^{-\tilde{\rho}_{2,****} t}\|n_\varepsilon (\cdot,T_{1,****})\|_{L^{2}(\Omega)}\leq C_{4,****}.}\\
\end{array}
\label{zjccffgbhjffghhjcghghjkjjhhjjjvvvbscz2.5297x96301ku}
\end{equation}
In the follwing, we will estimate
the last term on the right-hand side of \dref{5555fghbnmcz2.5ghjjjkkklu48cfg924ghyuji}. In fact,
applying the smoothing
properties of the
Stokes semigroup as well as \dref{3.10gghhjuuloollgghhhyhh} and the H\"{o}lder inequality to find $C_{5,****} > 0$, $C_{6,****} > 0$  and $\tilde{\lambda}_{1,****}$   such that
\begin{equation}
\begin{array}{rl}
&\disp\int_{T_{1,****}}^t\| e^{(t-s)\Delta}\nabla\cdot(n_\varepsilon (\cdot,s)\tilde{h}_\varepsilon (\cdot,s)\|_{L^\infty(\Omega)}ds\\
\leq&\disp C_{5,****}\int_{T_{1,****}}^t(t-s)^{-\frac{1}{2}-\frac{9}{19}}e^{-\tilde{\lambda}_{1,****}(t-s)}\|n_\varepsilon (\cdot,s)\tilde{h}_\varepsilon  (\cdot,s)\|_{L^{\frac{19}{6}}(\Omega)}ds\\
\leq&\disp C_{5,****}\int_{T_{1,****}}^t(t-s)^{-\frac{1}{2}-\frac{9}{19}}e^{-\tilde{\lambda}_{1,****}(t-s)}\| n_\varepsilon (\cdot,s)\|_{L^{\frac{10}{3}}(\Omega)}\|\tilde{h}_\varepsilon (\cdot,s)\|_{L^{\frac{190}{3}}(\Omega)}ds\\
\leq&\disp C_{6,****}~~\mbox{for all}~~ t>2T_{1,****},\\
\end{array}
\label{ccvbccvvbbnnndffghhjjvcvvbccfbbnfgbghjjccmmllffvvggcvvvvbbjjkkdffzjscz2.5297x9630xxy}
\end{equation}
where we have used the fact that
$$\int_{T_{1,****}}^t(t-s)^{-\frac{1}{2}-\frac{9}{19}}e^{-\tilde{\lambda}_{1,****}(t-s)}ds\leq \int_{1}^\infty\sigma^{-\frac{1}{2}-\frac{9}{19}}e^{-\tilde{\lambda}_{1,****}\sigma}d\sigma<+\infty.$$
In combination with \dref{5555fghbnmcz2.5ghjjjkkklu48cfg924ghyuji}--\dref{ccvbccvvbbnnndffghhjjvcvvbccfbbnfgbghjjccmmllffvvggcvvvvbbjjkkdffzjscz2.5297x9630xxy}, we may finally derive  
\begin{equation}
\begin{array}{rl}
\|n_\varepsilon (\cdot, t)\|_{L^{\infty}(\Omega)}\leq&\disp{C_{7,****}~ ~~ \mbox{for all}~~ t\geq 2T_{1,***}}\\
\end{array}
\label{cz2.57ghhhh151ccvhhjjjkkkffgghhuuiivhccvvhjjjkkhhggjjllll}
\end{equation}
by   some basic calculation.
This completes the proof of Lemma \ref{x344ccffggfhhlemma45625xxhjhjuioookloghyui}.
\end{proof}
Along the idea of the proof of Section  7.3 in \cite{Winklerpejoevsssdd793}, we could 
derive the H\"{o}lder regularity
 of the components of a solution on
intervals of the form $(T_0 ,T_0 + 1)$ for any large $T_0> 2$. For this purpose,
 we firstly introduce the following cut-off functions,  which will play a key role in deriving higher order regularity for solution of problem \dref{1.1fghyuisda}.
\begin{definition}\label{aaalemmaddffffdssfffgg4sssdddd5630}
Let $\xi: \mathbb{R} \rightarrow [0,1]$ be a smooth, monotone function, satisfying $\xi\equiv 0$ on $(-\infty,0]$ and
$\xi\equiv 0$ on $(1,\infty)$ and for any $t_0\in \mathbb{R}$ we let $\xi_{t_0}:= \xi (t -t_0 )$.
\end{definition}

Due to the  above cut-off function,
it follows from maximal Sobolev regularity that the solution $(n_\varepsilon, c_\varepsilon, u_\varepsilon)$ even satisfies estimates in
appropriate H\"{o}lder spaces:

\begin{lemma}\label{lemmddrgga45630hhuussddjjuuyy}
 Let $(n_\varepsilon, c_\varepsilon, u_\varepsilon)$ be a solution of \dref{1.1fghyuisda} as well as  \dref{x1.73142vghf48}  and \dref{x1.73142vgssdddhfjjk48} hold. Then for all
$p \geq 2$, there exist $T_{*,8} > 0$ and $C_{*,8} > 0$ such that for all $\varepsilon\in (0, 1)$,
\begin{equation}
\|u_\varepsilon\|_{L^{p}((t,t+1);W^{2,p}(\Omega))}
+\|u_{\varepsilon t}\|_{L^{p}(\Omega\times(t,t+1))}\leq C_{*,8} ~~\mbox{for all}~~ t\geq T_{*,8}.
\label{zjscz2.5297x9dddjkkkkkd630111kkhhffssddrddroojj}
\end{equation}
\end{lemma}
\begin{proof}
Firstly, for any $p\geq2,$
we invoke
Lemmas \ref{aaalemmaddffffsddddfffgg4sssdddd5630} and \ref{x344ccffggfhhlemma45625xxhjhjuioookloghyui} to fix $\tilde{T}_{1,*}$ as well as  $\tilde{C}_{1,*} > 0$ and $\tilde{C}_{2,*}> 0$ such that with
 $q :=\max\{6p,\frac{6p}{2p-3} \}$ we have
 \begin{equation}\|n_{\varepsilon}\|_{L^\infty((t-1;t+1);L^{\frac{3p}{2}}(\Omega))}\leq \tilde{C}_{1,*}~~~\mbox{for all}~~ t > \tilde{T}_{1,*}
\label{11111hhxxcdfvhhhvsssssdffssssssddfffcccjjghjjsdgggddddgdddffddfddddffddffssddcssdz2.5}
\end{equation}
and
 \begin{equation}\|u_{\varepsilon}\|_{L^\infty((t-1;t+1);L^{q}(\Omega))}\leq \tilde{C}_{2,*}~~~\mbox{for all}~~ t > \tilde{T}_{1,*}.
\label{11111hhxxcdfvhhhvssddsssssssssddfffcccjjghjjsdgggddddddgdddffddfddddffddffssddcssdz2.5}
\end{equation}
For given $t_0 > \tilde{T}_{1,*},$ we let $\xi_{t_0}$ be as defined in {Definition} \ref{aaalemmaddffffdssfffgg4sssdddd5630}.
Then the function $v_\varepsilon : \Omega\times(t_0-1,\infty) \rightarrow \mathbb{R}^3$ defined by $v_\varepsilon(x,t) := \xi_{t_0} (t)u_\varepsilon(x,t), (x,t)\in
 \Omega\times (t_0-1,\infty)$, is a
weak solution in  $L^\infty_{loc}([t_0-1,\infty);L^2_\sigma(\Omega)) \cap L^2_{loc}([t_0-1,\infty);W^{1,2}_{0,\sigma}(\Omega) \cap W^{2,2} (\Omega))$ of
%
%
 \begin{equation}
\begin{array}{ll}
   v_{\varepsilon t}+A v_\varepsilon=h_\varepsilon~~~\mbox{in}~~\Omega\times (t_0-1,\infty),\\
 \end{array}
 \label{11111hhxxghhhjjcdfvhhhvssddsssssssssddfffcccjjghjjsdgggddddddgdddffddfddddffddffssddcssdz2.5}
\end{equation}
 where $$ v_\varepsilon(\cdot,t_0-1)\equiv0~~~\mbox{and}~~~h_\varepsilon=\xi_{t_0} (t)\mathcal{P}[-\kappa(Y_\varepsilon u_\varepsilon \cdot \nabla)u_\varepsilon+n_\varepsilon\nabla \phi]+\xi'_{t_0} u_\varepsilon.$$
  To estimate the inhomogeneity $h_\varepsilon$ herein,  we first note that  the known maximal Sobolev regularity estimate for the Stokes semigroup (\cite{Gigass12176}) yields that there exist positive
constants $\tilde{k}_{1,*} > 0, \tilde{k}_{2,*}>0$ and $\tilde{k}_{3,*} > 0$  such that
\begin{equation}
\begin{array}{rl}
&\disp \int_{t_0-1}^{t_0+1}\|\xi_{t_0}(t)\mathcal{P}[n_\varepsilon(\cdot,t)\nabla\phi]\|_{L^{\frac{3p}{2}}(\Omega)}^{\frac{3p}{2}}dt\\
\leq&\disp \tilde{k}_{1,*}\int_{t_0-1}^{t_0+1}\|n_\varepsilon(\cdot,t)\nabla\phi\|_{L^{\frac{3p}{2}}(\Omega)}^{\frac{3p}{2}}\\
\leq&\disp \tilde{k}_{1,*}\|\nabla\phi\|_{L^\infty(\Omega)}\int_{t_0-1}^{t_0+1}\|n_\varepsilon(\cdot,t)\|_{L^{\frac{3p}{2}}(\Omega)}^{\frac{3p}{2}}dt\\
\leq&\disp \tilde{k}_{2,*}\\
\end{array}
\label{cz2.571hhhhh51lllllccvvhsdddddfccvvhjjjkkhhggjjlsdddlll}
\end{equation}
and
\begin{equation}
\begin{array}{rl}
\disp \int_{t_0-1}^{t_0+1}\|\xi'_{t_0} u_\varepsilon(\cdot,t)\|_{L^{\frac{3p}{2}}(\Omega)}^{\frac{3p}{2}}dt
\leq\disp \tilde{k}_{3,*}.\\
\end{array}
\label{cz2.571hhhhh51lllllccvvhsssdddddddfccvvhjjjkkhhggjjlsdddlll}
\end{equation}
From the boundedness of the Helmholtz projection in $L^p$-spaces and the H\"{o}lder inequality we derive from Lemma \ref{aaalemmaddffffsddddfffgg4sssdddd5630} that there exist positive constant  $\tilde{k}_{i,*}, i=4,\ldots,8$  such that
 for any $ t_0 > \tilde{T}_{1,*}$
\begin{equation}
\begin{array}{rl}
&\disp \int_{t_0-1}^{t_0+1}\|\xi_{t_0}\mathcal{P}[Y_\varepsilon u_\varepsilon\cdot\nabla u_\varepsilon](\cdot,t)\|_{L^{\frac{3p}{2}}(\Omega)}^{\frac{3p}{2}}dt\\
\leq&\disp \tilde{k}_{4,*}\int_{t_0-1}^{t_0+1}\|\nabla u_\varepsilon(\cdot,t)\|_{L^{2p}(\Omega)}^{\frac{3p}{2}}\| u_\varepsilon(\cdot,t)\|_{L^{6p}(\Omega)}^{\frac{3p}{2}}dt\\
\leq&\disp \tilde{k}_{5,*}\int_{t_0-1}^{t_0+1}\|\nabla u_\varepsilon(\cdot,t)\|_{L^{2p}(\Omega)}^{\frac{3p}{2}}dt\\
\leq&\disp \tilde{k}_{6,*}\int_{t_0-1}^{t_0+1}\| u_\varepsilon(\cdot,t)\|_{W^{2,p}(\Omega)}^{p}\|u_\varepsilon(\cdot,t)\|_{L^{\frac{6p}{2p-3}}(\Omega)}^{\frac{p}{2}}dt\\
\leq&\disp \tilde{k}_{7,*}\int_{t_0-1}^{t_0+1}\| u_\varepsilon(\cdot,t)\|_{W^{2,p}(\Omega)}^{p}dt\\
\leq&\disp{\tilde{k}_{8,*}.}\\
\end{array}
\label{cz2.57ssdd1hhhhh51lllllccvvhsddddddfdfccvvhjjjkkhhggjjlsdddlll}
\end{equation}
As a consequence of \dref{cz2.571hhhhh51lllllccvvhsdddddfccvvhjjjkkhhggjjlsdddlll}--\dref{cz2.57ssdd1hhhhh51lllllccvvhsddddddfdfccvvhjjjkkhhggjjlsdddlll}, once more by maximal Sobolev regularity estimates, now applied to \dref{11111hhxxghhhjjcdfvhhhvssddsssssssssddfffcccjjghjjsdgggddddddgdddffddfddddffddffssddcssdz2.5}, we obtain $k_{9,****} > 0$ and $k_{10,****} > 0$ satisfying
\begin{equation}
\begin{array}{rl}
&\disp \int_{t_0}^{t_0+1}\|u_\varepsilon(\cdot,t)\|_{W^{2,\frac{3p}{2}}(\Omega)}^{\frac{3p}{2}}dt+
\int_{t_0}^{t_0+1}\|v_{\varepsilon t}(\cdot,t)\|_{W^{2,\frac{3p}{2}}(\Omega)}^{\frac{3p}{2}}dt\\
\leq&\disp \int_{t_0-1}^{t_0+1}\|u_\varepsilon(\cdot,t)\|_{W^{2,\frac{3p}{2}}(\Omega)}^{\frac{3p}{2}}dt+
\int_{t_0-1}^{t_0+1}\|v_{\varepsilon t}(\cdot,t)\|_{W^{2,\frac{3p}{2}}(\Omega)}^{\frac{3p}{2}}dt\\
\leq&\disp \tilde{k}_{9,*}\int_{t_0-1}^{t_0+1}\|h_\varepsilon(\cdot,t)\|_{L^{\frac{3p}{2}}(\Omega)}^{\frac{3p}{2}}dt\\
\leq&\disp \tilde{k}_{10,*}.\\
\end{array}
\label{cz2.571hhhhh51lllllccvvhsdddddfccvvhssdddjjjkkhhggjjlsdddlll}
\end{equation}
Therefore, by the definition of $\xi$, for any $p> 1$, there exist positive constants $T_{*,8}  > 0$ and $C_{*,8}  > 0$ such that for any
$t > T_{*,8} $ and for all $\varepsilon\in (0, 1)$, 
\begin{equation}
\|u_\varepsilon(\cdot,t)\|_{L^{p}((t,t+1);W^{2,p}(\Omega))} +\|u_{\varepsilon t}(\cdot,t)\|_{L^{p}(\Omega\times(t,t+1))}\leq C_{*,8} .
\label{zjscz2.5297x9dddjkkkkkd630111kkhhffrddroojj}
\end{equation}
This establishes \dref{zjscz2.5297x9dddjkkkkkd630111kkhhffssddrddroojj} and thereby completes the proof.
\end{proof}

\begin{lemma}\label{lemmassdd4563sdff0hhuujjuuyy}
Let $\alpha\geq1$ and \dref{x1.73142vgssdddhfjjk48} hold.
Then one can find $\mu_{*,9}\in(0, 1)$ and $T_{*,9}>2$ and $C_{*,9}>0$ such that for each $\varepsilon\in (0, 1)$ and $t>T_{*,9}$  we have
%
%
\begin{equation}
\|u_\varepsilon(\cdot,t)\|_{C^{1+\mu_{*,9},\frac{\mu_{*,9}}{2}}(\bar{\Omega}\times[t,t+1]);\mathbb{R}^3)}  \leq C_{*,9}.
\label{zjscz2.5297x9630111kkffghhhhiioo}
\end{equation}
\end{lemma}
\begin{proof}
Using sufficiently high values of $p$, in view of Lemma \ref{lemma45eertt630hhuussddjjuuyy}, an application of the embedding result in \cite{AmannAmannmo1216} refines this into the
assertion on H\"{o}lder continuity.
\end{proof}

Applying a similar reasoning (see the proof of Lemma \ref{lemmddrgga45630hhuussddjjuuyy}), concerning $c_\varepsilon$ we obtain bounds of the same kind by using Lemmas \ref{x344ccffggfhhlemma45625xxhjhjuioookloghyui} and \ref{lemmassdd4563sdff0hhuujjuuyy}--\ref{xccffgghhlemma4563025xxhjhjuioookloghyui}.
%
\begin{lemma}\label{lemma45eertt630hhuussddjjuuyy}
 Let $(n_\varepsilon,c_\varepsilon,u_\varepsilon)$ be a solution of \dref{1.1fghyuisda}. If \dref{x1.73142vghf48}  and \dref{x1.73142vgssdddhfjjk48} holds, then for all
$p >1$, there exist $T_{*,10}>2$ and $C_{*,10}>0$ such that for any $\varepsilon\in (0, 1)$ and $t\geq T_{*,10}$ we have
\begin{equation}
\|c_\varepsilon(\cdot,t)\|_{L^{p}((t,t+1);W^{2,p}(\Omega))}
+\|c_{\varepsilon t}(\cdot,t)\|_{L^{p}(\Omega\times(t,t+1))}\leq C_{*,10}.
\label{zjscz2.5297x9dddjkkkkkd630111kksddddhhffssddrddroojj}
\end{equation}
\end{lemma}
\begin{proof}
Firstly, applying Lemmas \ref{x344ccffggfhhlemma45625xxhjhjuioookloghyui}, \ref{xccffgghhlemma4563025xxhjhjuioookloghyui} and \ref{lemmassdd4563sdff0hhuujjuuyy},  there exist $\tilde{T}_{1,**} > 2$
 and positive constants  $\tilde{C}_{1,**}$ as well as $\tilde{C}_{2,**}$ and $\tilde{C}_{3,**}$ such that
\begin{equation}
\|n_\varepsilon\|_{L^{2p}(\Omega\times(t_0-1,t_0+1))}
\leq \tilde{C}_{1,**} ~~\mbox{for all}~~ t_0\geq \tilde{T}_{1,**}
\label{zjscz2.5297x9dddjkkkkkd630111kksddddhddddhffssddrddroojj}
\end{equation}
as well as
\begin{equation}
\|u_{\varepsilon }\|_{L^{\infty}((t_0-1,t_0+1);L^{\infty}(\Omega))}\leq \tilde{C}_{2,**}~~\mbox{for all}~~ t_0\geq \tilde{T}_{1,**}
\label{zjscz2.5297x9dddjkkkkkd630111kksddddhdddhffssddrddroojj}
\end{equation}
and
\begin{equation}
\|\nabla c_\varepsilon\|_{W^{1,2p}(\Omega\times(t_0-1,t_0+1))}
\leq \tilde{C}_{3,**} ~~\mbox{for all}~~ t_0\geq \tilde{T}_{1,**}.
\label{zjscz2.5297x9dddjkkkkkd630111kksddddhhffssddddddrddroojj}
\end{equation}
Now for fixed $t_0 > \tilde{T}_{1,**}$ we let $\xi_{t_0}$ be as given by {Definition} \ref{aaalemmaddffffdssfffgg4sssdddd5630}
 and consider 
the following problem
\begin{equation}
 \left\{\begin{array}{ll}
   \tilde{v}_{\varepsilon t}-\Delta \tilde{v}_\varepsilon=\tilde{h}_\varepsilon~~~~x\in\Omega, t>t_0-1,\\
\disp\frac{\partial \tilde{v}_\varepsilon}{\partial\nu}=0,~~~x\in\partial\Omega, t>t_0+1,\\
 \disp{\tilde{v}_\varepsilon(x,t_0-1)=0},\quad
x\in \Omega,\\
 \end{array}\right.
\label{222zjscz2.52ssdd9ddff7x96301sddd11kddffkhhffrssddroojj}
\end{equation}
where
$$\tilde{h}_\varepsilon(\cdot,t):=\xi_{t_0}(t)(-c_\varepsilon+n_\varepsilon -u_\varepsilon\cdot\nabla c_\varepsilon)+ \xi'_{t_0}(t)c_\varepsilon.$$
Here using \dref{zjscz2.5297x9dddjkkkkkd630111kksddddhddddhffssddrddroojj}--\dref{zjscz2.5297x9dddjkkkkkd630111kksddddhhffssddddddrddroojj} we see that for some $\tilde{C}_{4,**}> 0$ and $\tilde{C}_{5,**}>0$ such that
\begin{equation}
\begin{array}{rl}
&\disp{\int_{t_0-1}^{t_0+1}\|\tilde{h}_\varepsilon(\cdot,t)\|_{L^{2p}(\Omega)}^{2p}dt}\\
\leq&\disp{\int_{t_0-1}^{t_0+1}\|\xi_{t_0}(t)(-c_\varepsilon+n_\varepsilon -u_\varepsilon\cdot\nabla c_\varepsilon)+ \xi'_{t_0}(t)c_\varepsilon\|_{L^{2p}(\Omega)}^{2p}dt}\\
\leq&\disp{\tilde{C}_{4,**}\int_{t_0-1}^{t_0+1}\left[\|c_\varepsilon(\cdot,t)\|_{L^{2p}(\Omega)}^{2p}+\|n_\varepsilon(\cdot,t)\|_{L^{2p}(\Omega)}^{2p}\right]dt}\\
&\disp{+\tilde{C}_{4,**}\int_{t_0-1}^{t_0+1}\left[\|u_\varepsilon(\cdot,t)\|_{L^{\infty}(\Omega)}^{2p}\|\nabla c_\varepsilon(\cdot,t)\|_{L^{2p}(\Omega)}^{2p}+2 \|\xi'_{0}\|_{L^\infty(\mathbb{R})}^{2p}\|c_\varepsilon(\cdot,t)\|_{L^{2p}(\Omega)}^{2p}\right]dt}\\
\leq&\disp{\tilde{C}_{5,**}.}\\
\end{array}
\label{zjccffgbhjssddddcvvvbscz2.5297x96301ku}
\end{equation}
Therefore, we may argue as above to conclude from maximal Sobolev regularity estimates that
there exist $\tilde{C}_{6,**} > 0$ and $\tilde{C}_{7,**} > 0$ such that
\begin{equation}
\begin{array}{rl}
&\disp\int_{t_0-1}^{t_0+1}\|\tilde{v}_\varepsilon(\cdot,t)\|_{W^{2,p}(\Omega))}^{2p}dt
+\|\tilde{v}_{\varepsilon t}(\cdot,t)\|_{L^{2p}(\Omega)}^{2p}dt\\
\leq& \tilde{C}_{6,**} \disp\int_{t_0-1}^{t_0+1}\|\tilde{h}_\varepsilon(\cdot,t)\|_{L^{2p}(\Omega))}^{2p}dt \\
\leq& \tilde{C}_{7,**}.\\
\end{array}\label{zjscz2.5297x9dddjkkkkkd63011sddf1kkhhffssddrddroojj}
\end{equation}
Again since $\tilde{v}_\varepsilon = c_\varepsilon$ a.e. in $\Omega\times (t_0 ,t_0 + 1)$ by definition of $\xi_{t_0}$, this shows \dref{zjscz2.5297x9dddjkkkkkd630111kksddddhhffssddrddroojj} and
thereby completes the proof of Lemma \ref{lemma45eertt630hhuussddjjuuyy}.
\end{proof}
Again, this implies the H\"{o}lder estimates as follows.
\begin{lemma}\label{11lemma4563sdff0hhuujjuuyy}
Suppose that $\alpha$ and $C_S$ satisfy  \dref{x1.73142vghf48}  and \dref{x1.73142vgssdddhfjjk48}, respectively, where $C_S$ is the same as  \dref{x1.73142vghf48gg}.  
Then one can find $\mu_{*,11}\in(0, 1)$ as well as  $T_{*,11}>2$ and $C_{*,11}>0$ such that for all $\varepsilon\in(0,1)$ and $t>T_{*,11}$,
%
%
\begin{equation}
\|c_\varepsilon(\cdot,t)\|_{C^{1+\mu_{*,11},\frac{\mu_{*,11}}{2}}(\bar{\Omega}\times[t,t+1])}  \leq C_{*,11}.
\label{zjscz2.5297x9630111kkffghhhhiioo}
\end{equation}
\end{lemma}
\begin{proof}
In precisely the same manner as Lemma  \ref{lemmassdd4563sdff0hhuujjuuyy} was derived from Lemma
\ref{lemmddrgga45630hhuussddjjuuyy}, this follows from Lemma \ref{lemma45eertt630hhuussddjjuuyy} by application of a standard embedding result.
\end{proof}
Using once more Lemma \ref{x344ccffggfhhlemma45625xxhjhjuioookloghyui} and Lemma \ref{lemmaghjffggsjjjjjsddgghhmk4563025xxhjklojjkkk}, we can improve our knowledge on the regularity of $n_\varepsilon$.
\begin{lemma}\label{lemma45630hhuujjuuyy}
Under the assumptions of Theorem \ref{thaaaeorem3},
then one can find $\mu_{*,12}\in(0, 1)$ as well as  $T_{*,12}>0$ and $C_{*,12}>0$ such that whenever $\varepsilon\in (0, 1)$, we have
%
\begin{equation}
\|n_\varepsilon(\cdot,t)\|_{C^{\mu_{*,12},\frac{\mu_{*,12}}{2}}(\bar{\Omega}\times[t,t+1])}  \leq C_{*,12} ~~\mbox{for all}~~ t>T_{*,12}.
\label{ddhhhzjscz2.5297x9630111kkhhiioo}
\end{equation}
\end{lemma}
\end{lemma}

\begin{proof}
Firstly,  in light of Lemma \ref{x344ccffggfhhlemma45625xxhjhjuioookloghyui} as well as  Lemma \ref{11lemma4563sdff0hhuujjuuyy} and Lemma  \ref{lemmassdd4563sdff0hhuujjuuyy}, there exist
 $\tilde{\mu}_{1,***}\in(0,1)$ as well as $\tilde{T}_{1,***} > 2$ and $\tilde{C}_{1,***} > 0$ such that  for all $t>\tilde{T}_{1,***}$,
\begin{equation}
\begin{array}{rl}
&\|\nabla c_\varepsilon\|_{C^{\tilde{\mu}_{1,***},\frac{\tilde{\mu}_{1,***}}{2}}(\bar{\Omega}\times[t,t+1])}+\| n_\varepsilon\|_{L^{\infty}(\bar{\Omega}\times[t,t+1])}
 +\| u_\varepsilon\|_{C^{\tilde{\mu}_{1,***},\frac{\tilde{\mu}_{1,***}}{2}}(\bar{\Omega}\times[t,t+1])}
 \leq \tilde{C}_{1,***}.
\end{array}\label{222zjscz2.52ssdd97x9630111fssddfggkkhhffrroojj}
\end{equation}
In order to  derive \dref{ddhhhzjscz2.5297x9630111kkhhiioo}, we may now use that $\nabla\cdot(n_\varepsilon u_\varepsilon) = u_\varepsilon\cdot\nabla n_\varepsilon$ thanks to
the fact that $\nabla u_\varepsilon\equiv 0$, we have
\begin{equation}
\begin{array}{rl}
&n_{\varepsilon}(\cdot,t)\\
=&e^{(t-\tilde{T}_{1,***})\Delta}n_\varepsilon(\cdot,\tilde{T}_{1,***})\\
&\disp+\int_{\tilde{T}_{1,***}}^{t}e^{(t-s)\Delta}\nabla\cdot\{n_{\varepsilon}(\cdot,s)
F_{\varepsilon}(n_{\varepsilon}(\cdot,s))S_\varepsilon(x,n_\varepsilon(\cdot,s),c_\varepsilon(\cdot,s))\nabla c_\varepsilon(\cdot,s)+n_{\varepsilon}(\cdot,s) u_{\varepsilon}(\cdot,s)\} ds\\
\end{array}\label{5555hhjjjfghsdddffdfggbnmcz2.5ghju48cfg924ghyuji}
\end{equation}
for all $t>\tilde{T}_{1,***}$.
Therefore, in view of  \dref{222zjscz2.52ssdd97x9630111fssddfggkkhhffrroojj} and the  standard regularity arguments involving well-known smoothing properties of the Neumann heat semigroup $(e^{\tau\Delta})_{\tau\geq0}$ (see e.g.  \cite{Winkler21215}),
 one can readily derive the existence of $\mu_{*,12}\in(0, 1)$ as well as  $T_{*,12}$ and $C_{*,12}>0$
such that for any  $\varepsilon>0$, 
%
\dref{ddhhhzjscz2.5297x9630111kkhhiioo} holds.
\end{proof}


Straightforward applications of standard Schauder estimates for the Stokes evolution equation and the
heat equation, respectively, finally yield eventual smoothness of the solution $(n_\varepsilon, c_\varepsilon, u_\varepsilon)$.

\begin{lemma}\label{lemma45630hhuujjsdfffggguuyy}
Under the assumptions of Theorem \ref{thaaaeorem3},
then one can find $\mu_{*,13}\in(0, 1)$ and $T_{*,13}>2$ such that for some $C_{*,13} > 0$ with the property that for each $\varepsilon\in (0, 1)$,
\begin{equation}
\|u_\varepsilon(\cdot,t)\|_{C^{2+\mu_{*,13},1+\frac{\mu_{*,13}}{2}}(\bar{\Omega}\times[t,t+1];\mathbb{R}^3)} \leq C_{*,13} ~~\mbox{for all}~~ t>T_{*,13}
\label{222zjscz2.5297x9630111kkhhffrroojj}
\end{equation}
%
%
as well as
\begin{equation}
\|c_\varepsilon(\cdot,t)\|_{C^{2+\mu_{*,13},1+\frac{\mu_{*,13}}{2}}(\bar{\Omega}\times[t,t+1])}  \leq C_{*,13} ~~\mbox{for all}~~ t>T_{*,13}
\label{222zjscz2.5297x9630111kkhhiioo}
\end{equation}
and
\begin{equation}
\|n_\varepsilon(\cdot,t)\|_{C^{2+\mu_{*,13},1+\frac{\mu_{*,13}}{2}}(\bar{\Omega}\times[t,t+1])} \leq C_{*,13} ~~\mbox{for all}~~ t>T_{*,13}.
\label{222zjscz2.5297x96dfgg30111kkhhffrroojj}
\end{equation}
\end{lemma}
\begin{proof}
We first combine Lemma \ref{lemmassdd4563sdff0hhuujjuuyy} as well as Lemma \ref{11lemma4563sdff0hhuujjuuyy} and Lemma \ref{lemma45630hhuujjuuyy} to infer
the existence of $\tilde{\mu}_{1,****}\in(0,1)$ as well as $\tilde{T}_{1,****} > 0$ and $\tilde{C}_{1,****} > 0$ such that  for all $t>\tilde{T}_{1,****}$,
\begin{equation}
\begin{array}{rl}
&\|u_\varepsilon\cdot \nabla c_\varepsilon\|_{C^{\tilde{\mu}_{1,****},\frac{\tilde{\mu}_{1,****}}{2}}(\bar{\Omega}\times[t,t+1])}+\|u_\varepsilon\cdot \nabla n_\varepsilon\|_{C^{\tilde{\mu}_{1,****},\frac{\tilde{\mu}_{1,****}}{2}}(\bar{\Omega}\times[t,t+1])}\\
 &+\| n_\varepsilon\|_{C^{\tilde{\mu}_{1,****},\frac{\tilde{\mu}_{1,****}}{2}}(\bar{\Omega}\times[t,t+1])}+\| c_\varepsilon\|_{C^{\tilde{\mu}_{1,****},\frac{\tilde{\mu}_{1,****}}{2}}(\bar{\Omega}\times[t,t+1])}\\
 \leq & \tilde{C}_{1,****}.
\end{array}\label{222zjscz2.52ssdd97x9630111ffggkkhhffrroojj}
\end{equation}
The standard parabolic Schauder estimates applied to the second  equation in \dref{1.1fghyuisda} (see \cite{Ladyzenskajaggk7101})
thus provide $\tilde{C}_{2,****} > 0$ fulfilling
\begin{equation}
\|c_\varepsilon\|_{C^{2+\tilde{\mu}_{1,****} ,1+\frac{\tilde{\mu}_{1,****} }{2}}(\bar{\Omega}\times[t,t+1])}\leq \tilde{C}_{2,****} ~~~\mbox{for all}~~ t > \tilde{T}_{1,****}.
\label{222zjscz2.52ssdd97x9630111kddffkhhffrroojj}
\end{equation}
According to Lemma \ref{lemma45630hhuujjuuyy}, it is possible to fix $\tilde{\mu}_{2,****}\in (0,1)$ as well as
$\tilde{T}_{2,****}> 2$ and $\tilde{C}_{3,****}> 0$ such that
\begin{equation}
\|n_\varepsilon\|_{C^{1+\tilde{\mu}_{2,****},\frac{\tilde{\mu}_{2,****}}{2}}(\bar{\Omega}\times[t,t+1])}+\| u_\varepsilon\|_{C^{1+\tilde{\mu}_{2,****},\frac{\tilde{\mu}_{2,****}}{2}}(\bar{\Omega}\times[t,t+1];\mathbb{R}^3)}\leq \tilde{C}_{3,****}~~~\mbox{for all}~~ t > \tilde{T}_{2,****}.
\label{222zjscz2.52ssdd97x9630111kddffkhhffrssddroojj}
\end{equation}
We next set $T:= \tilde{T}_{2,****}+1$ and let $t_0 > T$ be given. Then with $\xi_{t_0}$ taken from Definition \ref{aaalemmaddffffdssfffgg4sssdddd5630},
we again use that $v_\varepsilon(\cdot,t) := \xi_{t_0}u_\varepsilon(\cdot,t), (x\in\Omega,t> t_0-1)$, is a solution of
\begin{equation}
 \left\{\begin{array}{ll}
   v_{\varepsilon t}-\Delta v_\varepsilon=h_\varepsilon~~~~x\in\Omega, t>t_0-1,\\
 \disp{v_\varepsilon(x,t_0-1)=0},\quad
x\in \Omega,\\
 \end{array}\right.
\label{222zjscz2.52ssdd9ddff7x9630111kddffkhhffrssddroojj}
\end{equation}
 where $$h_\varepsilon=-\kappa(Y_\varepsilon u_\varepsilon \cdot \nabla)v+\nabla (\xi P_\varepsilon)+\xi n_\varepsilon\nabla \phi+\xi' u_\varepsilon.$$
 Now from \dref{222zjscz2.52ssdd97x9630111kddffkhhffrssddroojj} and the smoothness of $\xi$ we readily obtain $\tilde{\mu}_{3,****}\in (0,1)$
and $\tilde{C}_{4,****}> 0$ fulfilling
\begin{equation}
\|h_\varepsilon\|_{C^{\tilde{\mu}_{3,****},\frac{\tilde{\mu}_{3,****}}{2}}(\bar{\Omega}\times[t_0-1,t_0+1];\mathbb{R}^3)}\leq \tilde{C}_{4,****},
\label{222zjscz2.52ssdd97x9dddd630111kddffkhhffrssddroojj}
\end{equation}
so that, the regularity estimates from Schauder theory for the Stokes evolution equation
(\cite{SolonnikovSolonnikov1216}) ensure that \dref{222zjscz2.52ssdd9ddff7x9630111kddffkhhffrssddroojj} possesses a classical solution $\bar{v}\in {C^{2+\tilde{\mu}_{3,****},1+\frac{\tilde{\mu}_{3,****}}{2}}(\bar{\Omega}\times[t_0-1,t_0+1])}$ satisfying
\begin{equation}
\|\bar{v}\|_{C^{2+\tilde{\mu}_{3,****},1+\frac{\tilde{\mu}_{3,****}}{2}}(\bar{\Omega}\times[t_0-1,t_0+1];\mathbb{R}^3)}\leq \tilde{C}_{5,****}
\label{222zjscz2.52ssdd97x9dddd63011ssdd1kddffkhhffrssddroojj}
\end{equation}
with some $\tilde{C}_{5,****} > 0$ which is independent of $t_0$. This combined with the
uniqueness property of \dref{222zjscz2.52ssdd9ddff7x9630111kddffkhhffrssddroojj}, one can prove
\begin{equation}
\|u_\varepsilon\|_{C^{2+\tilde{\mu}_{3,****},1+\frac{\tilde{\mu}_{3,****}}{2}}(\bar{\Omega}\times[t,t+1];\mathbb{R}^3)}\leq \tilde{C}_{6,****}.
\label{222zjscz2.52ssdd9ddff7x9dddd63011ssdd1kddffkhhffrssddroojj}
\end{equation}
Again relying on Lemma \ref{lemma45630hhuujjuuyy}, this in turn warrants that for
some $\tilde{\mu}_{4,****}\in(0,1), \tilde{T}_{4,****}> 0$ and $\tilde{C}_{7,****}> 0$ such that for all $t>\tilde{T}_{4,****}$
\begin{equation}
\|\nabla\cdot(n_{\varepsilon}F_{\varepsilon}(n_{\varepsilon})S_\varepsilon(x, n_{\varepsilon}, c_{\varepsilon})\nabla c_{\varepsilon})\|_{C^{\tilde{\mu}_{4,****},\frac{\tilde{\mu}_{4,****}}{2}}(\bar{\Omega}\times[t,t+1])}+\|u_{\varepsilon}\cdot\nabla n_{\varepsilon}\|_{C^{\tilde{\mu}_{4,****},\frac{\tilde{\mu}_{4,****}}{2}}(\bar{\Omega}\times[t,t+1])}\leq \tilde{C}_{7,****},
\label{222zjscz2.52ssdd9ddff7x9ddkklldd63011ssdd1kddffkhhffrssddroojj}
\end{equation}
which along with the Schauder theory  establishes 
\begin{equation}
\|n_\varepsilon\|_{C^{2+\tilde{\mu}_{4,****},1+\frac{\tilde{\mu}_{4,****}}{2}}(\bar{\Omega}\times[t,t+1])}\leq \tilde{C}_{8,****}.
\label{222zjscz2.52ssdd9ddff7x9dsddddd63011ssdd1kddffkhhffrssddroojj}
\end{equation}
Finally, choose $$T_{*,13}=\max\{\tilde{T}_{1,****},\tilde{T}_{2,****},\tilde{T}_{3,****},\tilde{T}_{4,****}\}$$
 and $$\mu_{*,13}=\min\{\tilde{\mu}_{1,****},\tilde{\mu}_{2,****},\tilde{\mu}_{3,****},\tilde{\mu}_{4,****}\},$$ then \dref{222zjscz2.52ssdd97x9630111kddffkhhffrroojj}, \dref{222zjscz2.52ssdd9ddff7x9dddd63011ssdd1kddffkhhffrssddroojj}, \dref{222zjscz2.52ssdd9ddff7x9dsddddd63011ssdd1kddffkhhffrssddroojj} imply \dref{222zjscz2.5297x9630111kkhhffrroojj}--\dref{222zjscz2.5297x96dfgg30111kkhhffrroojj}.
\end{proof}

%
%
%
%
%
%
%
%
%
%
%
%
%
%
%
%
%
%
%
%

Having found uniform H\"{o}lder bounds on $n_\varepsilon, c_\varepsilon$ and  $u_\varepsilon$ for $\varepsilon> 0$ in the previous three lemmas (see Lemmas \ref{lemma45630hhuujjuuyy} and \ref{lemma45630hhuujjsdfffggguuyy}), also $n, c$
and $u$ share this regularity and these bounds.

\begin{lemma}\label{lemma45630223}
Assume that   $\alpha\geq1$ and $C_S<2\sqrt{C_N}$, where
 $C_N$  is the best  Poincar\'{e} constant and $C_S$ is given by  \dref{x1.73142vghf48gg}. There exist $\theta\in (0,1)$ as well as   $T_0 > 0$ and $(\varepsilon_j)_{j\in \mathbb{N}}\subset (0, 1)$ of the sequence
from Lemma \ref{sedddlemmaddffffdfffgg4sssdddd5630} such that for any $t>T_0 $
\begin{equation}
 \left\{\begin{array}{ll}
 n\in C^{2+\theta,1+\frac{\theta}{2}}(\bar{\Omega}\times[t,t+1]),\\
  c\in  C^{2+\theta,1+\frac{\theta}{2}}(\bar{\Omega}\times[t,t+1]),\\
  u\in  C^{2+\theta,1+\frac{\theta}{2}}(\bar{\Omega}\times[t,t+1];\mathbb{R}^3),\\
   \end{array}\right.\label{1.ffhhh1hhhjjkdffggdfghyuisda}
\end{equation}
 that $\varepsilon_j\searrow 0$ as $j\rightarrow\infty$  and
\begin{equation}
 \left\{\begin{array}{ll}
 n_\varepsilon\rightarrow n~~\in C^{1+\theta,\frac{\theta}{2}}(\bar{\Omega}\times[t,t+1]),\\
  c_\varepsilon\rightarrow c~~\in C^{1+\theta,\frac{\theta}{2}}(\bar{\Omega}\times[t,t+1]),\\
 u_\varepsilon\rightarrow u~~\in C^{1+\theta,\frac{\theta}{2}}(\bar{\Omega}\times[t,t+1];\mathbb{R}^3)\\
   \end{array}\right.
   and\label{1.ffgghhhhh1dffggdfghyuisda}
\end{equation}
as $\varepsilon=\varepsilon_j\searrow 0$.
Moreover, there is $C > 0$ such that
\begin{equation}
\|c(\cdot,t)\|_{C^{2+\theta,1+\frac{\theta}{2}}(\bar{\Omega}\times[t,t+1])}\leq C ~~\mbox{for all}~~ t>T_0
\label{222zjscz2.5297x9630111kkhhfsddddfrroojj}
\end{equation}
as well as
\begin{equation}
\|n(\cdot,t)\|_{C^{2+\theta,1+\frac{\theta}{2}}(\bar{\Omega}\times[t,t+1])}+\|u(\cdot,t)\|_{C^{2+\theta,1+\frac{\theta}{2}}(\bar{\Omega}\times[t,t+1]);\mathbb{R}^3)}\leq C ~~\mbox{for all}~~ t>T_0.
\label{sss222zjscz2.5297x9sdddd630111kkhhfsddddfrroojj}
\end{equation}

\end{lemma}
\begin{proof}
According to  Lemma \ref{lemma45630hhuujjsdfffggguuyy} and Lemma \ref{sedddlemmaddffffdfffgg4sssdddd5630}, an application of the Aubin-Lions lemma (see \cite{Simon}) provides a sequence
$(\varepsilon_j)_{j\in \mathbb{N}}\subset (0, 1)$ such that $\varepsilon_j\searrow 0$ as $j\rightarrow\infty$ and such that 
%
\dref{1.ffhhh1hhhjjkdffggdfghyuisda}--\dref{sss222zjscz2.5297x9sdddd630111kkhhfsddddfrroojj} hold.  And thereby completes the proof.
\end{proof}

\begin{lemma}\label{sssslemma45ssddddff630hhuujjsdfffggguuyy}
Let $\alpha\geq1$.
Then one can find $\theta\in(0, 1)$ and $T>0$ such that 
\begin{equation}
\|c(\cdot,t)\|_{C^{2+\theta,1+\frac{\theta}{2}}(\bar{\Omega}\times[T,\infty))} \leq C
\label{222zjscffgggz2.5fff297x9630111kkhhffrroojj}
\end{equation}
%
%
as well as
\begin{equation}
\|u(\cdot,t)\|_{C^{2+\theta,1+\frac{\theta}{2}}(\bar{\Omega}\times[T,\infty);\mathbb{R}^3)} \leq C
\label{222zjscffgggz2.5sdddff297x9630111kkhhffrroojj}
\end{equation}
and
\begin{equation}
\|n(\cdot,t)\|_{C^{2+\theta,1+\frac{\theta}{2}}(\bar{\Omega}\times[T,\infty))}   \leq C.
\label{222zjscz2ggggg.5297x9630111kkhhiioo}
\end{equation}
\end{lemma}
\begin{proof}
Let $g := -\xi c+n\xi -\xi u\cdot\nabla c+c\xi'$, where $\xi:=\xi_{T_0}$ is given by {Definition} \ref{aaalemmaddffffdssfffgg4sssdddd5630}
and $T_0$ is  same as   the previous lemmas.
Then we  consider the following problem
\begin{equation}
 \left\{\begin{array}{ll}
   \tilde{c}_t-\Delta \tilde{c}=g~~~~x\in\Omega, t>T_0,\\
 \disp{\tilde{c}(T_0)=0},\quad
x\in \Omega,\\
\disp{\frac{\partial\tilde{c}}{\partial \nu}=0},\quad
x\in \partial\Omega.\\
 \end{array}\right.
\label{222zjscz2.52ssdd9dssddddff7x9630111kddffkhhffssdddrssddroojj}
\end{equation}
According to Lemma \ref{lemma45630223} and {Definition} \ref{aaalemmaddffffdssfffgg4sssdddd5630}, we can find 
$\theta\in(0,1)$ 
 such that
$$g~~\mbox{is bounded in}~~ C^{\theta} (\bar{\Omega}\times(T_0, \infty)).$$
Therefore, 
 the regularity estimates from Schauder theory for the  parabolic equation
(see e.g. III.5.1 of \cite{Ladyzenskajaggk7101}) ensure that  problem \dref{222zjscz2.52ssdd9dssddddff7x9630111kddffkhhffssdddrssddroojj} admits a unique solution $\tilde{c}\in C^{2+\theta,1+\frac{\theta}{2}}(\bar{\Omega}\times[T_0+1,\infty)).$
This combined with the property of $\xi$ implies that
\begin{equation}c\in C^{2+\theta,1+\frac{\theta}{2}}(\bar{\Omega}\times[T_0+1,\infty)).
\label{222zjscz2.52ssdd9dssddddff7x9630111kddffkhhffssddddffdrssddroojj}
\end{equation} 
Finally,
employing almost exactly the same arguments as in the proof of Lemma  \ref{lemma45630hhuujjsdfffggguuyy} (the minor necessary changes are left as an easy exercise to the reader), and taking advantage of \dref{sss222zjscz2.5297x9sdddd630111kkhhfsddddfrroojj}, we conclude that
\begin{equation}u\in C^{2+\theta,1+\frac{\theta}{2}}(\bar{\Omega}\times[T_0+1,\infty);\mathbb{R}^3)
\label{222zjscz2.52ssdd9dssddddff7x9630111kddddffffkhhffssdddrssddsdfgghhroojj}
\end{equation}
and
 \begin{equation}n\in C^{2+\theta,1+\frac{\theta}{2}}(\bar{\Omega}\times[T_0+1,\infty)),
 \label{222zjscz2.52ssdd9dssddddff7x9630111kddddffffkhhffssdddrssddssddffdfgghhroojj}
 \end{equation}
whence combining the result of \dref{222zjscz2.52ssdd9dssddddff7x9630111kddffkhhffssddddffdrssddroojj}  completes the proof.
\end{proof}

%
%

Our main result on eventual smoothness thereby becomes immediate.
%

\begin{lemma}\label{lemma4dd5630hhuujjuuyy}
Under the assumptions of Theorem \ref{thaaaeorem3},  the solution $(n,c,u)$ of \dref{1.1fghyuisda} constructed in Lemma \ref{sedddlemmaddffffdfffgg4sssdddd5630} satisfies
\begin{equation}n(\cdot,t)\rightarrow \bar{n}_0~~\mbox{as well as } ~~~c(\cdot,t)\rightarrow \bar{n}_0~~\mbox{and}~~~u(\cdot,t)\rightarrow0
~~\mbox{in}~~~L^\infty(\Omega),
\label{233ddxcvbbggdddddddfghhdfssdffffgcz2vv.5ghju4ss8cfg9ddsddddffff24ssdddghddfgggyddfggusdffji}
\end{equation}
where $\bar{n}_0=\frac{1}{|\Omega|}\int_{\Omega}n_0$.
\end{lemma}
\begin{proof}
Firsly, due to Lemmas \ref{sedddlemmaddffffdfffgg4sssdddd5630}, we derive from Lemma \ref{sssslemma45ssddddff630hhuujjsdfffggguuyy} that
\begin{equation}n(t)\rightarrow \bar{n}_0~~~\mbox{as well as}~~
c(t)\rightarrow \bar{n}_0~~\mbox{and}~~u(t)\rightarrow0~~~\mbox{in}~~~L^2(\Omega)~~~\mbox{as}~~~t\rightarrow\infty,
\label{aahhxxcdfvhhhvssssssssdsssjjdfffddffssllllddcssdz2.ssdd5}
\end{equation}
where $\bar{n}_0$ is given by \dref{1111hhxxcdfvhhhvsddfffgssjjdfffsfffsddcsssz2.5}.
Next, due to Lemma \ref{sssslemma45ssddddff630hhuujjsdfffggguuyy}, one can obtain there exist positive  constants $M_1$ and $T$ such that for all $t>T$
\begin{equation}\|n (\cdot,t)\|_{C^{2+\theta}(\bar{\Omega})} +\|c (\cdot,t)\|_{C^{2+\theta}(\bar{\Omega})}+\|u (\cdot,t)\|_{C^{2+\theta}(\bar{\Omega})}  \leq M_1.
\label{aahhxxcdfvhhhvssssssssdsssjjdfffddffssddcssdz2.ssdd5}
\end{equation}
Therefore, a straightforward reasoning based on interpolation and the compactness of the first among the continuous embeddings $C^{2+\theta}(\bar{\Omega}) \hookrightarrow\hookrightarrow L^\infty (\Omega)\hookrightarrow L^2 (\Omega)$ shows that \dref{aahhxxcdfvhhhvssssssssdsssjjdfffddffssllllddcssdz2.ssdd5}
 and \dref{aahhxxcdfvhhhvssssssssdsssjjdfffddffssddcssdz2.ssdd5}  entail \dref{233ddxcvbbggdddddddfghhdfssdffffgcz2vv.5ghju4ss8cfg9ddsddddffff24ssdddghddfgggyddfggusdffji}. In fact,
for any  $\eta > 0$, we may employ an Ehrling-type lemma to pick $M_{2} > 0$ fulfilling
 \begin{equation}\|n(\cdot,t)-\bar{n}_0\|_{L^{\infty}(\Omega)}\leq\frac{\eta}{2 M_{1} }\|n(\cdot,t)-\bar{n}_0\|_{C^{2+\theta}(\bar{\Omega})}+ M_{2}\|n(\cdot,t)-\bar{n}_0\|_{L^{2}(\Omega)}
\label{233ddxcvbbggdddddddfghhdfgcz2vv.5ghju4ss8cfg9ddsddddffff24ssdddghddfgggyddfggusdffji}
\end{equation}
as well as
 \begin{equation}\|c(\cdot,t)-\bar{n}_0\|_{L^{\infty}(\Omega)}\leq\frac{\eta}{2 M_{1} }\|c(\cdot,t)-\bar{n}_0\|_{C^{2+\theta}(\bar{\Omega})}+ M_{2}\|c(\cdot,t)-\bar{n}_0\|_{L^{2}(\Omega)}
\label{ghhddxcvbbggdddddddfghhdfgcz2vv.5ghju4ss8cfg9ddsddddffff24ssdddghddfgggyddfggusdffji}
\end{equation}
and
 \begin{equation}\|u(\cdot,t)\|_{L^{\infty}(\Omega)}\leq\frac{\eta}{2 M_{1} }\|u(\cdot,t)\|_{C^{2+\theta}(\bar{\Omega})}+ M_{2}\|u(\cdot,t)\|_{L^{2}(\Omega)}.
\label{dllkoooopdxcvbbggdddddddfghhdfgcz2vv.5ghju4ss8cfg9ddsddddffff24ssdddghddfgggyddfggusdffji}
\end{equation}
Now due to \dref{aahhxxcdfvhhhvssssssssdsssjjdfffddffssllllddcssdz2.ssdd5}, we may choose $t_0 > \max\{1,T\}$ large enough
such that for all $t>t_0$,
 \begin{equation}\|n(\cdot,t)-\bar{n}_0\|_{L^2(\Omega)}+\|c(\cdot,t)-\bar{n}_0\|_{L^2(\Omega)}+\|u(\cdot,t)\|_{L^2(\Omega)}<\frac{\eta}{2 M_{2}},
\label{234ddxcvbbggdddddddfghhdfgcz2vv.5ghju4ss8cfg9ddsddddffff24ghddfgggyddfggusdffji}
\end{equation}
which in conjunction with \dref{233ddxcvbbggdddddddfghhdfgcz2vv.5ghju4ss8cfg9ddsddddffff24ssdddghddfgggyddfggusdffji}--\dref{dllkoooopdxcvbbggdddddddfghhdfgcz2vv.5ghju4ss8cfg9ddsddddffff24ssdddghddfgggyddfggusdffji} entails that
$$
\begin{array}{rl}\|n(\cdot,t)-\bar{n}_0\|_{L^{\infty}(\Omega)}\leq&\disp{\frac{\eta}{2 M_{1} }\|n(\cdot,t)-\bar{n}_0\|_{C^{2+\theta}(\bar{\Omega})}+ M_{2}\|n(\cdot,t)-\bar{n}_0\|_{L^{2}(\Omega)}}\\
<&\disp{\frac{\eta}{2 M_{1} } M_{1}+ M_{2}\frac{\eta}{2 M_{2} }}\\
=&\eta~~~\mbox{for all}~~ t > t _0,\\
\end{array}
$$
$$
\begin{array}{rl}\|c(\cdot,t)-\bar{n}_0\|_{L^{\infty}(\Omega)}\leq&\disp{\frac{\eta}{2 M_{1} }\|c(\cdot,t)-\bar{n}_0\|_{C^{2+\theta}(\bar{\Omega})}+ M_{2}\|c(\cdot,t)-\bar{n}_0\|_{L^{2}(\Omega)}}\\
<&\disp{\frac{\eta}{2 M_{1} } M_{1}+ M_{2}\frac{\eta}{2 M_{2} }}\\
=&\eta~~~\mbox{for all}~~ t > t _0\\
\end{array}
$$
as well as
and
$$
\begin{array}{rl}\|u(\cdot,t)\|_{L^{\infty}(\Omega)}\leq&\disp{\frac{\eta}{2 M_{1} }\|u(\cdot,t)\|_{C^{2+\theta}(\bar{\Omega})}+ M_{2}\|u(\cdot,t)\|_{L^{2}(\Omega)}}\\
<&\disp{\frac{\eta}{2 M_{1} } M_{1}+ M_{2}\frac{\eta}{2 M_{2} }}\\
=&\eta~~~\mbox{for all}~~ t > t _0,\\
\end{array}
$$
which combined with the fact that $\eta> 0$ is arbitrary yields to \dref{233ddxcvbbggdddddddfghhdfssdffffgcz2vv.5ghju4ss8cfg9ddsddddffff24ssdddghddfgggyddfggusdffji}, 
and thereby completes the proof.

\end{proof}
%
%
%
%
%
%
%
%
%
%
%
%
%
%
%
%
%
%
%
%

In order to prove Theorem \ref{thaaaeorem3}, we now only have to collect the results prepared during this section:

{\bf Proof of Theorem  \ref{thaaaeorem3}.} 


\begin{proof}
The statement is evidently implied by  Lemmas \ref{lemma45630223}--\ref{lemma4dd5630hhuujjuuyy}.
\end{proof}

{\bf Acknowledgement}:
This work is partially supported by Shandong Provincial
Science Foundation for Outstanding Youth (No. ZR2018JL005).

\end{document}